\documentclass[11pt]{amsart}
\usepackage{geometry}                
\usepackage{graphicx}
\usepackage{amssymb}
\usepackage{epstopdf}
\newtheorem{theorem}{Theorem}[section]
\newtheorem{lemma}{Lemma}[section]
\newtheorem{corollary}{Corollary}[section]
\newtheorem{proposition}{Proposition}[section]
\newtheorem{conjecture}{Conjecture}[section]
\newtheorem{definition}{Definition}[section]
\newtheorem{example}{Example}[section]
\newtheorem{remark}{Remark}[section]

\numberwithin{equation}{section}

\def\Z{\mathbb Z}

\def\R{\mathbb R}
\def\P{\mathbb P}
\def\C{\mathbb C}

\def\D{\Delta}
\def\Om{\Omega}

\def\om{\omega}
\def\d{\partial}
\def\s{\sigma}
\def\e{\epsilon}
\def\a{\alpha}
\def\b{\beta}
\def\g{\gamma}

\def\sI{{\mathsf I}}
\def\sM{{\mathsf M}}
\def\sR{{\mathsf R}}

\def\bfc{{\mathbf c}}
\def\bfc\Theta{\mathcal c\Theta}
\def\cP{\mathcal P}
\def\bfOm{{\mathbf \Omega}}
\def\paOm#1#2{{\mathbf \Omega}_{#1,\,|\sim|' #2}}

\title[Spaces of polynomials as Grassmanians for immersions \& embeddings]{Spaces of polynomials with constrained divisors as Grassmanians for immersions \& embeddings}
\author{Gabriel Katz}
\date{01/07/2022}                                           
\address{5 Bridle Path Circle, Framingham, MA 01701, U.S.A.}
\email {gabkatz@gmail.com}

\begin{document}
\maketitle 

\begin{abstract} Let $Y$ be a smooth compact $n$-manifold. We study smooth embeddings and immersions $\b: M \to \mathbb R \times Y$ of compact $n$-manifolds $M$ such that $\beta(M)$ avoids some a priory chosen closed poset $\Theta$ of {\sf tangent patterns} to the fibers of the obvious projection $\pi: \mathbb R \times Y \to Y$. Then, for a fixed $Y$, we introduce an equivalence relation between such $\beta$'s; it is a crossover between pseudo-isotopies and bordisms. We call this relation {\sf quasitopy}. In the study of quasitopies, the spaces $\mathcal P_d^{\mathbf c\Theta}$ of real univariate polynomials of degree $d$ with real divisors, whose combinatorial patterns avoid a given closed poset $\Theta$, play the classical role of Grassmanians.  We compute the quasitopy classes $\mathcal{QT}_d^{\mathsf{emb}}(Y, \mathbf c\Theta)$ of $\Theta$-constrained embeddings $\beta$ in terms of homotopy/homology theory of spaces $Y$ and $\mathcal P_d^{\mathbf c\Theta}$. We prove also that the quasitopies of emeddings stabilize, as $d \to \infty$.
\end{abstract}

\section{Introduction}
This paper is the first in a series, inspired by the pioneering works of Arnold \cite{Ar}, \cite{Ar1}, \cite{Ar2}, and Vassiliev \cite{V}. It is based heavily on computations from \cite{KSW1}, \cite{KSW2}. In this introduction, we will describe our results informally, in a manner that clarifies their nature, but does not involve their most general forms, which carry the burden of combinatorial decorations. 

We study immersions and embeddings $\b: M \to \R \times Y$ of compact smooth $n$-manifolds $M$ into products $\R \times Y$, where $Y$ is a given compact smooth $n$-manifold. We denote by $\mathcal L$ the oriented $1$-dimensional foliation on $\R \times Y$, formed by the fibers of the obvious projection $\pi: \R \times Y \to Y$. We are interested in the ways $\b(M)$ may be tangent to the fibers of $\mathcal L$. More precisely, we impose restrictions $\Theta$ on the combinatorial patterns $\om$ of such tangencies and organize $\{\b: M \to \R \times Y\}$ that do not violate $\Theta$ into a set $\mathcal{QT}(Y, \mathbf c\Theta)$ of equivalence classes; we call equivalent $\b$'s {\sf quasitopic}. The quasitopies is a hibrid between classical notions of pseudo-isotopies and bordisms (see Fig. 1). In the paper, we deal with various brands of quasitopies, richly decorated by combinatorial data.  In this introduction, to simplify the description of our results, we omit the decorations. In particular, we may insist that the cardinality of the fiber of $\pi \circ \b$ is less than or equal to a given natural number $d$. In such a case, we are dealing with the quasitopies  $\mathcal{QT}_d(Y, \mathbf c\Theta)$.

The main observation of the paper is that spaces $\mathcal P_d^{\mathbf c\Theta}$ of real degree $d$ monic polynomials in one variable, whose real divisors do not belong to the forbidden poset $\Theta$, play a role of Grassmanians for immersions/embeddings $\b$ as above. This is reminiscent to the role played by the Stiefel manifolds of $k$-frames in $\R^m$ in the Smale theory of immersions \cite{S}. At least for embeddings,  we are able to compute $\mathcal{QT}_d^{\mathsf{emb}}(Y, \mathbf c\Theta)$ in terms of the homotopy theory of $\mathcal P_d^{\mathbf c\Theta}$ and $Y$. Our computation uses the results from \cite{KSW1}, \cite{KSW2} about homotopy and homology types of $\mathcal P_d^{\mathbf c\Theta}$ for many specific $\Theta$'s. This information allows us to describe cohomological invariants of $\mathcal{QT}_d^{\mathsf{emb}}(Y, \mathbf c\Theta)$ and $\mathcal{QT}_d^{\mathsf{imm}}(Y, \mathbf c\Theta)$ (characteristic classes of quasitopies) in purely combinatorial terms, related to $\Theta$. Here the superscript ``$\mathsf{imm}$" stands for immersions.  In many special cases, we see subtle connections between $\mathcal{QT}_d^{\mathsf{emb}}(Y, \mathbf c\Theta)$ and homotopy groups of spheres. Implicitly, this connection goes back to \cite{Ar}, \cite{Ar1} and \cite{V}. 

Let $D^n$ denote the standard $n$-ball. We notice that the sets $\mathcal{QT}_d^{\mathsf{emb}}(D^n, \mathbf c\Theta)$ and $\mathcal{QT}_d^{\mathsf{imm}}(D^n, \mathbf c\Theta)$ have a {\sf group structure} (abelian for $n>1$). We denote these groups by $\mathsf{G}_d^{\mathsf{emb}}(n; \mathbf c\Theta)$ and $\mathsf{G}_d^{\mathsf{imm}}(n; \mathbf c\Theta)$ and prove that $\mathsf{G}_d^{\mathsf{imm}}(n; \mathbf c\Theta)$ is a split extension of $\mathsf{G}_d^{\mathsf{emb}}(n; \mathbf c\Theta)$. Then we are able to validate a group isomorphism $\mathsf{G}_d^{\mathsf{emb}}(n; \mathbf c\Theta) \approx \pi_n(\mathcal P_d^{\mathbf c\Theta}, pt)$.

With the help of the boundary connected sum operation, the groups $\mathsf{G}_d^{\mathsf{emb}}(n; \mathbf c\Theta)$ and $\mathsf{G}_d^{\mathsf{imm}}(n; \mathbf c\Theta)$ act on the sets $\mathcal{QT}_d^{\mathsf{emb}}(Y, \mathbf c\Theta)$ and $\mathcal{QT}_d^{\mathsf{imm}}(Y, \mathbf c\Theta)$, respectively. 

For two pairs $A \supset B$, $C \supset D$ of topological spaces, let $[(A, B), (C, D)]$ denote the homotopy classes of continuous maps $f: (A, B) \to (C, D)$, where $f(B) \subset D$. 

We prove that the classifying map $\Phi^{\mathsf{emb}}: \mathcal{QT}_d^{\mathsf{emb}}(Y, \mathbf c\Theta) \to [(Y, \d Y), (\mathcal P_d^{\mathbf c\Theta}, pt)]$ is a {\sf bijection}, and the classifying map $\Phi^{\mathsf{imm}}: \mathcal{QT}_d^{\mathsf{imm}}(Y, \mathbf c\Theta) \to [(Y, \d Y), (\mathcal P_d^{\mathbf c\Theta}, pt)]$ is a {\sf surjection}. Moreover, these maps are equivariant with respect to the $\mathsf{G}_d^{\mathsf{emb}}(n; \mathbf c\Theta)$- or $\mathsf{G}_d^{\mathsf{imm}}(n; \mathbf c\Theta)$-actions on the quasitopies (the sources of the classifying maps $\Phi^{\mathsf{emb/imm}}$) and the corresponding action of the homotopy groups $\pi_n(\mathcal P_d^{\mathbf c\Theta}, pt)$ on the sets $[(Y, \d Y), (\mathcal P_d^{\mathbf c\Theta}, pt)]$ (the targets of the maps $\Phi^{\mathsf{emb/imm}}$). 

We present also a variety of results about the stabilization of $\mathcal{QT}_d^{\mathsf{emb}}(Y, \mathbf c\Theta)$ as $d \to \infty$.\smallskip

In \cite{K6}, we apply and extend the result of this article to a study of the traversing smooth flows on manifolds $X$ with boundaries, where the poset $\Theta$ controls  the combinatorial patterns of tangency of the flow trajectories to the boundary $\d X$.


\section{Spaces of real polynomials with constrained real divisors}
In this section, for our readers' convenience, we state a number of results from \cite{KSW1} and \cite{KSW2} about the topology of spaces of real monic univariate polynomials with constrained real divisors. These results will be crucial for the applications to follow.\smallskip

Let $\cP_d$ denote the space of real monic univariate polynomials of degree $d$. Given a polynomial $P(u) = u^d + a_{d-1}u^{d-1} + \cdots + a_0$ with real coefficients,  we consider its {\sf real divisor} $D_\R(P)$. 

Let $\om_i$ denotes the multiplicity of the $i$-th real root of $P$, the roots being ordered by their magnitude.  The ordered $\ell$-tuple 
$\om = (\omega_1 ,\ldots, \omega_\ell)$ is called 
the {\sf real root multiplicity pattern} of $P(x)$, or the {\sf multiplicity pattern} for short. 

 Such sequences $\om$ form a {\sf universal poset} $(\mathbf\Om, \succ)$. The partial order ``$\succ$" in $\mathbf\Om$ is defined in terms of two types of elementary operations: {\sf merges} $\{\mathsf M_i\}_i$ and {\sf inserts} $\{\mathsf I_i\}_i$. 
 
The operation $\mathsf M_i$ merges a pair of adjacent entries $\om_i, \om_{i+1}$ of $\om = (\om_1, \dots, \om_i, \om_{i+1}, \dots, \om_q)$ into a single component $\tilde\om_i = \om_i + \om_{i+1}$, thus forming a new shorter sequence $\mathsf M_i(\om) = (\om_1, \dots, \tilde\om_i, \dots, \om_q)$. The operation $\mathsf I_i$ either insert $2$ in-between $\om_i$ and $\om_{i+1}$, thus forming a new longer sequence $\mathsf I_i(\om) = (\dots, \om_i, 2, \om_{i+1}, \dots)$, or, in the case of $\mathsf I_0$, appends $2$ before the sequence $\om$, or,  in the case $\mathsf I_q$,  appends $2$ after the sequence $\om$. 

 So the {\sf merge} operation 
$\mathsf M_j: \mathbf\Om\to \mathbf\Om$ \;  
sends $\om = (\om_1,\ldots , \om_\ell)$ to the composition $$\mathsf M_j(\om) = (M_j(\om)_1,\ldots, M_j(\om)_{\ell-1}),$$ 
where, for any $j \geq \ell$, one has $\mathsf M_j(\om) = \om$,  and for $1 \leq j < \ell$, one has 
\begin{eqnarray}\label{eq2.M} 
\left\{
\begin{array}{lll}
\mathsf M_j(\omega)_i & = & \omega_i \; \textrm{ if }\, i < j,\\ 
\mathsf M_j(\omega)_j  & = & \omega_j + \omega_{j+1},\\ 
\mathsf M_j(\omega)_i  & = & \omega_{i + 1} \; \textrm{ if }\, i+1 < j \leq  \ell-1.
\end{array}\right.
\end{eqnarray}

Similarly, we introduce the {\sf insert} operation 
$\mathsf I_j: \mathbf\Om \to\mathbf \Om$ that 
sends $\om = (\om_1,\ldots , \om_\ell)$ to the composition $\mathsf I_j(\om) = (I_j(\om)_1,
\ldots, I_j(\om)_{\ell+1}),$ where for any $j > \ell+1$,  one has $\mathsf I_j(\om) = \om$, and  for $1 \leq j \leq \ell+1$, one has
\begin{eqnarray}\label{eq2.I}  
\left\{
\begin{array}{lll}
\mathsf I_j(\omega)_i & = & \omega_i\; \textrm{ if }\, i < j, \\ 
\mathsf I_j(\omega)_j & = & 2,\\ 
\mathsf I_j(\omega)_i & = & \omega_{i - 1} \; \textrm{ if }\, j \leq i \leq \ell+1. 
\end{array}\right.
\end{eqnarray}

\begin{definition}\label{def.order}
For $\om, \om' \in \mathbf\Om$, we define the order relation $\om \succ \om'$ if one can produce $\om'$ from $\om$ by applying a sequence of the elementary operations in (\ref{eq2.M}) and (\ref{eq2.I}).  \hfill $\diamondsuit$
\end{definition}

For a sequence $\omega = (\omega_1, \omega_2, \, \dots \, ,  \omega_q) \in \mathbf\Om$, we introduce the {\sf norm} and the {\sf reduced norm} of $\omega$ by the formulas: 
\begin{eqnarray}\label{eq2.2}
|\omega| =_{\mathsf{def}} \sum_i\; \omega_i \quad \text{and} \quad |\omega|'  =_{\mathsf{def}} \sum_i\; (\omega_i -1).
\end{eqnarray}
Note that $q$, the cardinality of the {\sf support} of $\omega$, is equal to $|\omega| - |\omega|'$.
\smallskip

Let $\mathring{\sR}^\omega_d$ be of the set of all polynomials with root multiplicity pattern $\omega$, and let $\sR_d^\omega$ be its closure in $\cP_d$.

 For a given collection $\Theta$ of multiplicity patterns $\{\om\}$ which share the pairity of their norms $\{|\om|\}$ and is closed under the merge and insert operations, we consider the union $\cP_d^\Theta$ of the subspaces $\{\sR_d^{\omega}\}_\om$, taken over all $\omega \in \Theta$ such that $|\om| \leq d$ and $|\om| \equiv  d \mod 2$. We denote by $\cP_d^{\mathbf c\Theta}$ its {\sf complement} $\cP_d \setminus \cP^\Theta_d$. \smallskip

Since $\cP^\Theta_d$ is contractible, it makes more sense to consider its one-point compactification $\bar{\cP}_d^\Theta$ (which is the union of the one point compactifications $\bar{\sR}_d^\omega$ for $\omega \in \Theta$ with all the points at infinity being identified). 
If the set $\bar{\cP}_d^{\Theta}$ is closed in $\bar{\cP}_d$, by the Alexander duality on the sphere  $\bar{\cP}_d \cong S^{d}$, we get
$$H^j(\cP_d^{\mathbf c\Theta}; \Z) \approx H_{d - j -1}(\bar{\cP}_d^{\Theta}; \Z).$$
This implies that  the spaces 
$\cP_d^{\mathbf c\Theta}$ and $\bar{\cP}_d^{\Theta}$ 
carry the same (co)homological information. Let us describe it in pure combinatorial terms. \smallskip

For a  subposet $\Theta \subset \bfOm$ and natural numbers $d, k$, we introduce the following notations:
\begin{eqnarray}\label{eq.Theta} 
\Theta_{[d]} =_{\mathsf{def}} &\{\om \in \Theta : \; |\om| = d \},\quad
\Theta_{\langle d]} =_{\mathsf{def}} \{\om \in \Theta : \; |\om| \leq d\}, \\ 
\Theta_{|\sim|' =k} =_{\mathsf{def}} & \{\om \in \Theta : \; |\om|' = k \}, \quad \Theta_{|\sim|' \geq k} =_{\mathsf{def}} \{\om \in \Theta : \; |\om|' \geq k \}.  \nonumber
 \end{eqnarray}

Assuming that $\Theta\subset \bfOm$ is a {\sf closed} sub-poset, let 
$$\mathbf c\Theta =_{\mathsf{def}} \mathbf\Om \setminus \Theta \text{\; and \;} \mathbf c \Theta_{\langle d]} =_{\mathsf{def}} \bfOm_{\langle d]} \setminus \Theta_{\langle d]}.$$

We denote by $\Z[\bfOm_{\langle d]}]$ the free $\Z$-module, generated by the elements of $\bfOm_{\langle d]}$. Using the merge operators $\sM_k$ and the insert operators $\sI_k$ on $\bfOm_{\langle d]}$,   
we define  two homomorphisms of $\Z[\bfOm_{\langle d]}]$: 
$$\d_{\sM} (\omega) =_{\mathsf{def}} - \sum_{k=1}^{s_\omega-1} (-1)^k \sM_k(\omega) \, \text{  and   }\, \partial_{\sI} (\omega) =_{\mathsf{def}}  \sum_{k=0}^{s_\omega} (-1)^k \sI_k(\omega),$$
where $s_\omega =_{\mathsf{def}} |\omega| - |\omega|'$ is the cardinality of the support of $\omega$. The homomorphisms $\d_{\sM}$ and $\d_{\sI}$ are anti-commuting differentials.

Next, we introduce the homomorphism 
$$\d  =_{\mathsf{def}} \d_{\sM} +\d_{\sI} : \Z[\bfOm_{\langle d]}] \to \Z[\bfOm_{\langle d]}]$$ by the formula
\begin{eqnarray}\label{eq.d+d}
\quad \quad \d(\omega) =_{\mathsf{def}} \left\{
\begin{array}{ll}
-\sum_{k = 1}^{s_\omega -1}(-1)^k\, \sM_k(\omega) + \sum_{k = 0}^{s_\omega}(-1)^k\, \sI_k(\omega),\;\; \text{for}\; |\omega| < d,   \\
 \\
-\sum_{k = 1}^{s_\omega -1}(-1)^k\, \sM_k(\omega),\;\; \text{for}\; |\omega| = d. 
\end{array}\right.
\end{eqnarray}
\smallskip

It is easy to see that for a closed poset $\Theta_{\langle d]} \subset \bfOm_{\langle d]}$, the restrictions of the differentials $\d, \d_{\sM}$, and $\d_{\sI}$ to the free $\Z$-module $\Z[\Theta_{\langle d]}]$ are well-defined. Thus, for any closed sub-poset  $\Theta_{\langle d]} \subset \mathbf\Om_{\langle d]}$, we may consider the differential complex $\d: \Z[\Theta_{\langle d]}] \to \Z[\Theta_{\langle d]}]$, whose $(d-j)$-grading is defined by the module $\Z[\Theta_{\langle d],\, |\sim|' = j}]$. 
We denote by $\d^\ast: \Z[\Theta_{\langle d]}]^\ast \to \Z[\Theta_{\langle d]}]^\ast$ its dual, where $\Z[\Theta_{\langle d]}]^\ast := \mathsf{Hom}(\Z[\Theta_{\langle d]}], \Z)$ and the operator $\d^\ast$ is the dual of the boundary operator $\d$. \smallskip

Then we consider the quotient set $\Theta^\#_{\langle d]} := \mathbf\Om_{\langle d]}\big / \Theta_{\langle d]}$. For the closed subposet $\Theta_{\langle d]}$, the partial order in $\mathbf\Om_{\langle d]}$ induces a partial order in the quotient $\Theta^\#_{\langle d]}$.

Next, we introduce a new differential complex $(\Z[\Theta^\#_{\langle d]}], \d^\#)$ by including it in the short exact sequence of differential complexes:
\begin{eqnarray}\label{eq.quotient_complex}
0 \to (\Z[\Theta_{\langle d]}], \d) \to (\Z[\mathbf\Om_{\langle d]}], \d) \to (\Z[\Theta^\#_{\langle d]}], \d^\#) \to 0.
\end{eqnarray}

We will rely on the following result from \cite{KSW2}, which reduces the computation of the (reduced) cohomology $\tilde H^\ast(\cP_d^{\mathbf c\Theta}; \Z)$ to Algebra and Combinatorics. 

\begin{theorem}\label{thA}{\bf (\cite{KSW2})} 
  Let $\Theta \subset \bfOm_{\langle d]}$ be a closed subposet. 
  Then, for any $j \in [0, d]$, we get group isomorphisms
   \begin{eqnarray}\label{eq.3.5A} 
    \tilde H^j(\cP_d^{\mathbf c\Theta}; \Z)\;  \stackrel{\mathcal D}{\approx} \;H_{d-j}(\bar{\cP}_d, \bar{\cP}_d^{\Theta}; \Z)\; 
    \approx  H_{d- j}\big(\d^\#:  \Z[\Theta^\#] \to \Z[\Theta^\#]\big), \\
    \quad \tilde  H_j(\cP_d^{\mathbf c\Theta}; \Z)\;  \stackrel{\mathcal D}{\approx} \;H^{d-j}(\bar{\cP}_d, \bar{\cP}_d^{\Theta}; \Z)\; 
    \approx  H_{d- j}\big((\d^\#)^\ast:  (\Z[\Theta^\#])^\ast \to (\Z[\Theta^\#])^\ast\big), \nonumber
      \end{eqnarray}
      where $\mathcal D$ denote the Poincar\'{e} duality isomorphisms. \hfill $\diamondsuit$
 \end{theorem}

For $d' \geq d$, $d' \equiv d \mod 2$, consider an embedding $\e_{d, d'}: \mathcal P_d \to \mathcal P_{d'}$, defined by 
\begin{eqnarray}\label{eq.stable} 
\e_{d, d'}(P)(u) =_{\mathsf{def}} (u^2+1)^{\frac{d' -d}{2}} \cdot P(u).
\end{eqnarray}
It preserves the 
$\bfOm$-stratifications of  the two spaces by the combinatorial types $\om$ of real divisors. 

\begin{definition}\label{def.profinite}
We call a closed poset $\Theta\subseteq \mathbf\Om$  {\sf profinite} if, for all  integers $q \geq 0$, there exist only finitely many elements $\om \in \Theta$ such that $|\om|' \leq q$. 
  \hfill $\diamondsuit$
\end{definition}

Note that any finitely generated (i.e., having a finite set of maximal elements) closed $\Theta$ is profinite.\smallskip

For a profinite $\Theta$, the embedding $\e_{d, d'}$ makes it possible to talk about \emph{stabilization of the homology/cohomology} of spaces $\bar{\mathcal P}_d^\Theta$ and $\mathcal P_d^{\mathbf c\Theta}$, as $d \to \infty$. \smallskip

For a closed poset $\Theta \subset \bfOm$, consider the closed finite poset $\Theta_{\langle d]} = \bfOm_{\langle d]} \cap \Theta$. It is generated by some    
\emph{maximal} elements $\omega_\star^{(1)}, \dots , \omega_\star^{(\ell)}$, $\ell = \ell(d)$.  We introduce two useful quantities: 
\begin{eqnarray}\label{eq.eta(Theta)}
\eta_\Theta(d) & =_{\mathsf{def}} & \max_{i \,\in \,[1,\, \ell(d)]} \big\{\big(|\om_\star^{(i)}| - 2|\om_\star^{(i)}|'\big)\big\}, \\
\psi_\Theta(d) & =_{\mathsf{def}} & \frac{1}{2}\big(d + \eta_\Theta(d)\big).
\end{eqnarray}

Note that $$\psi_\Theta(d) =  \frac{1}{2}\;\min_{i \, \in \, [1,\ell(d)]}\big\{ (d- |\om_\star^{(i)}|') + (|\om_\star^{(i)}|-|\om_\star^{(i)}|')\big\} < d,$$ 
where both summands, the ``codimension" $(d- |\om_\star^{(i)}|')$ and the ``support" $(|\om_\star^{(i)}|-|\om_\star^{(i)}|')$, are positive and each one does not exceed $d$. At the same time,
$\eta_\Theta(d)$ may be negative. \smallskip

The quantity $\xi_\Theta(d+2) := d +2 - \psi_\Theta(d+2)$ is employed on many occasions. It plays a key role in describing the stabilization in the homology $H_\ast(\mathcal P_d^{\mathbf c\Theta}; \Z)$.\smallskip
 \smallskip

Now we are in position to state one of the main stabilization results from \cite{KSW2}: 

\begin{theorem}\label{th.main_stab} {\bf (short stabilization: $\mathbf{\{d \Rightarrow d+2\}}$)} Let $\Theta$ be a closed subposet of $\mathbf\Om$. Let the embedding $\e_{d, d+2}:\, \mathcal P_{d} \subset \mathcal P_{d+2}$ be as in (\ref{eq.stable}). \begin{itemize}
\item Then, for all  $j \geq  \psi_\Theta(d+2) -1$, we get a homological isomorphism 
$$(\e_{d, d+2})_\ast: H_j(\bar{\mathcal P}_{d}^\Theta; \Z) \approx H_{j+2}(\bar{\mathcal P}_{d+2}^{\Theta}; \Z),$$

\item and, for all $j \leq d+2 - \psi_\Theta(d+2)$, 
a homological isomorphism  $$(\e_{d, d+2})_\ast: H_j( \mathcal P_{d}^{\mathbf c\Theta}; \Z) \approx H_j(\mathcal P_{d+2}^{\mathbf c\Theta}; \Z). \quad \quad \diamondsuit$$
 \end{itemize}
\end{theorem}

\begin{corollary}\label{cor.stable_homology} {\bf (long stabilization: $\mathbf{\{d \Rightarrow \infty\}}$)} If  $\Theta\subset \mathbf\Omega$  is a closed profinite subposet,  
then for each $j$, the homomorphism
$$(\e_{d,d'})_\ast:\; H_j( \mathcal P_{d}^{\mathbf c\Theta}; \Z) \approx H_j(\mathcal P_{d'}^{\mathbf c\Theta}; \Z)$$
is an isomorphism for all sufficiently big $d \leq d'$, $d \equiv d' \mod 2$. 

As a result, for such profinite $\Theta$'s, we may talk about the stable homology $H_j(\mathcal P_{\infty}^{\mathbf c\Theta}; \Z)$, the direct limit  $\lim_{d \to \infty}H_j( \mathcal P_{d}^{\mathbf c\Theta}; \Z)$. \hfill  $\diamondsuit$
\end{corollary}

Let us describe few special cases of posets $\Theta$ 
from \cite{KSW2}. 
For $k \geq 1$,  $q \in [0, d]$, and $q \equiv d \mod 2$, let us consider the closed poset
\begin{eqnarray}\label{Theta-wedge} 
\bfOm_{|\sim|^{'} \geq k}^{(q)} =_{\mathsf{def}} \big\{ \om \in \bfOm_{\langle d]} \text{ such that } |\omega|' \geq k \text{ and } |\om| \geq q\big\}.
\end{eqnarray}
Note that, for $\Theta = \bfOm_{|\sim|^{'} \geq k}^{(0)} := \bfOm_{|\sim|' \geq k}$, the space $\bar{\cP}_d^\Theta$ is the entire $(d-k)$-skeleton of $\bar{\cP}_d$. 
 
 \begin{eqnarray}\label{eq.A-bouquet} 
 \text{ Let \;} A(d, k, q) =_{\mathsf{def}}\; \big|\chi\big((\Z\big[\bfOm_{|\sim|^{'} \geq k}^{(q)}\big], \d)\big)\big|,
 \end{eqnarray}
 the absolute value of the Euler number of the differential complex $\big(\Z\big[\bfOm_{|\sim|^{'} \geq k}^{(q)}\big], \d\big)$.

\begin{proposition}{\bf (\cite{KSW2})} 
  \label{prop.skeleton}
  Fix $k \in [1, d)$ and $q \ge 0$ such that $q \equiv d \mod 2$, and set $\Theta = \bfOm_{|\sim|^{'} \geq k}^{(q)}$. 
  
  Then $\cP^{\mathbf c\Theta}_d$ has the homotopy type of 
	a bouquet of $(k-1)$-dimensional spheres. The number of spheres in the bouquet
	equals $A(d, k, q)$.\hfill $\diamondsuit$
\end{proposition}

\begin{proposition}{\bf (\cite{KSW1})} 
  \label{prop.fundamental_groups}
  Let $\Theta \subseteq \paOm{\langle d]}{\geq 2}$ 
  be a closed poset. For $d' \geq d+2$ such that $d'  \equiv d \mod 2$, 
  let $\hat\Theta_{\{d'\}}$ be the 
  smallest closed poset in $\bfOm_{\langle {d'}]}$ containing
  $\Theta$. 
  
  Then for $d' \geq d+2$, we have an isomorphism 
	$\pi_1\big(\cP_{d'}^{\mathbf c \hat \Theta_{\{d'\}}}\big) \cong \pi_1\big(\cP_{d+2}^{\mathbf c\Theta_{\{d+2\}}}\big)$ of the fundamental groups. \hfill $\diamondsuit$
\end{proposition} 



\section{Immersions and embeddings with restricted tangency patterns to the product $1$-foliations}

\subsection{On immersions $\{\b: M^n \to \R \times Y^n\}$ and their self-intersections loci}
\smallskip
We adopt the following definition of a proper immersion. 
 
\begin{definition}\label{def.im} Let $M$ be a compact smooth $n$-manifold and $W$ a compact smooth $(n+1)$-manifold. We call an immersion $\b: M \to W$ {\sf proper} if:
 \begin{enumerate}
\item $\b(\d M) \subset \d W$,
 \item $\b$ is transversal to $\d W$ (thus the kernel of $D(\b|_{\d M})$ is also trivial).  \hfill $\diamondsuit$ 
\end{enumerate}
\end{definition}   

When $M$ is closed, requirements (1) and (2) are vacuous. When $Y$ is closed, then $M$ must be closed as well. \smallskip 

\begin{remark}
\emph{In this paper, for a \emph{given} compact smooth $n$-manifold $Y$, we are only interested in proper immersions/embeddings $\b$ of compact smooth $n$-manifolds $M$ into the product $\R \times Y$. 
The existence of such $\b$ puts substantial restrictions on the \emph{topological nature} of $M$. For example, the tangent $n$-bundle $\tau_M$ must be a subbundle of the $(n+1)$-bundle $\b^\ast(\tau_Y) \oplus \underline\R$. By the classical Hirsch Theorem \cite{Hi}, the existence of such an immersion $\b$ in the  homotopy class of a given map  $\tilde\b: M \to \R \times Y$ is equivalent to the property of the stable bundle $\tilde\b^\ast(\tau_{\R \times Y}) \oplus \nu_M$ to be of {\sf geometric} dimension $\leq 1$. Here $\nu_M$ is the stable normal bundle of $M$. The geometric dimension of a stable bundle $\eta$ is defined to be the minimal dimension of a bundle $\eta^\sharp$ in the stable equivalence class of $\eta$. 
}
\hfill $\diamondsuit$
\end{remark}

Definition \ref{def.q-separable} below requires to introduce of the following notion.   Let $G, H$ be two smooth hypersurfaces in a smooth manifold $X$.  We fix a point $a \in G \cap H$ and consider its neighborhood $U_a \subset X$. Let $g: U_a \to \R$, $h: U_a \to \R$ be two smooth functions such that $0$ is a regular value for both, and $G = g^{-1}(0), H = h^{-1}(0)$. 

Let $q$ be a natural number. We say that two hypersurfaces $G, H$  are $q$-{\sf jet equivalent} at $a \in G \cap H$, if there exists smooth functions $\phi, \chi: U_a \to \R$ such that $\phi(a) \neq 0, \chi(a) \neq 0$ and the $q$-jet $jet^q_a(\chi \cdot g) = jet^q_a(\phi \cdot h)$. In other words, $jet^q_a(\chi) \cdot jet^q_a(g)= jet^q_a(\phi) \cdot jet^q_a(h)$ modulo the the $(q+1)^{st}$ power of the maximal ideal $\mathsf m_0$ in the polynomial ring in $\dim X$ variables. 

We denote by $\mathsf{j}^q_a(H)$ the $q$-jet equivalence class of $H \subset X$ at $a$, and by $\mathsf{J}^q_a$ the set of all such $q$-jet equivalence classes of hypersurfaces through $a$. The space $\mathsf{J}^q_a$ depends on finitely many parameters, the number of which is a function of $\dim X$ and $q$. 
Since $\mathsf{J}^1_a$ is the Grassmanian of hyperplanes, we view $\mathsf{J}^q_a$ as a generalized Grassmanian. We denote by $\mathsf{J}^q(X) \to X$ the fibration over $X$, whose fiber is $\mathsf{J}^q_a$.

\begin{definition}\label{def.q-separable}
 We call an immersion $\b: M \to X$ of  a smooth compact $n$-manifold $M$ into a smooth compact $(n+1)$-manifold $X$ {\sf $q$-separable} if: 
\begin{itemize}
\item for any point $a \in \b(M)$, the set $\b^{-1}(a)$ is finite,
\item there is a natural number $q$ such that, for any point $a \in \b(M)$, the $q$-jet equivalence classes of the local branches of $\b(M)$, labeled by points $\b^{-1}(a)$, are all distinct.  \hfill $\diamondsuit$
\end{itemize} 
\end{definition}

For the rest of the section, we fix a compact connected 
smooth $n$-manifold $Y$ and consider proper immersions 
$\b: M \to \R \times Y$ of  compact 
smooth $n$-manifolds $M$. By a small perturbation, we may assume that different local branches of $\b(M)$ are {\sf in  general position} at their mutual intersections (see \cite{GG}, Theorem 4.13). In particular, we may assume that $\b$ is $1$-separable in the sense of Definition \ref{def.q-separable}. 
\smallskip

Recall the following definition from \cite{LS} and \cite{Th},  adjusted for our setting. 

\begin{definition}\label{def.k-normal}  A smooth proper immersion $\b: M \to \R \times Y$ is called $k$-{\sf normal} if, for each $k$-tuple of distinct points $x_1, \ldots , x_k \in M$ with $\b(x_1) = \ldots = \b(x_k)$, the images of the tangent spaces $T_{x_1}M, \ldots, T_{x_k}M$ under the differential $D\b$ are in general position in the tangent space $T_y(\R \times Y)$, where $y = \b(x_i)$. Also, for such a $k$-tuple, the images of the tangent spaces $T_{x_1}(\d M), \ldots, T_{x_k}(\d M)$ under the differential $D(\b|_{\d M})$ are in general position in the tangent space $T_y(\R \times \d Y)$. \hfill $\diamondsuit$
\end{definition}

For a given proper immersion $\b: (M, \d M) \to (\R \times Y,\, \R \times \d Y)$, we consider the $k$-fold product map $(\b)^k: (M)^k \to (\R \times Y)^k$. Let $\Sigma_k$ be the preimage of the diagonal $\D \subset (\R \times Y)^k$ under the map $(\b)^k$. For a $k$-normal $\b$, $\Sigma_k$ is a smooth manifold of dimension $n-k+1$ \cite{LS}. Following  \cite{LS}, we call $\Sigma_k$ {\sf the $k$ self-intersection manifold of} $\b$. Note that the same construction applies to $\b|: \d M \to \R \times \d Y$, producing the boundary $\d\Sigma_k$. 

Let $p_1: (M)^k \to M$ be the obvious projection on the first factor of the product. Then $p_1: \Sigma_k \to M$ is an \emph{immersion} \cite{LS}. Its image $p_1(\Sigma_k)$ is the set of points $x_1 \in M$ such that there exist distinct $x_2, \ldots , x_k \in M$ with the property $f(x_2) = f(x_1),\, \ldots ,\, f(x_k) = f(x_1)$. 

If both $M$ and $Y$ are orientable, then so is $\Sigma_k$. In such a case, we denote by $[\Sigma_k]^\b_{\star} \in H_{n-k+1}(Y, \d Y; \Z)$ the image of the relative fundamental class $[\Sigma_k]$ under the composition $\pi \circ \b \circ p_1$, where $\pi: \R \times Y \to Y$ is the obvious projection. Without orientability assumptions, only the classes $[\Sigma_k]^\b_{\star} \in H_{n-k+1}(Y, \d Y; \Z_2)$ are available. \smallskip

Let $\mathcal D_M: H_\ast(M, \d M) \to H^{n-\ast}(M)$ be the Poincar\'{e} duality isomorphism. (Here, depending on the orientability of $M$, the coefficients are $\Z$ or $\Z_2$.)  We have also the Poincar\'{e} duality isomorphism $\mathcal D_{\R \times Y}: H_\ast(\R \times Y,\, \R \times \d Y) \to H^{n+1-\ast}_c(\R \times Y),$ where the subscript $c$ indicates the cohomology with compact support and the choice of coefficients $\Z$ or $\Z/2\Z$ depends on the orientability of $Y$. \smallskip

Let $\nu_\b := \b^\ast(\tau_{\R \times Y})/\b^\ast(D\b(\tau_M))$ be the $1$-bundle over $M$, normal to $\b(M)$, and  let $W^1(\nu_\b)$ denote its first Stiefel-Whitney characteristic class. Put $\tilde\Sigma_k := p_1(\Sigma_k) \subset M$. Then applying the main result of  \cite{LS} to our setting, for $k > 1$, we get a pleasing formula
\begin{eqnarray}\label{eq.L_S} 
\mathcal D_M([\tilde\Sigma_k]) = \pm\big(\b^\ast\circ\mathcal D_{\R \times Y}\circ \b_\ast([\tilde\Sigma_k]) - W^1(\nu_\b)\big)^{k-1}.
\end{eqnarray}

Applying $\pi \circ \b \circ p_1$ to $\Sigma_k$,these constructions have the following direct implication.

\begin{lemma}\label{lem.k-normal imm}
For compact oriented $n$-manifolds $M$ and $Y$ and $k \in [2, n+1]$, a proper $k$-normal immersion $\b: M \to \R \times Y$ generates the homology class $[\Sigma_k]^\b_{\star} \in H_{n-k+1}(Y, \d Y; \Z)$  and its Poincar\'{e} dual $[\Sigma_k]_\b^{\star} \in H^{k-1}(Y; \Z)$. For non-orientable $M$ or $Y$, a proper $k$-normal immersion $\b$ produces similar classes $[\Sigma_k]^\b_{\star}$ and $[\Sigma_k]_\b^{\star}$ in the homology/cohomology with coefficients in $\Z_2$. \hfill $\diamondsuit$
\end{lemma}

\begin{lemma}\label{lem.Sigma_k_is_bordism_invariant} Let  two proper $k$-normal immersions $$\b_0: (M_0, \d M_0) \to (\R \times Y \times \{0\},\, \R \times \d Y  \times\{0\})$$ $$\b_1: (M_1, \d M_1) \to (\R \times Y \times \{1\}, \R \times \d Y \times \{1\})$$ be cobordant with the help of a proper $k$-normal immersion $$B: (N, \delta N) \to (\R \times Y \times [0, 1],\, \R \times \d Y  \times [0, 1]),$$
where  the boundary $\d N = (M_0 \sqcup -M_1) \cup_{(\d M_0 \sqcup - \d M_1)} \delta N$,  $B|_{M_0} = \b_0$, $B|_{M_1} = \b_1$, and  $B(\delta N) \subset \R \times \d Y\times [0, 1]$.

Then the $k$-intersection manifolds $$\b_0 \circ (p_0)_1 : (\Sigma_k^{\b_0},  \Sigma_k^{\b_0^\d}) \to (\R \times Y \times \{0\},\, \R \times \d Y \times \{0\}),$$
 $$\b_1 \circ (p_1)_1 : (\Sigma_k^{\b_1},  \Sigma_k^{\b_1^\d}) \to (\R \times Y \times \{1\},\, \R \times \d Y \times \{1\})$$ 
are cobordant with the help of the map $$B \circ P_1: (\Sigma_k^B,  \Sigma_k^{B^\delta}) \to (\R \times Y \times [0,1],\, \R \times \d Y \times [0,1]).$$ Here $(p_0)_1: \Sigma_k^{\b_0} \to M_0$, $(p_1)_1: \Sigma_k^{\b_1} \to M_0$, and $P_1: \Sigma_k^B \to N$ denote the projections on the first factors of the $k$-th powers $(M_0)^k$, $(M_1)^k$, and $(N)^k$, respectively. \smallskip
The maps, $(p_0)_1$, $(p_1)_1$, and $P_1$, are immersions.

If the manifolds $M_0, M_1, N,$ and $Y$ are orientable, then so are the manifods $\Sigma_k^{\b_0}$, $\Sigma_k^{\b_1}$, and $\Sigma_k^{B}$.
\end{lemma}

\begin{proof} By the definition of $k$-normality, the $k$-th power $B^k$ of $B$ is transversal to the diagonal $\tilde\Delta  \subset (\R \times Y \times [0, 1])^k$, a manifold with corners $\d Y \times \d[0, 1]$. Thus $\Xi_k^B := (B^k)^{-1}(\tilde\Delta) \subset (N)^k$ will be a compact smooth manifold of dimension $n-k +2$. Its projection $P: (N)^k \to N$ on the first factor $N$ is an immersion \cite{LS}. So the composition of $B \circ P:\, \Xi_k \to \R \times Y \times [0, 1]$ delivers a cobordism between the pair of maps
$$\b_0 \circ (p_0)_1: (\Sigma_k(\b_0), \d\Sigma_k(\b_0)) \to (\R \times Y \times \{0\}, \R \times \d Y \times \{0\})$$
 $$\b_1 \circ (p_1)_1: (\Sigma_k(\b_1), \d \Sigma_k(\b_1)) \to (\R \times Y \times \{1\}, \R \times \d Y \times \{1\}).$$\smallskip
By \cite{LS}, if $M_0, M_1, N,$ and $Y$ are orientable, then so are $\Sigma_k^{\b_0}$, $\Sigma_k^{\b_1}$, and $\Xi_k^B$.
\end{proof}

\begin{corollary}\label{cor.Sigma_k_is_bordism_invariant} Let $\pi: \R \times Y \to Y$ be the obvious projection. Let  proper $k$-normal immersions $\b_0: (M_0, \d M_0) \to (\R \times Y,\, \R \times \d Y)$ and $\b_1: (M_1, \d M_1) \to (\R \times Y,\, \R \times \d Y)$ be cobordant with the help of a proper immersion\footnote{not necessary $k$-normal!} $B: (N, \delta N) \to (\R \times Y \times [0, 1],\, \R \times \d Y  \times [0, 1])$, as in Lemma \ref{lem.k-normal imm}.
 Then both maps $$\pi \circ \b_0 \circ (p_0)_1: (\Sigma_k^{\b_0},  \Sigma_k^{\b_0^\d}) \to (Y, \d Y) \text{\; and \;} \pi \circ \b_1 \circ (p_1)_1: (\Sigma_k^{\b_1},  \Sigma_k^{\b_1^\d}) \to (Y, \d Y)$$ define the same element $[(\Sigma_k^\b, \d\Sigma_k^\b) \to (Y, \d Y)]$ of the relative bordism group $\mathbf B_{n+1-k}(Y, \d Y)$. 

If $M_0, M_1, N,$ and $Y$ are orientable, then the two immersions $\b_0, \b_1$  define the same element of the oriented bordism group $\mathbf{OB}_{n+1-k}(Y, \d Y)$. 
\end{corollary}

\begin{proof} By the Thom Multijet Transversality Theorem \cite{Th1}, we may smoothly perturb the cobordism map $B$ within the space of immersions, without changing it on $\d N \setminus \delta N = M_0 \sqcup M_1$, so that its $k$-th power $B^k$ will become transversal to the diagonal $\tilde\Delta  \subset (\R \times Y \times [0, 1])^k$. For such a perturbation, $B$ becomes $k$-normal, and $\Xi_k := (B^k)^{-1}(\tilde\Delta) \subset (N)^k$ is a compact manifold of dimension $n-k +2$. Now the claim follows from Lemma \ref{lem.Sigma_k_is_bordism_invariant}. 
\end{proof}

\begin{conjecture}\label{conj.separation} Any proper immersion $\b: (M, \d M) \to (\R\times Y, \R\times \d Y)$ admits a $C^\infty$-approximation $\tilde\b$ such that the immersion $\tilde\b$ is $k$-normal for all $k \in [2, n+1]$ and each map $\tilde\b \circ p_1 |: (\Sigma_k^{\tilde\b},  \Sigma_k^{\tilde\b^\d}) \to (\R \times Y, \R \times \d Y)$ is transversal to the $1$-foliation $\mathcal L$, generated by the fibers of the obvious projection $\pi: \R \times Y \to Y$. 
\hfill $\diamondsuit$
\end{conjecture}

The conjecture would imply that, by a perturbation, one can separate in $\R\times Y$ the self-intersections of $\b$ from its tangencies to $\mathcal L$. It is well-known that, for $k \geq 2$, each individual map $(\tilde\b \circ p_1)|_{\Sigma_k^{\tilde\b}}$ can be perturbed to become transversal to $\mathcal L$.

\subsection{Constrained patterns of tangency to the foliation $\mathcal L$ and the classifying maps to the spaces $\mathcal P_d^{\mathbf c\Theta}$}

Let $X$ be a compact topological space. We denote by $C(X)$ the algebra of real continuous functions on $X$. Let 
$$\{P(u) = u^d + a_{d-1}u^{d-1} + \ldots + a_1 u + a_0\}$$ be a family of real monic polynomials on $X$, where the functional coefficients $a_i \in C(X)$. We fix a closed subposet $\Theta \subset \mathbf\Om_{\leq d}$ and assume that, for each point $x\in X$, the polynomial $$P(u, x) = u^d + a_{d-1}(x)u^{d-1} + \ldots + a_1(x) u + a_0(x)$$ has a real zero divisor whose combinatorial pattern $\om(x) \in \mathbf c\Theta := \mathbf \Om \setminus \Theta$.  

Evidently, such polynomial family $P(u)$ gives rise to a continuous map $\Phi_P: X \to \mathcal P_d^{\mathbf c\Theta}$, defined by the formula $x \to P(\sim, x)$, where the target space $\mathcal P_d^{\mathbf c\Theta}$ is an open subset of $\mathcal P_d \approx \R^d$. The preimage of $\mathbf c\Theta$-stratification of $\mathcal P_d^{\mathbf c\Theta}$ under $\Phi_P$ produces a an interesting stratification $\mathcal S(X, P)$ of the space $X$. 

When $X$ is a smooth compact manifold, it is possible to perturb the coefficients of $P$ to form a new family $\tilde P$ so that the perturbed map $\Phi_{\tilde P}$ will be transversal to each pure stratum $\{\mathcal P_d^{\om}\}_{\om \in \mathbf c\Theta}$. If $X$ has a boundary $\d X$, then we may  assume also that $\Phi_{\tilde P}\big |_{\d X}$ is transversal to each pure stratum $\mathcal P_d^{\om}\subset \mathcal P_d^{\mathbf c\Theta}$.

For such a polynomial family $\{\tilde P(u, x)\}_{x \in X}$, the stratification $\mathcal S(X, \tilde P)$ of $X$ mimics the geometry of the $\mathbf c\Theta$-stratification of the target space $\mathcal P_d^{\mathbf c\Theta}$ and gives rise to a variety of topological invariants of $X$, generated by the polynomial family. 
\smallskip

The polynomial family $P = \{P(u, x)\}_{x \in X}$ as above generates  loci 
\begin{eqnarray}
\d\mathcal E_P =_{\mathsf{def}} \; \{(u, x) | \; P(u, x) = 0 \} \subset \R \times X,\nonumber \\
\mathcal   E_P =_{\mathsf{def}} \; \{(u, x) | \; P(u, x) \leq 0 \} \subset \R \times X,
\end{eqnarray}
 which are maped on $X$ by the obvious projection $\pi: \R \times X \to X$. The fiber over $x \in X$ of the projection $\pi^\d: \d\mathcal E_P \to X$ is the support of a real zero divisor $D_\R(x)$ of $P(u, x)$ of a degree $\leq d$. For some $x \in X$, the $x$-fiber may be empty;  then the divisor $D_\R(x)$ is equal to zero.
\smallskip

This setting leads to a natural question: ``What are the maximal sets $\{U_\a \subset X\}_\a$ (open, closed, any ...) which admit a continuous section $\s_\a: U_\a \to \d\mathcal E_P$ of the projection $\pi^\d: \d\mathcal E_P \to X$?" In other words, given a particular root $u_\star$ of $P(u, x_\star)$, what is the \emph{maximal} set $U_{u_\star, x_\star} \subset X$ over which a global solution of the equation $\{P(u, x) = 0\}$, which extends $u_\star$, exists?  The question resembles the question about the maximal domain of an analytic continuation of a given analytic function $f$ (leading to the Riemannian surface associated with $f$). Perhaps, tackling this question deserves a different paper, dealing with the interplay between the fundamental groups $\pi_1(X)$ and $\pi_1(\mathcal P_d^{\mathbf c\Theta})$...
\smallskip

Let $C^k(\sim)$ denote the space of $k$ times continuously differentiable functions. When $X$ is a $C^k$-differentiable manifold and all the coefficients of polynomials $a_i \in C^k(X)$, then $\d\mathcal E_P$ is a $C^k$-differentiable hypersurface in $\R \times X$, since $\d\mathcal E_P$ admits the $C^k$-differentiable parametrization  $(u, a_{d-1}, \ldots, a_1) \to (u, a_{d-1}, \ldots, a_1, P(u) - P(0)).$ \smallskip

Now, given an appropriate hypersurface $\d \mathcal E \subset \R \times X$, we aim \emph{to reverse} the correspondence $P \leadsto (\mathcal E_P, \d \mathcal E_P)$ by constructing a polynomial family $P = P(\d \mathcal E)$ on $X$ that generates $\d\mathcal E$.\smallskip

For a given $d$, let us consider the crucial for us $(d+1)$-dimensional domain
\begin{eqnarray}\label{eq.E}
\mathcal E_d & =_{\mathsf{def}} & \big\{(u, P) \in  \R \times \mathcal P_d |\; P(u) \leq 0\big\} \text{\; and its boundary} \nonumber \\
\d\mathcal E_d & =_{\mathsf{def}} & \big\{(u, P) \in  \R \times \mathcal P_d |\; P(u) = 0\big\}. 
\end{eqnarray}

Since $\d\mathcal E_d$ is a graph of the map $(u, a_1, \dots , a_{d-1}) \longrightarrow  a_0 = -\sum_{i= 1}^{d-1} a_i u^i$, it follows that $\mathcal E_d$ is diffeomorphic to the half-space $\R^{d+1}_+$, and $\d\mathcal E_d$ to the space $\R^d$.  \smallskip

We denote by $\pi_d: \mathcal E_d \to \mathcal P_d$ the obvious projection map.  For $d \equiv 0 \mod 2$, its fibers are compact (are finite unions of closed intervals and singletons).
\smallskip

Let $Y$ be a smooth compact manifold, and $\b: (M, \d M) \to (\R \times Y, \R \times \d Y)$ an immersion. We denote by $\mathcal L$ the one-dimensional oriented foliation, defined by the fibers of the projection map $\pi: \R \times Y \to Y$, and by $\mathcal L_y$ the fiber $\pi^{-1}(y)$.\smallskip 
 
If $\d Y \neq \emptyset$, we add a collar to $Y$ and denote the resulting manifold by $\hat Y$. Similarly, if $\d M \neq \emptyset$, we add a collar to $M$, which results in a new manifold $\hat M$. We still assume that $\b(\d M) \subset \R \times \d Y$ and that $\hat \b(\hat M \setminus M) \subset \R \times (\hat Y\setminus Y)$. Using that $\b$ is transversal to $\R \times \d Y$, we may assume that $\b: M \to \R \times Y$ extends to an immersion  $\hat\b: \hat M \to \R \times \hat Y$. \smallskip

Next, for each point $b \in M$, we introduce a natural number 
$\mu^\b(b)$,  {\sf the multiplicity of tangency} between the 
$b$-labeled local branch of $\hat\b(\hat M)$ 
and the leaf of $\mathcal L$ through the point $a := \b(b)$. If a local branch $\tilde M$ of $\hat\b(\hat M)$ is given as the zero set of a locally defined smooth function $z: \R \times \hat Y \to \R$ which has $0$ as its regular value, then the multiplicity/order of tangency of the $\pi$-fiber $\mathcal L_y$ with $\tilde M$ at a point $a = (u, y)\in \tilde M \cap \mathcal L_y$ is defined as the natural number $\mu := \mu^\b(b)$ such that the jet $\mathsf{jet}^{\mu -1}_{a}(z|_{\mathcal L_y})=0$, but  $\mathsf{jet}^\mu_{a}(z|_{\mathcal L_y})\neq 0$.
In particular, if the branch is transversal to the leaf, then $\mu^\b(b) =1$.\smallskip

We fix a natural number $d$ and  assume that an immersion $\b: M \to \R \times Y$ is such that each leaf 
$\{\mathcal L_{ y}\}_{y \in Y}$ hits the image $\b(M)$ so that: 
\begin{eqnarray}\label{eq5.1}
  \mu^\b(y) =_{\mathsf{def}} \sum_{\{a\, \in \, \mathcal L_{y}\, \cap \, \b(M)\}} 
  \Big(\sum_{\{b\, \in \, \b^{-1}(a)\}} \mu^\b(b)\Big) \; \leq \; d.
\end{eqnarray}

Note that this assumption rules out the infinite multiplicity tangencies $\mu^\b(y)$. 
\smallskip

We order the points $\{a_i\}$ of $\mathcal L_{y} \cap \b(M)$ by the 
values of their projections on $\R$ and introduce {\sf the combinatorial pattern} 
$\om^\b(y)$ of $y \in Y$ as the ordered sequence of multiplicities 
$$\big\{\omega^\b_i(y) =_{\mathsf{def}} \sum_{\{b\, \in \, \b^{-1}(a_i)\}} \,
\mu^\b(b)\big\}_i.$$ We denote by $D^\b(y)$ the real divisor on $\mathcal L_y$, whose support is $\mathcal L_{y} \cap \b(M)$ and whose multiplicities are the $\{\omega^\b_i(y)\}_i$.\smallskip

In particular, the combinatorial type of $D^\b(y)$ is sensitive to the points $ y \in Y$ over which the multiple self-intersections of $\b$ reside. Note also that the definition of the divisor $D^\b(y)$ depends only on the image $\b(M) \subset \R \times Y$ of $M$.
\smallskip

A special case of the following theorem my be found in \cite{KSW1}, Proposition 3.1.

\begin{theorem}\label{th.LIFT} 
  Let $M$ and $Y$ be smooth compact $n$-manifolds. \smallskip
     
  For any  proper (in the sense of Definition \ref{def.im}) immersion $\b: (M, \d M) \to (\R \times Y,\, \R \times \d Y)$ which satisfies inequality
  \eqref{eq5.1} and the parity condition $\mu^\b(y) \equiv d \mod 2$, there exists a smooth map $\Phi^\b: Y \to \cP_d$ such that the locus
  $$\big\{(u, y) \in  \R \times Y |\;\; \Phi^\b(y)(u) = 0\big\} = 
	\b(\d X).$$

  If, for a given pair of closed posets $\Theta \subset \Lambda \subset \bfOm_{\langle d]}$, 
  the immersion $\b$ is such that no $\omega^\b(y)$ belongs to
  $\Theta$ and, for $y \in \d Y$, no $\omega^\b(y)$  belongs to
  $\Lambda$, then $\Phi^\b$ maps $Y$ to the open subset $\cP_d^{\mathbf c\Theta} \subset  \cP_d$ and $\d Y$ to the open subset $\cP_d^{\mathbf c\Lambda}  \subset  \cP_d$.  
 Moreover, any two such maps $\Phi^\b_0, \Phi^\b_1: (Y, \d Y) \to (\cP_d^{\mathbf c\Theta}, \cP_d^{\mathbf c\Lambda})$ are homotopic  as maps of pairs.
\end{theorem}

\begin{proof} Let us denote by $\hat M^\circ$ the interior of $\hat M$ and by  $\hat Y^\circ$ the interior of $\hat Y$.
  
 We claim that the set $\hat\b(\hat M^\circ)$ may be viewed as the solution set of the equations 
  $$\big\{u^d + \sum_{j=0}^{d-1}a_j(y)\,u^j = 0\big\}_{y \in \hat Y^\circ},$$
 where $\{a_j: \hat Y^\circ \to \R\}_j$ are some smooth functions. 

  Let us justify this claim. By Lemma 4.1 from \cite{K2} and 
  Morin's Theorem \cite{Mor} (both based on the Malgrange Preparation Lemma), if a particular branch 
  $\hat\b(\hat M^\circ)_\kappa$ of $\hat\b(\hat M^\circ)$ is tangent to the leaf $\mathcal L_{y_\star}$ 
  at a point $a_\star = (\a, y_\star)$ with the order of tangency $j = j(a_\star, \kappa)$, then 
  there is a system of local coordinates $(\tilde u, \tilde x, \tilde z)$ in the 
  vicinity of $a \in \R \times \hat Y^\circ$ such that:

  \begin{itemize}
    \item[(1)] 
      $\hat\b(\hat M^\circ)_\kappa$ is given by the equation 
		  $\big\{ \tilde u^j + \sum_{k=0}^{j-1}\,\tilde x_k \, \tilde u^k = 0\big\}$,
		  where $\tilde x_k(y_\star) = 0$ and $\tilde u(\a) = 0$,

    \item[(2)] each nearby leaf $\mathcal L_{y}$ is given by the 
	    equations $\{\tilde x = \overrightarrow{const},\, \tilde z = \overrightarrow{const'}\}$. \end{itemize}
 
  Letting $\tilde u= u -\a$ and writing $\tilde x_k$'s as smooth functions of 
  $y \in \hat Y^\circ$, the same locus $\hat\b(\hat M^\circ)_\kappa$ can be given by the equation
   \begin{eqnarray}\label{eq5.2}
   \big\{P_{\a, \kappa}(u, y) := 
  (u-\a)^j + \sum_{j=0}^{j-1}a_{\kappa, k}(y)\,(u-\a)^k = 0\big\}, 
  \end{eqnarray} 
  where $a_{\kappa, k}: \hat Y^\circ \to \R$ are smooth functions, vanishing at 
  $y_\star$. 
  Therefore, there exists an open neighborhood $U_{y_\star}$ of 
  $y_\star$ in $\hat Y^\circ$ such that, in $\R \times U_{y_\star}$, 
  the locus $\hat \b(\hat M^\circ)$ is given by the monic polynomial equation 
  \begin{eqnarray}\label{eq5.3}
  \Big\{P_{y_\star}(u, y) := 
    \prod_{(\a,\, y_\star)\, \in \, \mathcal L_{y_\star}\, \cap\, \hat\b(\hat M^\circ)} 
    \Big(\prod_{\kappa \in A_\a} P_{\a, \kappa}(u, y)\Big) = 0 \Big\},
    \end{eqnarray}
  of degree $\mu^\b(y_\star) \leq d$ in $u$. Here the finite set $A_\a$ 
  labels the local branches of $\b(M)$ that contain the point 
  $(\a, y_\star) \in \mathcal L_{y_\star} \cap \b(M)$. 

  By multiplying $P_{y_\star}(u, y)$ with 
  $(u^2 +1)^{\frac{d- \mu^\b(y_\star)}{2}}$, we get a polynomial 
  $\tilde P_{y_\star}(u, y)$ of degree $d$. For each $y \in U_{y_\star}$, $\tilde P_{y_\star}(u, y)$ shares with $P_{y_\star}(u, y)$ the zero set $\hat\b(\hat M^\circ) \cap (\R \times U_{y_\star})$, as well as the divisors $D^{\hat\b}(y)$.  

  For each $y \in \hat Y^\circ$, we consider the space $\mathcal X_{\hat \b}(y)$ of monic polynomials $\tilde P(u)$ of degree $d$ such that their real divisors 
  coincide with the $\hat\b$-induced divisor $D^{\hat\b}(y)$. We view 
  $\mathcal X_{\hat \b} =_{\mathsf{def}}\, \coprod_{y \in \hat Y^\circ} \mathcal X_{\hat\b}(y)$ 
  as a subspace of $\hat Y^\circ \times \mathcal P_d$. It is equipped with the obvious 
  projection $p: \mathcal X_{\hat \b} \to \hat Y^\circ$. The smooth sections of the map $p$ are exactly the smooth functions $\tilde P(u, y)$ that interest us. By the previous argument, $p$ admits a \emph{local} smooth section over the vicinity of each $y_\star \in \hat Y^\circ$.
  
  Evidently, each $p$-fiber $\mathcal X_{\hat\b}(y)$ is a convex set. 
  Thus, given finitely many smooth 
  sections $\{\sigma_i\}_i$ of $p$, we conclude that 
  $\sum_i \phi_i \cdot \sigma_i$ 
  is again a section of $p$, provided that the smooth functions 
  $\phi_i: Y \to [0, 1]$ have the property $\sum_i \phi_i \equiv \mathbf 1$.
  Note that an individual term $\phi_i \cdot \sigma_i$ may not belong to $\mathcal P_d$ due to the failure of the polynomials to be monic.
  \smallskip

Since $Y$ is compact, it admits a finite cover by the open sets $\{U_{y_i} \subset \hat Y^\circ\}_i$ as above (see formulas (\ref{eq5.2}), (\ref{eq5.3})). 
  Let $\{\phi_i: Y \to [0, 1]\}_i$ be a smooth partition of unity, 
  subordinated to this finite cover. Then the monic $u$-polynomial 
  $$\tilde P(u,y) := \sum_i \phi_i(y)\cdot 
     \tilde P_{y_i}(u, y)$$ of degree $d$ has the desired 
  properties. In particular, its divisor $\tilde D^\b(y)= D^\b(y)$ for each 
  $y \in Y$. Thus, using $\tilde P(u, y)$, any immersion 
  $\b: (M, \d M) \to (\R \times Y,\, \R \times \d Y)$, such that: \smallskip
  
  (i) $\b$ is transversal to $\R \times \d Y$, 
  
  (ii) no $\omega^\b(y)$ belongs to $\Theta$,  
  
  (iii) for $y \in \d Y$, no $\omega^\b(y)$ belongs to $\Lambda$, \smallskip
  
\noindent  is realized by a smooth map $\Phi^\b: (Y, \d Y) \to (\cP_d^{\mathbf c\Theta}, \cP_d^{\mathbf c\Lambda})$ for which $\b(M) = \{\Phi^\b(y)(u) = 0\big\}$. 
  
If two such maps $\Phi^\b_0, \Phi^\b_1: (Y, \d Y) \to (\cP_d^{\mathbf c\Theta}, \cP_d^{\mathbf c\Lambda})$ share the same divisor $D^\b(y)$ for each $y \in Y$, then the quotients $\Phi^\b_0(y)/ \Phi^\b_1(y)$ and $\Phi^\b_1(y)/ \Phi^\b_0(y)$ are \emph{strictly positive}  rational functions  
and their numerators and denominators are monic non-vanishing polynomials of degree $d$.  
Such rational functions form a convex set $\mathcal Q_+(d)$ that retracts to the point-function $\mathbf 1$. Therefore $\Phi^\b_0, \Phi^\b_1$ produce a continuous map $\Phi^\b_0/ \Phi^\b_1$ into a contractible space $\mathcal Q_+(d)$. As a result, by the linear homotopy $\{(1-t)\Phi^\b_0+ t\Phi^\b_1\}_{t \in [0, 1]}$, $\Phi^\b_0, \Phi^\b_1$ are homotopic maps.
\end{proof}

\begin{figure}[ht]
\centerline{\includegraphics[height=2.7in,width=5in]{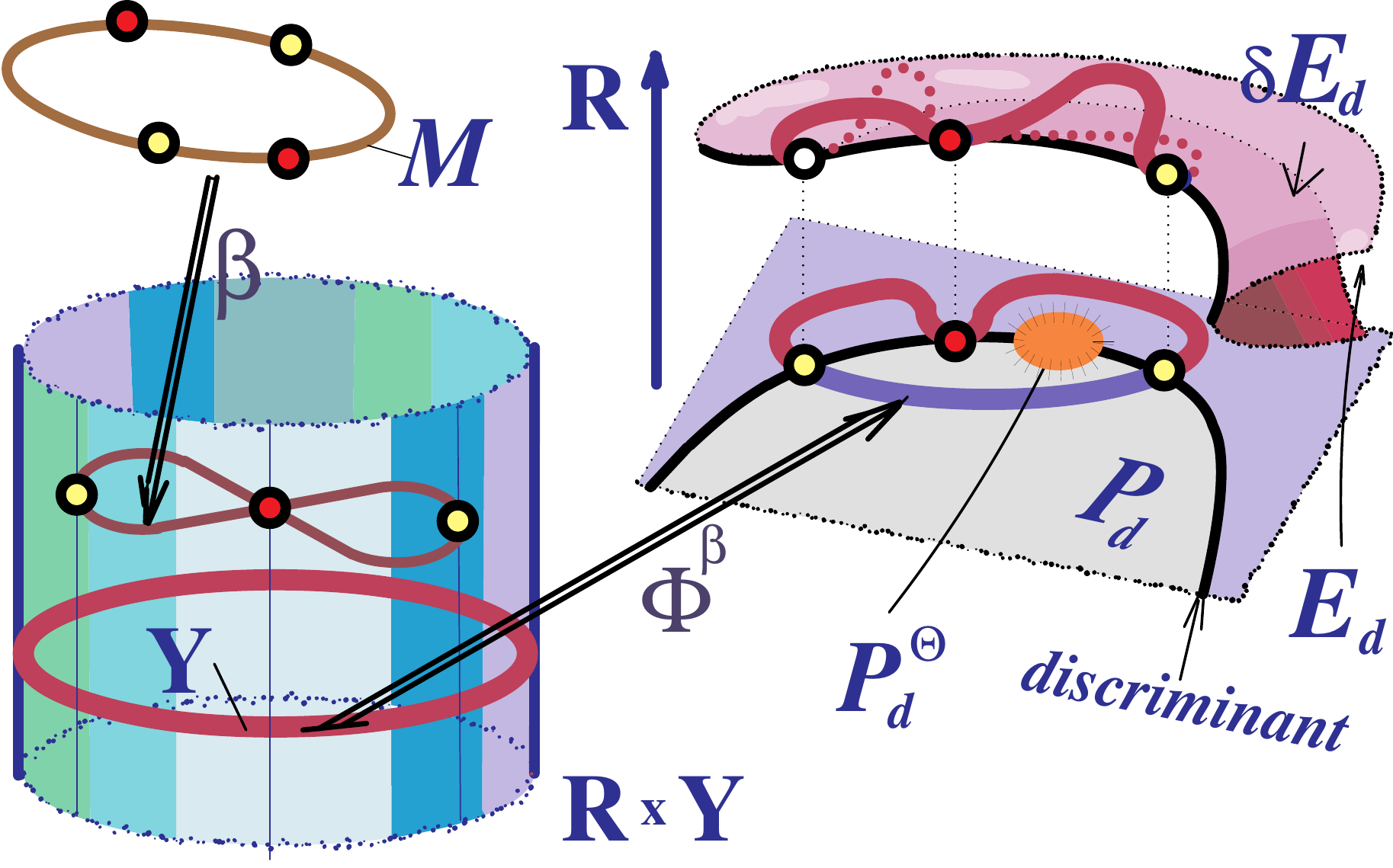}}
\bigskip
\caption{\small{An immersion $\b:M \to \R \times Y$ and the classifying map $\Phi^\b: Y \to \cP_d^{\mathbf c\Theta}$ its generates. Note that $\b(M) = (\mathsf{id}\times \Phi^\b)^{-1}\big((\R \times \Phi^\b(Y)) \cap \d\mathcal E_d\big)$}. To get the figure realistic in $3D$, we have chosen the simplest possible case $d=2$, $\Theta = \emptyset$. } 
\label{fig.quasitopy}
\end{figure}

Fig. 1, which illustrates the construction of the classifying map $\Phi^\b$, reveals an interesting duality: if $(u_\star, y_\star) \in \R \times Y$ is the point of {\it transversal} self-intersection of $\b(M) = ``\infty"$, then $y_\star$ is mapped by $\Phi^\b$ to a point where the curve $\Phi^\b(Y)$ is {\it tangent} to the discriminant variety $\mathcal D_2$; in contrast, if $(u_\ast, y_\ast) \in \R \times Y$ is the point of {\it tangency} of $\b(M)$ to $\mathcal L$, then $y_\ast$ is mapped by $\Phi^\b$ to a point  where  $\Phi^\b(Y)$ is {\it transversal} to $\mathcal D_2$.

\begin{definition}\label{def.E-regular} 
Let $Y$ be a smooth compact $n$-manifold and the domain $\mathcal E_d \subset \R \times \mathcal P_d$ be as in (\ref{eq.E}). A smooth map $\Phi: Y \to \mathcal P_d$ is called $(\d\mathcal E_d)$-{\sf regular} if the maps $\R \times Y \stackrel{\mathsf{id} \times \Phi}{\longrightarrow} \R \times \mathcal P_d$ and $\R \times \d Y \stackrel{\mathsf{id} \times \Phi^\d}{\longrightarrow} \R \times \mathcal P_d$ are transversal to the hypersurface $\d\mathcal E_d$.
 \hfill $\diamondsuit$
\end{definition}

Thus, for a $(\d\mathcal E_d)$-regular $\Phi$, the preimage $(\mathsf{id} \times \Phi)^{-1}(\d\mathcal E_d)$ is a compact smooth $n$-manifold $(M, \d M) \subset (\R \times Y,\, \R\times \d Y)$. 
Of course, if $\b: (M, \d M) \to (\R \times Y,\, \R\times \d Y)$ is an immersion, but not an embedding, then any associated map $\Phi^\b: (Y, \d Y) \to (\cP_d^{\mathbf c\Theta}, \cP_d^{\mathbf c\Lambda})$ is not $\d \mathcal E_d$-regular. \smallskip

Note that, in Fig. 1, we can perturb the map $\Phi^\b: Y \to \cP_d$ so that it will become  transversal to the discriminant variety. Such a perturbation $\tilde\Phi^\b$ may be assumed to be $(\d\mathcal E_d)$-regular. It will resolve the image $\b(M)$ into a nonsingular manifold $N \subset \R \times Y$. In general, $N$ may be topologically different from the original $M$. These observations are generalized in the next couple lemmas.

\begin{lemma}\label{lem.E-regular} For a compact smooth manifold $Y$, the $(\d\mathcal E_d)$-regular maps $\{\Phi: Y \to \mathcal P_d\}$ form an open and dense set in the space $C^\infty(Y, \mathcal P_d)$ of all smooth  maps.
\end{lemma}

\begin{proof}
A smooth map $\Phi: Y \to \mathcal P_d$, given by $d$ functions-coefficients $x_{d-1}(y), \dots , x_1(y), x_0(y)$ on $Y$,  is $(\d\mathcal E_d)$-regular if and only if, in any local coordinate system $\vec y = \{y_1, \dots y_n\}$ on $Y$, the system 
\begin{eqnarray}\label{E_reg} 
 u^d + x_{d-1}u^{d-1} + \ldots x_1u + x_0 & = & 0 \nonumber  \\
d\, u^{d-1} + (d-1) x_{d-1}u^{d-2} + \ldots + x_1 & = & 0 \nonumber \\ 
\Big\{\frac{\d x_{d-1}}{\d y_j} u^{d-1} + \ldots \frac{\d x_{1}}{\d y_j}u + \frac{\d x_0}{\d y_j} & = & 0\Big\}_{j \in [1, n]}
\end{eqnarray}
of $(n+2)$ equations has no solutions in $\vec y$ for all $u$, and a similar property holds for $\d Y$. Indeed, $\d\mathcal E_d$ is given by the equation $\wp(u, \vec x) := u^d + x_{d-1}u^{d-1} + \ldots x_1u + x_0 = 0$. The pull-back $\Psi^\ast(\wp)$ of $\wp$ under the map $\Psi = (\mathsf{id}, \Phi): \R \times Y \to \R \times \mathcal P_d$ is the function $$u^d + x_{d-1}(y) u^{d-1} + \ldots x_1(y) u + x_0(y)$$ on $\R \times Y$. So the first equation in (\ref{E_reg})  defines the preimage of $\d\mathcal E$ under $\Psi$.  The transversality of $\Psi$ to $\d\mathcal E_d$ can be expressed as the non-vanishing of the $1$-jet of $\Psi^\ast(\wp)$ along the locus $\{\Psi^\ast(\wp) = 0\}$. In local coordinates on $\R \times Y$, the vanishing of $\mathsf{jet}^1(\Psi^\ast(\wp))$ is exactly the constraints imposed by (\ref{E_reg}).
\smallskip

Note that, for each $u \in \R$, the system (\ref{E_reg}) imposes $(n+2)$ \emph{affine} constraints on the functions $\{x_k: Y \to \R\}_{k \in [1,d]}$ and their first derivatives $\big\{\frac{\d x_k}{\d y_j}\big\}$. Thus for any $u$, (\ref{E_reg}) defines an affine subbundle $\mathcal W(u)$ of the jet bundle $\{\mathsf J^1(Y, \mathcal P_d) \to Y\}$. The union $\mathcal W = \bigcup_{u \in \R} \mathcal W(u)$ is a ruled variety, residing in $\mathsf J^1(Y, \mathcal P_d)$. Since $\text{codim}(\mathcal W(u)) = n+2$, the codimension of $\mathcal W$ in $\mathsf J^1(Y, \mathcal P_d)$ is $n+1$. 

Consider the jet map $\mathsf{jet}^1(\Phi): Y \to \mathsf J^1(Y, \mathcal P_d)$. By the Thom Transversality Theorem (see \cite{GG}, Theorem 4.13), the space of $\Phi$ for which $\mathsf{jet}^1(\Phi)$ is transversal to the subvariety $\mathcal W$ is open and dense (recall that $Y$ is compact). Since $Y$ is $n$-dimensional, this transversality implies that $(\mathsf{jet}^1(\Phi))(Y) \cap \mathcal W = \emptyset$ for an open and dense set of maps $\Phi$. 

A similar arguments apply to the smooth maps $\Phi^\d : \d Y \to \mathcal P_d$. Thus we may perturb first any given $\Phi: Y \to \mathcal P_d$ to insure the $(\d\mathcal E_d)$-regularity of $\Phi^\d = \Phi|_{\d Y}$ and then perturb $\Phi$ to insure its $(\d\mathcal E_d)$-regularity, while keeping the regularity of $\Phi^\d$.   

Therefore, the set of $(\d\mathcal E_d)$-regular maps $\Phi$ is open and dense in $C^\infty(Y, \mathcal P_d)$. 
\end{proof}

\begin{corollary} \label{cor.E-reg} 
Let $\Theta \subset  \Lambda \subset \mathbf\Om_{\langle d]}$ be closed subposets. For a compact smooth manifold $Y$, the $(\d\mathcal E_d)$-regular maps $\{\Phi: (Y, \d Y) \to \mathcal (P_d^{\mathbf c \Theta}, P_d^{\mathbf c\Lambda})\}$ form an open and dense set in the space of all smooth maps $C^\infty\big((Y, \d Y), (P_d^{\mathbf c \Theta}, P_d^{\mathbf c\Lambda})\big)$.
\end{corollary}

\begin{proof} Since $Y$ is compact and the posets $\Theta, \Lambda$ are closed, the target spaces $\mathcal P_d^{\mathbf c \Theta} \supset \mathcal P_d^{\mathbf c \Lambda}$ are both open in $\mathcal P_d$. Therefore the claim follows from Lemma \ref{lem.E-regular}.  
\end{proof}

\begin{lemma}\label{E_bounds} 
For a compact $n$-dimensional $Y$, any $\d \mathcal E_d$-regular smooth map $h: Y \to \cP_d^{\mathbf c\Theta}$ is realized by an \emph{embedding} $\b_h: (M, \d M) \hookrightarrow (\R\times Y,\, \R\times \d Y)$, where the smooth manifold $M$ is $n$-dimensional and $\mathsf{rk}(D\b_h) = n$. The manifold $M$ is orientable if $Y$ is.\smallskip

If $\d Y = \emptyset$ and and $d \equiv 0 \mod 2$, then $M$ is a boundary of a compact (orientable, if $Y$ is orientable) $(n+1)$-manifold $L \subset \R\times Y$. 
\end{lemma}

\begin{proof}
A $\d \mathcal E_d$-regular smooth map $h: Y \to \cP_d^{\mathbf c\Theta}$ produces the locus $M = (\mathsf{id} \times h)^{-1}(\d \mathcal E_d)$ which is a smooth manifold by the transversality of $\mathsf{id} \times h$ to $\d \mathcal E_d$. If $Y$ is orientable, so is $\R \times Y$. Thus, the the pull-back of the normal vector field to $\d \mathcal E_d$ in $\mathcal E_d$ helps to orient the   tangent bundle of $M$.
 
Given an embedding $\b: M \hookrightarrow \R\times Y$, $\dim M = n$, whose tangent patterns belong to $\mathbf{c}\Theta$ and using that $d \equiv 0 \mod 2$, we conclude that when $Y$ is closed, then $\b(M)$ bounds a compact $(n+1)$-manifold $X \subset \R \times Y$. Indeed, for each $y \in Y$, we consider the even degree $d$ real monic polynomial $P_y(u)$ whose real divisor is $\mathcal L_y \cap \b(M)$, counted with the multiplicities. Then we define $L \cap \mathcal L_y$ as the compact set of $u \in \R$ that satisfies $\{P_y(u) \leq 0\}$. The $Y$-family of such inequalities determines $L$.
\end{proof}

As a result,  any closed $n$-manifold $M$, which is \emph{not} (orientably, if $Y$ is orientable) cobordant to $\emptyset$, does not arise via $\d \mathcal E_d$-regular maps $h: Y \to \mathcal P_d$ for any closed $Y$ and $d \equiv 0 \mod 2$. \smallskip

 Consider the {\sf tautological function} $\wp: \R \times \mathcal P_d \to \R$ whose value at the point $(u, P)$ is $P(u)$.

\begin{definition}\label{def.kE-regular} 
Let $Y$ be a smooth compact $n$-manifold.  Let $k \leq d+1$.

We call a smooth map $\Phi: Y \to \mathcal P_d$  $(\d\mathcal E_d,  
k)$-{\sf regular} if the map $\R \times Y \stackrel{\mathsf{id} \times \Phi}{\longrightarrow} \R \times \mathcal P_d$ 
has the following properties: 

\begin{itemize}
\item the pull-back $(\mathsf{id} \times\Phi)^\ast(\wp)$ of the function $\wp$ in the vicinity $U_a$ of every point $$a = (u, y) \in \mathcal M_\Phi := \{(\mathsf{id}\times \Phi)^\ast(\wp) = 0\}$$  is locally a product of at most $\ell \leq k$ smooth functions $\{z_i: U_a \to \R\}_{i \in [1, \ell]}$  for each of which $0$ is a regular value, \smallskip

\item there is a natural number $q$ such that the $q$-jet equivalence classes of the local branches $\{z_i = 0\}_i$ of $\mathcal M_\Phi$ at $a$ are distinct for all points 
$a$,\footnote{The branches are $q$-separable in the sense of Definition \ref{def.q-separable}.} \smallskip

\item if $y \in \d Y$,  the restrictions $\{z_i |_{\R \times \d Y}\}_i$ have $0$ as a regular value, and the $q$-jet equivalence classes of local branches $\{z_i |_{\R \times \d Y} = 0\}_i$ are distinct as well. \hfill $\diamondsuit$
\end{itemize}
\end{definition}

Evidently,  $(\d\mathcal E_d, 1)$-regular map is $(\d\mathcal E_d)$-regular. \smallskip

Informally, we would like to think of $(\d\mathcal E_d, k)$-regular maps $\Phi$ as hypersurfaces that divide the space of $(\d\mathcal E_d, k-1)$-regular maps into chambers. Presently, this interpretation is just a wishful thinking... What is clear that, if a map $\Phi: Y \to \mathcal P_d$ is $(\d\mathcal E_d, k)$-regular, then there exists its open neighborhood in the space $C^\infty(Y, \mathcal P_d)$ that does not contain any $(\d\mathcal E_d, l)$-maps, where $l > k$.

\begin{lemma}\label{lem_kE-regular}
If a map $\Phi: Y \to \mathcal P_d$ is $(\d\mathcal E_d, k)$-regular, then the locus 
$$\mathcal M_\Phi =_{\mathsf{def}} \{(\mathsf{id}\times \Phi)^\ast(\wp) = 0\} \subset \R \times Y$$ is a compact $n$-dimensional set with singularities of the local types $\{\prod_{i=1}^{\ell} z_i = 0\}$, where $2 \leq \ell \leq k \leq d+1$, and each $z_i$ has $0$ as its regular value. 

Moreover, $\mathcal M_\Phi$ is the image of a smooth compact $n$-manifold $M$ under an immersion $\b: M \to   \R \times Y$. 
\end{lemma}

\begin{proof} The validation of the first claim is on the level of definitions. 
\smallskip

Recall that we denote by $\mathsf{J}^q(X) \to X$ the fibration over a smooth manifold $X$, whose fiber over a point $x \in X$ is $\mathsf{J}^q_x$, the space of $q$-equivalence classes of germs of smooth hypersurfaces in $X$ at $x$.

It remains to show that, if a map $\Phi: Y \to \mathcal P_d$ is 
$(\d\mathcal E_d, k)$-regular and, for some $q \geq 1$, the branches of the locus $\mathcal M_\Phi  \subset \R \times Y$ at their mutual intersections are $q$-separable, then $\mathcal M_\Phi$ is the immersed image of some $n$-manifold $M$. In fact, there is an obvious \emph{canonical resolution} $M \subset \mathsf{J^q}(\R \times Y)$ of $\mathcal M_\Phi$,  so that the projection $\b: M \subset \mathsf{J^q}(\R \times Y) \stackrel{\pi}{\longrightarrow} \R \times Y$ is the desired immersion. We just associate with each point $a \in \{\prod_{i=1}^{\ell} z_i = 0\}$ the (unordered) set of $\ell \leq k$ distinct points $\{\mathsf J^q_a(\{z_i = 0\})\}_{i \in [1, \ell]} \subset \mathsf{J}^q_a(\R \times Y)$. In this way, each local branch $\mathcal M_{\Phi, i} := \{z_i = 0\}$ of $\mathcal M_\Phi$ defines a smooth section $\s_i: \mathcal M_{\Phi, i} \to \mathsf J^q(\R \times Y)|_{\mathcal M_{\Phi, i}}$. Over the vicinity of $a$, thanks to the $q$-separability, all the sections $\{\s_i\}_i$ are disjoint. Since $\pi \circ \s_i = \mathsf{id}$, we conclude that $\b: M \to \mathcal M_\Phi$ is an immersion. 
\end{proof}

\begin{definition}\label{def.k-flat} 
We say that an immersion $\b: M \to \R \times Y$ is $k$-{\sf flat}, if  
for all $y \in Y$, the reduced multiplicity
\begin{eqnarray}\label{eq.k-flat}
  (\mu^{\b})'(y) =_{\mathsf{def}} \sum_{\{a\, \in \, \mathcal L_{y}\, \cap \, \b(M)\}} 
  \Big(\sum_{\{b\, \in \, \b^{-1}(a)\}} (\mu^\b(b) -1)\Big) \; < \; k.
\end{eqnarray}

In particular, if $\b$ is $k$-flat and $k$-normal, then the $k$-intersection manifold $\Sigma^\b_k = \emptyset$. \hfill $\diamondsuit$
\end{definition}

\subsection{Quasitopies of immersions \&  embeddings with constrained tangencies against the background of $\mathbf 1$-foliations $\mathcal L$: the case of general combinatorics} \hfill \break

We use the abbreviation ``$\mathsf{imm}$" for immersions and ``$\mathsf{emb}$" for regular embeddings. When the arguments work equally well for both types, we  use the abbreviation ``$\mathsf{i/e}$".
\smallskip

Fix two natural numbers $d \leq d'$, $d \equiv d' \mod 2$ and consider the embedding $\e_{d, d'}: \mathcal P_d \hookrightarrow \mathcal P_{d'}$, defined by the formula (\ref{eq.stable}).
Recall that it preserves the $\mathbf\Om$-stratifications of the two spaces, $\mathcal P_d$ and $\mathcal P_{d'}$, by the combinatorial types $\om$ of real divisors $D_\R(P)$, $D_\R(\e_{d, d'}(P))$. 

For a closed profinite (see Definition \ref{def.profinite}) poset $\Theta \subset \mathbf\Om$, the embeddings $\e_{d, d'}$ make it possible to talk about the stabilization in the homology of spaces $\mathcal P_d^{\mathbf c\Theta}$, as $d \to \infty$. With the help of $\{\e_{d, d'}\}$, it also makes sense to introduce the limit spaces $\bar{\mathcal P}_\infty^\Theta$ and $\mathcal P_\infty^{\mathbf c \Theta}$ \cite{KSW2}.
\smallskip 
 
For a smooth compact connected $n$-manifold $Y$,  
put $Z =_{\mathsf{def}}\, Y \times [0, 1]$ and $Z^\delta =_{\mathsf{def}}\, \d Y \times [0, 1]$. We denote by $\mathcal L^\bullet$ the $1$-dimensional oriented foliation of $\R \times Z$, produced by the fibers of the obvious projection $\Pi: \R \times Z \to Z$. \smallskip

In the Definition \ref{quasi_isotopy_of_ immersions} below, central to our investigation, we start with a quite general set of combinatorial input data: $\{n, d, d', \Theta, \Theta', \Lambda\}$. We will gradually restrict them, as we develop the theory. Fig. \ref{fig.quasitopy} may help the reader to follow our unfortunately cumbersome notations.

\begin{definition}\label{quasi_isotopy_of_ immersions} Let us fix natural numbers $d \leq d'$, $d \equiv d' \mod 2$, and a triple of closed subposets $\Theta' \subset  \Theta \subset \Lambda$ of the universal poset $\mathbf\Om$. Let $Y$ be a fixed smooth compact $n$-manifold and $M_0, M_1$ be two smooth compact $n$-manifolds. 

We say that two proper immersions/embeddings, $$\b_0: (M_0, \d M_0) \to (\R \times Y,\, \R \times \d Y) \text{\; and \;} \b_1: (M_1, \d M_1) \to (\R \times Y,\, \R \times \d Y),$$ are $(d, d'; \mathbf c\Theta,  \mathbf c\Lambda; \mathbf c\Theta')$-{\sf quasitopic}, if 
there exists a compact smooth  $(n+1)$-manifold $N$ with corners $\d M_0 \coprod \d M_1$ and the boundary $\d N = (M_0 \coprod M_1) \cup \delta N$,\footnote{$\delta N$ is the closure of the complimentary to $M_0 \coprod M_1$ portion of $\d N$.} and a smooth proper immersion/embedding $B: N \to \R \times Z$ such that:

\begin{itemize} 
\item $B|_{M_0} = \b_0$, $B|_{M_1} = \b_1$, and $B(\delta N) \subset \R \times \d Y \times [0, 1] := \R \times Z^\delta$,
\smallskip
  
\item for each $z \in Z$, the total multiplicity $m_B(z)$ (see (\ref{eq5.1})) of $B(N)$ with respect to the fiber $\mathcal L^\bullet_z$ is such that: $m_B(z) \leq d'$, $m_B(z) \equiv d' \mod 2$, 
and the combinatorial tangency pattern $\omega^B(z)$ of $B(N)$ with respect to $\mathcal L^\bullet_z$ belongs to the poset $\mathbf c\Theta'$,  
\smallskip
  
\item for each $z \in Y\times \d[0,1]$, the total multiplicity of $B(N)$ with respect to  $\mathcal L_z$ is such that: $m_B(z) \leq d$,  $m_B(z) \equiv d \mod 2$,  and the combinatorial tangency pattern $\omega^B(z)$ of $B(N)$ with respect to $\mathcal L_z$ belongs to the poset $\mathbf c\Theta$, 
\smallskip
   
\item for each $z \in Z^\delta$, the multiplicity of $B(N)$ with respect to  $\mathcal L^\bullet_z$ is such that: $m_B(z) \leq d'$,  $m_B(z) \equiv  d' \mod 2$, 
and the combinatorial tangency pattern $\omega^B(z)$ of $B(\delta N)$ with respect to $\mathcal L^\bullet_z$ belongs to the poset $\mathbf c\Lambda$.
\end{itemize}
\smallskip

We denote by $\mathcal{QT}^{\mathsf{imm}}_{d, d'}(Y, \d Y; \mathbf{c}\Theta, \mathbf{c}\Lambda; \mathbf{c}\Theta')$ the set of {\sf quasitopy classes of such immersions} $\b: (M, \d M) \to  \R \times (Y, \d Y)$. 
\smallskip
We use the notation $\mathcal{QT}^{\mathsf{emb}}_{d, d'}(Y, \d Y; \mathbf{c}\Theta, \mathbf{c}\Lambda; \mathbf{c}\Theta')$ for the set of {\sf quasitopy classes of  embeddings} $\b: (M, \d M) \hookrightarrow  \R \times (Y, \d Y)$.
Finally, we  use the neutral notation $\mathcal{QT}^{\mathsf{i/e}}_{d, d'}(Y, \d Y; \mathbf{c}\Theta, \mathbf{c}\Lambda; \mathbf{c}\Theta')$ for both. \smallskip

It is possible to build a parallel notion of quasitopies for {\sf oriented} $M$'s  by insisting that the cobordism $N$ is oriented as well.  We use the notation $\mathcal{OQT}^{\mathsf{i/e}}_{d, d'}(Y, \d Y; \mathbf{c}\Theta, \mathbf{c}\Lambda; \mathbf{c}\Theta')$
for these oriented quasitopy classes.
\hfill $\diamondsuit$
\end{definition}

When $Y$ is a closed manifold, we get $\mathcal{QT}^{\mathsf{i/e}}_{d, d'}(Y; \mathbf{c}\Theta; \mathbf{c}\Theta')$, a simplification of our settings. Also, when $\Theta = \Theta'$, we get another natural simplification: $\mathcal{QT}^{\mathsf{i/e}}_{d, d'}(Y, \d Y; \mathbf{c}\Theta, \mathbf{c}\Lambda; \mathbf{c}\Theta)$. Both special cases, $d' = d$ and $d' = d+2$, have significant applications. 
\smallskip

\begin{figure}[ht]
\centerline{\includegraphics[height=3.2in,width=4.3in]{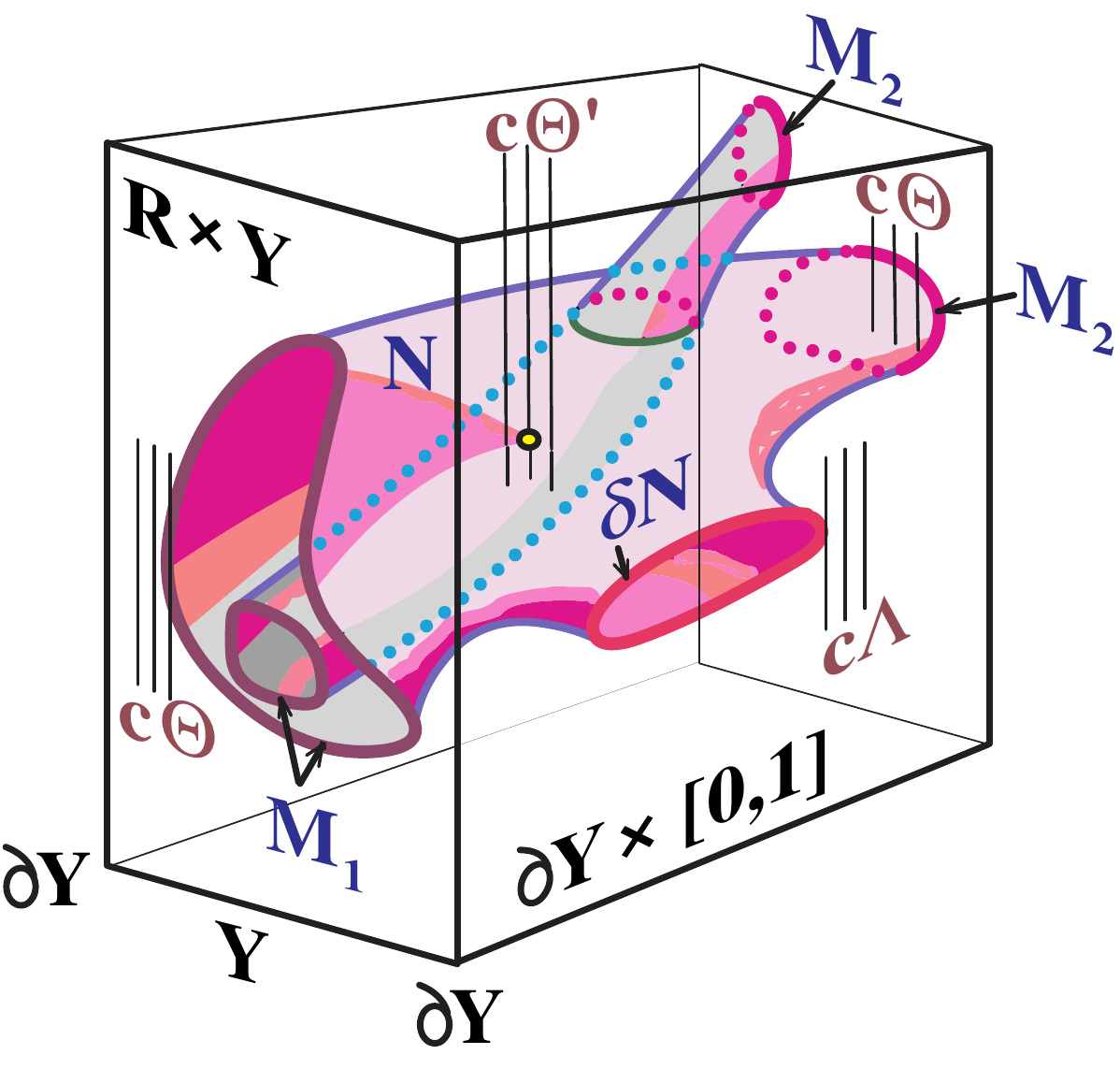}}
\bigskip
\caption{\small{The ingredients of Definition 3.7.}} 
\label{fig.quasitopy}
\end{figure}

\begin{remark} \emph{If in Definition \ref{quasi_isotopy_of_ immersions} we would require that the pair $(N, \delta N)$ is diffeomorphic to the pair $(M_0 \times [0, 1],\, \d M_0 \times [0, 1])$, then a more recognizable definition of {\sf pseudo-isotopy} would emerge. So the notion of {\sf quasitopy} is more flexible than the one of {\sf pseudo-isotopy}. It is closer to the notion of bordisms.}
\smallskip
\hfill $\diamondsuit$
\end{remark}

\begin{remark} \emph{The quasitopy $\mathcal{QT}^{\mathsf{i/e}}_{d, d'}(\sim,\; \d \sim; \mathbf{c}\Theta, \mathbf{c}\Lambda; \mathbf{c}\Theta')$ is covarientely functorial under the \emph{regular embeddings} $f: (Y, \d Y) \hookrightarrow (Y' , \d Y')$ of \emph{equidimentional} manifolds $Y, Y'$. At the same time, $\mathcal{QT}^{\mathsf{i/e}}_{d, d'}(\sim, \d \sim; \mathbf{c}\Theta, \mathbf{c}\Lambda; \mathbf{c}\Theta')$ is contravariantly  functorial under \emph{submersions} $g: (Y', \d Y') \to (Y , \d Y)$,  provided that the fibers of $g$ are closed manifolds. The contravariance is delivered via the pull back construction which involves $\mathsf{id}_\R \times g$ and $\b: (M, \d M) \to (\R \times Y, \R \times \d Y)$.} 
\hfill $\diamondsuit$
\end{remark}

Consider the group $\mathsf{Diff}(\mathcal L)$ of smooth 
diffeomorphisms of $\R \times Y$ that preserve the $1$-foliation $\mathcal L$ and its orientation. We denote by $\mathsf{Diff}_0(\mathcal L)$ its subgroup, generated by the  diffeomorphisms that are isotopic to the identity. Similar groups, $\mathsf{Diff}(\mathcal L^\bullet)$ and $\mathsf{Diff}_0(\mathcal L^\bullet)$ are available for the foliation $\mathcal L^\bullet$ on $\R \times Y \times [0,1]$. Note that the group $\mathsf{Diff}(Y)$ is, in the obvious way, a subgroup of $\mathsf{Diff}(\mathcal L)$, and the group $\mathsf{Diff}_0(Y)$ is a subgroup of $\mathsf{Diff}_0(\mathcal L)$.  \smallskip

If $\b: (M, \d M) \to (\R \times Y, \R \times \d Y)$ is a proper immersion/embedding as in Definition \ref{quasi_isotopy_of_ immersions}, then, for any $h \in \mathsf{Diff}_0(\mathcal L)$, the immersion/embedding $h \circ \b$ is   $(d, d'; \mathbf c\Theta,  \mathbf c\Lambda; \mathbf c\Theta')$-quasitopic to $\b$. In a similar way, $\mathsf{Diff}_0(\mathcal L^\bullet)$ acts on quasitopies of immersion/embedding. Therefore, in what follows, we may ignore the dependence of our constructions on the isotopies of the base manifold $Y$. 
\smallskip

\begin{definition}\label{def.m_hat}
For a given closed poset $\Theta_{\langle d]} \subset \mathbf \Om_{\langle d]}$, we denote by $\hat m(\mathbf c\Theta_{\langle d]})$ the maximum of entries $\om_i$ for all $\om = (\om_1, \ldots , \om_i, \ldots , \om_\ell) \in \mathbf c\Theta_{\langle d]}$. \hfill $\diamondsuit$
\end{definition}

\begin{proposition}\label{prop.k-normal_imm_to_bord} We adopt the notations of Corollary \ref{cor.Sigma_k_is_bordism_invariant}. Any proper immersion $\b: (M, \d M) \to (\R \times Y,\, \R \times \d Y)$, as in Definition \ref{quasi_isotopy_of_ immersions}, which is $k$-normal for all $k$ in the interval $[2,\, \min\{n+1,\, \hat m(\mathbf c\Theta_{\langle d]})\}]$, generates canonically smooth maps $$\big\{\pi \circ \b \circ p_1: (\Sigma_k^\b,  \Sigma_k^{\b^\d}) \to (Y, \d Y)\big\}_{k},$$
where $\Sigma^\b_k$ is the $k^{th}$ self-intersection manifold of $\b$. The relative non-oriented bordism classes $[\Sigma_k^\b] \in \mathbf B_{n-k+1}(Y, \d Y)$ of these maps $\pi \circ \b \circ p_1$ are invariants of the $(d, d'; \mathbf c\Theta,  \mathbf c\Lambda; \mathbf c\Theta')$-quasitopy class of $\b$.  

If $M, Y$ are oriented, then oriented bordism classes $[\Sigma_k^\b] \in \mathbf {OB}_{n-k+1}(Y, \d Y)$ are invariants of the oriented $(d, d'; \mathbf c\Theta,  \mathbf c\Lambda; \mathbf c\Theta')$-quasitopy class of $\b$. In particular, if $Y$ is closed, then the Pontryagin numbers of the manifold $\Sigma_k^\b$ are invariants of the quasitopy class of $\b$. 
\end{proposition}

\begin{proof} If $\b$ is an immersion as in Definition \ref{quasi_isotopy_of_ immersions}, then its sufficiently small perturbation still has combinatorial tangency patterns which belong to $\mathbf c\Theta$, since $\Theta$ is a closed poset in $\mathbf\Om_{\langle d]}$. This claim is based on the behavior of real divisors of real polynomials under their perturbations \cite{K3}. Similarly, a sufficiently small perturbation of any cobordism $B$ between such immersions, still will have combinatorial tangency patterns which belong to $\mathbf c\Theta'$, since $\Theta'$ is a closed poset in $\mathbf\Om_{\langle d']}$. Therefore, by \cite{LS}, we may assume that $\b$, within its $(d, d'; \mathbf c\Theta,  \mathbf c\Lambda; \mathbf c\Theta')$-quasitopy class, is $k$-normal for all $k \leq n+1$ and $B$ is $k$-normal for all $k \leq n+ 2$. Revisiting Definition \ref{def.m_hat}, we notice that also $k \leq \hat m(\mathbf c\Theta_{\langle d]})$ by the very definition of quasitopies.  
Now the claim follows directly from Corollary \ref{cor.Sigma_k_is_bordism_invariant}.
\end{proof}

We are interested in two \emph{special cases} of Definition \ref{quasi_isotopy_of_ immersions} to be referred in what follows as {\sf the $\Lambda$-condition}: 
\begin{eqnarray}\label{special}
\mathbf{(1)} & \quad  d' \equiv d \equiv 0 \mod 2 \text{\; and \;} \mathbf{c}\Lambda = (\emptyset), \nonumber \\
\mathbf{(2)} & \quad  d' \equiv d \equiv 1 \mod 2 \text{\; and \;} \mathbf{c}\Lambda = (1). 
\end{eqnarray}

Case {\bf(1)} forces $\delta N = \emptyset$; so $M_0, M_1$ must be closed and $\d N = M_0\coprod M_1$. This is evidently the case when $\d Y = \emptyset$. Case {\bf(2)} forces $B|_{\delta N}$ to be transversal to the foliation $\mathcal L^\bullet |_{\R \times \d Y \times [0, 1]}$; so we get diffeomorphisms $B(\delta N) \approx \d Y \times [0, 1]$ and $\b_i(\d M_i) \approx \d Y$. 
\smallskip
\smallskip

Given two compact connected 
$n$-manifolds with boundaries, let $Y_1\# Y_2$ denote their {\sf connected sum}  and $Y_1\#_\d Y_2$ their {\sf boundary connected sum}. The $1$-handle $H \approx D^{n-1} \times [0,1]$ that participates in the connected sum operation is attached to $Y_1 \coprod Y_2$.  
In the special case $\dim Y_1 = \dim Y_2 = 1$, the manifolds $Y_1, Y_2$ are closed 
segments and $Y_1\#_\d Y_2$ is understood as a new segment, obtained by attaching one end of $Y_1$ to an end of $Y_2$.

If the boundaries of $Y_1$ and $Y_2$ are connected, the smooth topological type of $Y_1\#_\d Y_2$ does not depend on how the $1$-handle $H$ is attached to $Y_1\coprod Y_2$ to form $Y_1\#_\d Y_2$. In general, to avoid ambiguity of the operation $\#_\d$, we pick some elements $\kappa_1 \in \pi_0(\d Y_1)$ and $\kappa_2 \in\pi_0(\d Y_2)$.\smallskip

\begin{definition}\label{def.notations}
{\sf Under the $\Lambda$-condition (\ref{special}), we simplify our notations as follows:}\smallskip
\begin{eqnarray}\label{def.new def A}
\quad \quad  \mathsf{QT}^{\mathsf{i/e}}_{d, d'}(Y; \mathbf{c}\Theta; \mathbf{c}\Theta') =_{\mathsf{def}}\; \mathcal{QT}^{\mathsf{i/e}}_{d, d'}(Y, \d Y; \mathbf{c}\Theta, (\emptyset); \mathbf{c}\Theta'), \; \text{ for } d' \equiv d \equiv 0 \mod 2,
\end{eqnarray}
 
\begin{eqnarray}\label{def.new def B}
\quad \quad  \mathsf{QT}^{\mathsf{i/e}}_{d, d'}(Y; \mathbf{c}\Theta; \mathbf{c}\Theta') =_{\mathsf{def}}\; \mathcal{QT}^{\mathsf{i/e}}_{d, d'}(Y, \d Y; \mathbf{c}\Theta, (1); \mathbf{c}\Theta'), \; \text{ for } d' \equiv d \equiv 1 \mod 2,
\end{eqnarray}

 \begin{eqnarray}\label{def.new def_C}
\quad \quad  \mathsf G^{\mathsf{i/e}}_{d, d'}(n; \mathbf{c}\Theta; \mathbf{c}\Theta') =_{\mathsf{def}}\; \mathcal{QT}^{\mathsf{i/e}}_{d, d'}(D^n, \d D^n; \mathbf{c}\Theta, (\emptyset); \mathbf{c}\Theta'),  \text{ for } d' \equiv d \equiv 0 \mod 2,
\end{eqnarray}
 
\begin{eqnarray}\label{def.new def_D}
\quad \quad \quad \mathsf G^{\mathsf{i/e}}_{d, d'}(n; \mathbf{c}\Theta; \mathbf{c}\Theta') =_{\mathsf{def}}\; \mathcal{QT}^{\mathsf{i/e}}_{d, d'}(D^n, \d D^n; \mathbf{c}\Theta, (1); \mathbf{c}\Theta'),  \text{ for } d' \equiv d \equiv 1 \mod 2, 
\end{eqnarray}
where $D^n$ stands for the standard $n$-ball. \hfill $\diamondsuit$
\end{definition}

Assuming the $\Lambda$-condition (\ref{special}), for any the choice of $\kappa_1 \in \pi_0(\d Y_1)$ and $\kappa_2 \in\pi_0(\d Y_2)$, let us  introduce an operation
\begin{eqnarray}\label{eq_cup}
\quad \uplus: \mathsf{QT}^{\mathsf{i/e}}_{d, d'}(Y_1; \mathbf{c}\Theta; \mathbf{c}\Theta')  \times \mathsf{QT}^{\mathsf{i/e}}_{d, d'}(Y_2; \mathbf{c}\Theta; \mathbf{c}\Theta') \to  \mathsf{QT}^{{\mathsf{i/e}}}_{d, d'}(Y_1\#_\d Y_2; \mathbf{c}\Theta; \mathbf{c}\Theta').
\end{eqnarray}
Let $H \approx D^{n-1} \times D^1$ be a $1$-handle attached to $Y_1 \coprod Y_2$, to the preferred connected components $\kappa_1$ and $\kappa_2$  of $\d Y_1$ and $\d Y_2$. The result is the connected sum $Y_1\#_\d Y_2$. 
From now and on, we assume that the handles are attached so that the corners are smoothened and the resulting manifold has a smooth boundary. Different attachments of $H$ produce diffeomorphc connected sums, provided that the disks $D^{n-1} \times \{0\} \subset H$ and $D^{n-1} \times \{1\} \subset H$ are placed in the same connected components $(\d Y_1)_{\kappa_1}$ and $(\d Y_2)_{\kappa_2}$ of the boundaries $\d Y_1$ and $\d Y_2$. 
\smallskip

In case {\bf(1)} from (\ref{special}), $\mathbf{c}\Lambda = (\emptyset)$, thus forcing the two immersions $\b_i: M_i \to \R \times Y_i$ ($i=1,2$) to be immersions of \emph{closed}  manifolds.  With $(\d Y_1)_{\kappa_1}$ and $(\d Y_2)_{\kappa_2}$ being fixed, the map $\b_1 \coprod \b_2: M_1 \coprod M_2 \to \R \times (Y_1\#_\d Y_2)$ is a well-defined immersion within its quasitopy class. Note that by {\bf(1)} from (\ref{special}),  $(\b_1 \coprod \b_2)( M_1 \coprod M_2)$ is disjoint from $\R \times \d(Y_1\#_\d Y_2)$. Thus  we put $\b_1 \uplus \b_2 =_{\mathsf{def}} \b_1 \coprod \b_2$. 
\smallskip

In case {\bf(2)} from (\ref{special}), by an action of a diffeomorphism from $\mathsf{Diff}_0(\mathcal L_i)$, we insure that $\b_i(\d M_i) = \{0\} \times \d Y_i$, where $i = 1, 2$ and $\{0\} \in \R$. The $\mathsf{Diff}_0(\mathcal L_i)$-action does not change the quasitopy classes of $\{\b_i\}$. As we attach a $1$-handle $H$ to  $Y_1 \coprod Y_2$ to form $Y_1\#_\d Y_2$, we simultaneously attach the $1$-handle $\tilde H = \{0\} \times H$ to $M_1 \coprod M_2$ to form $M_1 \#_\d M_2$ and extend $\b_1 \coprod \b_2$ across $\tilde H$ to a new map $\b: M_1 \#_\d M_2 \to \R \times (Y_1\#_\d Y_2)$. On $\tilde H$, $\b$ is the obvious diffeomorphism $\tilde H \to H$.  If $M_1$ and $M_2$ are oriented, the handle $\tilde H$ is attached so that the preferred orientations extend across the handle, and then $H$ is attached so that the map $\tilde H \to H$ is a regular embedding. The transversality of $\b_i$ to $\R \times \d Y_i$, together with the property {\bf (2)} from (\ref{special}), allows us to smoothen the immersed manifold $\b(M_1 \#_\d M_2) = \b_1(M_1) \bigcup \tilde H \bigcup  \b_2(M_2)$ in the neighborhood of $\tilde H$. This smoothing construction again is well-defined within the relative quasitopy classes of $\b_1, \b_2$. So we put $\b_1 \uplus \b_2 =_{\mathsf{def}} \b$. \smallskip
 
To summarize, the operation $\uplus$ is well-defined for the elements of the sets $\mathcal{QT}^{\mathsf{i/e}}_d(Y, \d Y; \mathbf{c}\Theta, \hfill \break (\emptyset) ;  \mathbf{c}\Theta')$, where $d \equiv 0 \mod 2$, and for the elements of the sets $\mathcal{QT}^{\mathsf{i/e}}_d(Y, \d Y; \mathbf{c}\Theta, (1);  \mathbf{c}\Theta')$, where $d \equiv 1 \mod 2$.

In particular, by fixing a diffeomorphism  $\phi: D^n \#_\d D^n \approx D^n$, we get two interesting special cases of the operation $\uplus$:

\begin{eqnarray}\label{eq_D^n}
\uplus: \mathsf{G}^{\mathsf{i/e}}_{d, d'}(n; \mathbf{c}\Theta; \mathbf{c}\Theta') \times \mathsf{G}^{\mathsf{i/e}}_{d, d'}(n; \mathbf{c}\Theta; \mathbf{c}\Theta') 
 \to \mathsf{G}^{\mathsf{i/e}}_{d, d'}(n; \mathbf{c}\Theta; \mathbf{c}\Theta'), \\
 \uplus: \mathsf{OG}^{\mathsf{i/e}}_{d, d'}(n; \mathbf{c}\Theta; \mathbf{c}\Theta') \times \mathsf{OG}^{\mathsf{i/e}}_{d, d'}(n; \mathbf{c}\Theta; \mathbf{c}\Theta') 
 \to \mathsf{OG}^{\mathsf{i/e}}_{d, d'}(n; \mathbf{c}\Theta; \mathbf{c}\Theta'). \nonumber
\end{eqnarray}

\begin{proposition}\label{group} 
For any two closed subposets $\Theta' \subset  \Theta  \subset \mathbf\Om$, where $(\emptyset) \notin \Theta$ for $d \equiv 0 \mod 2$ and $(1) \notin \Theta$ for $d \equiv 1 \mod 2$, the operation $\uplus$ introduces a \emph{group structure} to the sets $\mathsf{G}^{\mathsf{i/e}}_{d, d'}(n; \mathbf{c}\Theta; \mathbf{c}\Theta')$ and $\mathsf{OG}^{\mathsf{i/e}}_{d, d'}(n; \mathbf{c}\Theta; \mathbf{c}\Theta')$ (see (\ref{def.new def_C}) and (\ref{def.new def_D})).
\smallskip
For $n > 1$, all these groups are abelian.
\end{proposition}

\begin{proof} In this proof, for a given map $\a: X \to X$ of a topological space $X$,  we {\sf denote} by $\tilde\a$ the product map $\mathsf{id}\times \a: \R \times X \to \R \times X$. 

The following arguments work equally well for $\mathsf {OG}^{\mathsf{i/e}}_{d, d'}(n; \mathbf{c}\Theta; \mathbf{c}\Theta')$ and $\mathsf G^{\mathsf{i/e}}_{d, d'}(n; \mathbf{c}\Theta; \mathbf{c}\Theta')$, the oriented and non-oriented quasitopies over the $n$-ball $D^n$.\smallskip  

Let $B^n = D^n \#_\d D^n$ be a $n$-ball, represented as a connected sum of two standard balls $D^n$. Let $H \approx D^{n-1} \times I$  be the $1$-handle that participates in the construction of the ball $B^n$ as a connected sum:  $B^n = D^n \cup H \cup D^n$. We fix a diffeomorphism $\chi: D^n \subset B^n$, with identifies the ball $D^n$ with the first ball in the connected sum $B^n$, and a diffeomorphism $\phi: B^n \approx D^n$.   Consider an isotopy $\{\psi_t: B^n \to B^n\}_{t \in [0, 1]}$ that starts with the identity map and terminates with the diffeomorphism $\psi_1$ whose image is $\chi(D^n)  \subset B^n$. With the help of $\phi^{-1}$, we transfer the isotopy $\psi_t$ to an isotopy $\{\eta_t: D^n \to D^n\}_{t \in [0, 1]}$. The isotopy $\psi_t$ lifts to the isotopy $\tilde\psi_t = \mathsf{id} \times \psi_t$ of $\R \times B^n$, and the isotopy $\eta_t$ lifts to the isotopy $\tilde\eta_t = \mathsf{id} \times \eta_t$ of $\R \times D^n$. \smallskip

For $d \equiv 0 \mod 2$, the neutral element $\b_\star$ is represented by $\b_\star: \emptyset \to \R \times D^n$ (by the ``empty" cylinder). For $d \equiv 1 \mod 2$, the neutral element $\b_\star$ is represented by the obvious embedding $\b_\star:  D^n \to \{0\} \times  D^n \subset \R \times  D^n$.  

We use $\tilde\eta_t$ to show that, for any immersion $\b: M \to \R \times D^n$ and $d \equiv 0 \mod 2$, the immersion $\tilde\phi \circ (\b \uplus \b_\star): M \#_\d \,\emptyset \to \R \times D^n$ is isotopic (and thus quasitopic) to $\b: M \to \R \times D^n$. 
 We identify $M \#_\d D^n$  with $M$ with the help of a diffeomorphism $\rho$. For $d \equiv 1 \mod 2$ and any immersion $\b: M \to \R \times D^n$, with the help of  $\tilde\eta_t$, the immersion $\tilde\phi \circ (\b \uplus \b_\star) \circ \rho^{-1} : M \to \R \times D^n$ is isotopic  to $\b$. This validates the existence of the neutral element $\b_\star$ for the operation $\uplus$ within the $(d, d; \mathbf c\Theta,  \mathbf c\Lambda; \mathbf c\Theta)$-quasitopy classes. Of course, any $(d, d; \mathbf c\Theta,  \mathbf c\Lambda; \mathbf c\Theta)$-quasitopy is automatically a $(d, d'; \mathbf c\Theta,  \mathbf c\Lambda; \mathbf c\Theta')$-quasitopy.
\smallskip

Consider an involution $\tau: B^n \to B^n$ that flips the two copies of $D^n$ in $B^n$ and has the ball $B^{n-1}$ in the middle of $1$-handle $H \approx B^{n-1} \times I$ as its fixed point set. 
Then, for any immersion $\b: M \to \R \times D^n$, the immersion $\tilde\phi \circ \tilde\tau \circ \tilde\chi \circ \b$ plays the role of the inverse $\b^{-1}$ with respect to $\uplus$.  Indeed, consider the unit half-disk $D^2_+$ and the unit half-cylinder $C^{n+1}_+ := D^2_+ \times B^{n-1}$, inscribed in the cylinder $B^n \times [0,\, 1+\e]$, $\e >0$. We will use the rotations $\tau(\theta)$ of $D^2 \times B^{n-1} \supset C^{n+1}_+$ around the axis $B^{n-1}$ at the angles $\theta \in [0, \pi]$, so that $\tau(\pi) = \tau$. 
Let $N = M \times [0, \pi]$. In the case $d \equiv 0 \mod 2$, we form an immersion $$A: N  \to \R \times C^{n+1}_+ \subset \R \times B^n \times [0, 1+\e],$$ defined by the formula $A(m, \theta) = \tilde\tau(\theta) \circ \b (m)$. 

Let $\{\underline 0\} \in \R$ denote the origin.
In the case $d \equiv 1 \mod 2$,  to satisfy Definition \ref{quasi_isotopy_of_ immersions} for $\mathbf c\Lambda = (1)$, we need to insure that
$$A(N) \cap \big\{\R \times \big(\d(B^n \times [0, 1+\e]) \setminus (B^n \times \{0\})\big)\big\} = \{\underline 0\} \times \big\{\d(B^n \times [0, 1+\e]) \setminus (B^n \times \{0\})\big\}.$$
Therefore, we further isotop $A$ \emph{radially} in each of the multipliers $D^2_+$ onto the rectangle $D^1 \times [0, 1+\e]$. The result is an isotopy of $A$ inside the cylinder $\{\underline 0\} \times B^n \times [0, 1+\e]$.  

Since $\{\tilde\tau(\theta)\}$ preserve the combinatorial $\mathcal L$-tangency patterns of $\b(M)$ within to the ambient foliation $\mathcal L^\bullet$ on $\R \times B^n \times [0, 1+\e]$,  we conclude that the cobordism $A: N \to \R \times B^n \times [0, 1+\e]$ delivers $(d, d; \mathbf c\Theta,  \mathbf c\Lambda; \mathbf c\Theta)$-quasitopy between $\b \uplus \b^{-1}$ and the neutral element $\b_\star$, provided that $\mathbf c\Lambda$ is as in the hypotheses of the proposition. \smallskip

Next, we need to verify the associativity of the operation $\uplus$. The argument is similar to the one that has validated that $\b \uplus \b_\star$ is $(d, d; \mathbf c\Theta,  \mathbf c\Lambda; \mathbf c\Theta)$-quasitopic to $\b$. It uses diffeomorphisms $\phi_{21}: D^n \#_\d D^n \to D^n$ and $\phi_{32}: D^n \#_\d D^n \#_\d D^n \to D^n \#_\d D^n$,  embeddings $\chi_1, \chi_2: D^n \to D^n \#_\d D^n$ on the first and the second ball, and the embeddings $\chi_{12}, \chi_{23}: D^n \#_\d D^n \to D^n \#_\d D^n \#_\d D^n$ on the first and the second ball and on the second and third ball in $D^n \#_\d D^n \#_\d D^n$, respectively. We leave the rest of the argument to the reader.
\smallskip

The validation of the fact that $\uplus$ is commutative for $n > 1$ is a bit more involved and similar to the classical proof of the fact that the homotopy groups $\pi_n(\sim)$ of spaces are commutative for $n > 1$ (see Fig. \ref{fig.commutativity}). Let us sketch this validation. 

\begin{figure}[ht]
\centerline{\includegraphics[height=3in,width=4.7in]{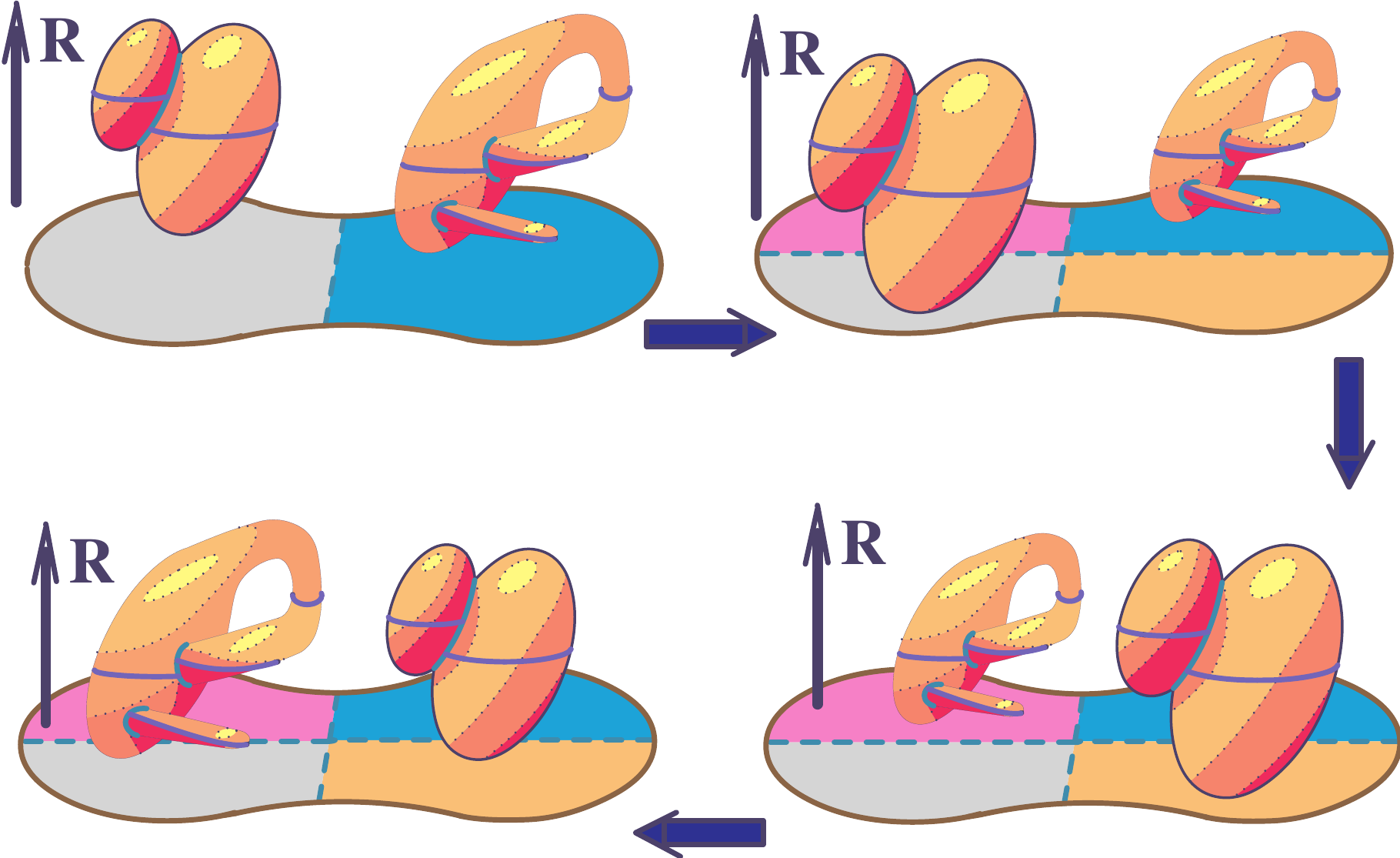}}
\bigskip
\caption{\small{Proving the commutativity of the operation $\uplus$ in the group $\mathsf{G}^{\mathsf{i/e}}_{d, d'}(n; \mathbf{c}\Theta; \mathbf{c}\Theta')$ for $n \geq 2$. The three bold arrows show the effects of isotopies in $\R \times B^n$.}} 
\label{fig.commutativity}
\end{figure}

As before, put $B^n = D^n_1 \cup H \cup D^n_2$.  We start with the case $d\equiv 0 \mod 2$ and two immersions $\b_i: M_i \to \R \times B^n$ ($i= 1,2$) such that the image of $\b_i$ belongs to the cylinder $\R \times D^n_i$. Consider an equatorial hyperplane $K^{n-1} \subset B^n$ (the horizontal punctured line in Fig. 3), transversal to the ball $B^{n-1} \subset H$ and such that $K^{n-1}$ is the fixed point set of another involution $\lambda$ on $B^n$. So the pair of $(n-1)$-balls $K^{n-1}$ and $B^{n-1}$ divide $B^n$ in the four regions (``quadrants"): $B^n_I, B^n_{II}, B^n_{III}, B^n_{IV}$, ordered counterclockwise around the axle $K^{n-1} \cap B^{n-1}$.  Each of these regions is homeomorphic to a $n$-ball. 

First, we choose an isotopy $\eta_{III}^t$ to compress $D^n_1 = B^n_{II}\cup B^n_{III}$ in the interior of the region $B^n_{III}$ and an isotopy $\eta_{I}^t$ to compress $D^n_2 = B^n_{I}\cup B^n_{IV}$ in the interior of the region $B^n_{I}$. We may assume that these isotopies are smooth in the interior of the relevant regions. We lift $\eta_{III}^t$ to an isotopy $\tilde\eta_{III}^t$ and compose it with $\b_1$.  We lift $\eta_{I}^t$ to an isotopy $\tilde\eta_{I}^t$ and compose it with $\b_2$. Let us denote these compositions by $\b_{1, III}$ and $\b_{2, I}$, respectively. Next, we choose an isotopy $\zeta^t_{II}$ to compress $B^n_I \cup B^n_{II}$ in the interior of $B^n_{II}$, and an isotopy $\zeta^t_{IV}$ to compress $B^n_{III} \cup B^n_{IV}$ in the interior of $B^n_{IV}$. The composition $\tilde\zeta^t_{II}\circ \b_{2,I}$ places the image of $M_2$ over the quadrant $B^n_{II}$ and composition $\tilde\zeta^t_{IV}\circ \b_{1,III}$ places the image $M_1$ over the quadrant $B^n_{IV}$. To complete the cycle that ``switches" $\b_1$ and $\b_2$, it remains to expand the image $\tilde\zeta^t_{II}\circ \b_{2,I}(M_2) \subset \R \times B^n_{II}$ into $\tilde\tau(M_2)$ and the image $\tilde\zeta^t_{IV}\circ \b_{1,III}(M_1) \subset \R \times B^n_{IV}$ into $\tilde\tau(M_1)$.\smallskip

In the case $d \equiv 1 \mod 2$, more careful arguments are needed to insure that $\b_1\uplus \b_2$ is $(d, d; \mathbf c\Theta,  (1); \mathbf c\Theta)$-quasitopic to $(\tilde\tau \circ \b_2) \uplus  (\tilde\tau \circ \b_1)$. They are similar to the ones we used to prove that $\b \uplus (\tilde\tau \circ \b)$ is $(d, d; \mathbf c\Theta,  (1); \mathbf c\Theta)$-quasitopic to the neutral element $\b_\star$. \end{proof}

\begin{remark} \emph{We stress that the ``$1$-dimensional" groups 
$\mathsf{G}^{\mathsf{i/e}}_{d, d'}(1; \mathbf{c}\Theta; \mathbf{c}\Theta')$ typically are \emph{not} commutative. There are quite a few examples (see Corollary \ref{cor.free_group} and Fig. 3) where, for $\Theta' = \Theta$ and $d' = d$, they are {\sf free groups} with the number of free generators being bounded from above by some quadratic functions in $d$. In general, these groups form an interesting class, whose presentations are well-understood \cite{KSW1} and whose group theoretical properties deserve a study.}
\hfill  $\diamondsuit$
\end{remark}

For a given group $\mathsf G$, we denote by $(\mathsf G)^r$ its $r$-fold direct product.

\begin{corollary}\label{cor_G^r-action} Let $\Theta' \subset  \Theta  \subset \mathbf\Om$ be two closed subposets, where $(\emptyset) \notin \Theta$ for $d \equiv 0 \mod 2$ and $(1) \notin \Theta$ for $d \equiv 1 \mod 2$. 

For any compact connected smooth $n$-dimensional manifold $Y$ with boundary that has $r$ connected components, the group $\big(\mathsf G^{\mathsf{i/e}}_{d, d'}(n; \mathbf{c}\Theta; \mathbf{c}\Theta')\big)^r$ acts, with the help of the operation $\uplus$ from (\ref{eq_cup}), on the set $\mathsf{QT}^{\mathsf{i/e}}_{d, d'}(Y; \mathbf{c}\Theta; \mathbf{c}\Theta')$.
Similarly,  $(\mathsf{OG}^{\mathsf{i/e}}_{d, d'}(n; \mathbf{c}\Theta; \mathbf{c}\Theta'))^r$ acts on $\mathsf{OQT}^{\mathsf{i/e}}_{d, d'}(Y; \mathbf{c}\Theta; \mathbf{c}\Theta')$.
\end{corollary}

\begin{proof} Let $\mathcal G := \mathsf G^{\mathsf{i/e}}_{d, d'}(n;\mathbf{c}\Theta; \mathbf{c}\Theta')$. By Proposition \ref{group}, $\mathcal G$ is a group with respect to the $\uplus$ operation. By the arguments that follow (\ref{eq_cup}), for each choice of a connected component $\d_\kappa Y$ of $\d Y$,  the group $\mathcal G$ acts via formula (\ref{eq_cup}) on the set $\mathsf{QT}^{\mathsf{i/e}}_{d, d'}(Y; \mathbf{c}\Theta; \mathbf{c}\Theta')$. In the case of an even $d$ and $\mathbf{c}\Lambda = (\emptyset)$, for a pair $\b: M \to \R \times Y$, and $\b': M' \to \R \times D^n$, the action is defined by the formula $\b \coprod \b': M \coprod M' \to \R \times (Y \#_{\d_\kappa} D^n) \approx \R \times Y$, where $\#_{\d_\kappa}$ is the boundary connected sum that employs the connected component $\d_\kappa Y$.  In the case of an odd $d$ and $\mathbf{c}\Lambda = (1)$, for a given pair $\b: M \to \R \times Y$ and $\b': M' \to \R \times D^n$, the action is defined by the formula $\b \#_{\d_\kappa} \b': M \#_{\d_\kappa^\dagger} M' \to \R \times (Y \#_{\d_\kappa} D^n)$, where $\#_{\d_\kappa^\dagger}$ is a boundary connected sum using the single component of $\d M$, embedded by $\b$ in $\R \times \d_\kappa Y$, and the single component of $\d M'$, embedded by $\b'$ in $\R \times \d D^n$. 

For different choices of elements  $\kappa \in \pi_0(\d Y)$, these $\mathcal G_\kappa$-actions  evidently commute. Thus $\mathcal G^r$ acts on $\mathsf{QT}^{\mathsf{i/e}}_{d, d'}(Y; \mathbf{c}\Theta; \mathbf{c}\Theta')$, where $r = \#(\pi_0(\d Y))$.
Similar arguments work for the oriented case.
 \end{proof}

\begin{remark}\label{rem.orbits}
\emph{We will see soon that the groups $\mathsf G^{\mathsf{i/e}}_{d, d'}(n; \mathbf{c}\Theta; \mathbf{c}\Theta')$ may have elements of infinite order, as well as some torsion, as complex as the torsion of the homotopy groups of spheres. As a result, the $\mathsf G^{\mathsf{i/e}}_{d, d'}(n; \mathbf{c}\Theta; \mathbf{c}\Theta')$-orbits in $\mathsf{QT}^{\mathsf{i/e}}_{d, d'}(Y; \mathbf{c}\Theta; \mathbf{c}\Theta')$ may have complex and diverse periods. We have a limited understanding of the orbit-space of this action; however, for embeddings, this problem is reduced to the homotopy theory of the spaces of real polynomials with constrained real divisors. In some cases (see Proposition \ref{prop.framed_bordisms} and Example \ref{ex.Sigma10}), we can estimate the ``size" of the orbit-space $\mathsf{QT}^{\mathsf{i/e}}_{d, d'}(Y; \mathbf{c}\Theta; \mathbf{c}\Theta')\big/ \mathsf G^{\mathsf{i/e}}_{d, d'}(n; \mathbf{c}\Theta; \mathbf{c}\Theta')$.} 

\hfill  $\diamondsuit$
\end{remark}


Given two pairs of spaces $X_1 \supset A_1$ and $X_2 \supset A_2$, we denote by $[(X_1,A_1), (X_2, A_2)]$ the set of homotopy classes of continuous maps $g: X_1 \to X_2$ such that $g(A_1) \subset A_2$.\smallskip

\begin{definition}\label{triples}
Given three pairs of spaces $X_1 \supset A_1$,  $X_2 \supset A_2$, and $X_3 \supset A_3$ and a fixed continuous map $\e: (X_2, A_2) \to (X_3, A_3)$, we denote by 
\begin{eqnarray}\label{homotopy_of_tripples}
[[(X_1, A_1),\, \e: (X_2, A_2) \to (X_3, A_3)]]
\end{eqnarray}
the set of homotopy classes $[g]$ of continuous maps $g: (X_1, A_1) \to (X_2, A_2)$, modulo the following equivalence relation: $[g_0] \sim [g_1]$, where $g_0: (X_1, A_1) \to (X_2, A_2)$ and $g_1: (X_1, A_1) \to (X_2, A_2)$, if the compositions $\e\circ g_0$ and $\e \circ g_1$ are homotopic as maps from $(X_1, A_1) $ to $(X_3, A_3)$. 
\hfill  $\diamondsuit$
\end{definition}

If all the spaces above are locally compact $CW$-complexes, then the set  in (\ref{homotopy_of_tripples}) is countable or finite.\smallskip

The next two propositions deliver new characteristic homotopy classes of immersions/embeddings against a fixed  background of $1$-foliations of the product type. 

\begin{proposition}\label{th.IMMERSIONS} Let $\Theta' \subset  \Theta \subset \Lambda \subset \mathbf \Om$ be closed subposets. Let $M$, $Y$ be smooth compact $n$-manifolds. 
\begin{itemize}
\item Any proper immersion/embedding $\b: (M, \d M) \to (\R \times Y,\, \R \times \d Y)$, such that:
\begin{enumerate}
 \item  $m^\b(y) \leq d$ (see (\ref{eq5.1})) and $m^\b(y) \equiv d \mod 2$ for all $y \in Y$, 
 \item the combinatorial patterns $\{\omega^\b(y)\}_{y \in Y}$ belong to $\mathbf c\Theta$, 
 \item  the combinatorial patterns $\{\omega^\b(y)\}_{y \in \d Y}$, belong to $\mathbf c\Lambda$,
 \end{enumerate}
 generates a continuous map $\Phi^\b_d: (Y, \d Y) \to (\cP_d^{\mathbf c\Theta}, \cP_d^{\mathbf c\Lambda})$, whose homotopy class $[\Phi^\b]$ depends on $\b$ only. \smallskip
  
 \item Moreover, any such pair $$\b_0: (M_0, \d M_0) \to (\R \times Y,\, \R \times \d Y) \text{\, and \,} \; \b_1: (M_1, \d M_1) \to (\R \times Y,\, \R \times \d Y)$$ of $(d, d'; \mathbf c\Theta, \mathbf c\Lambda; \mathbf c\Theta')$-quasitopic immersions/embeddings generates  
 homotopic composit maps  $$\e_{d, d'}\circ\Phi^{\b_0}: (Y, \d Y) \to (\cP_{d'}^{\mathbf c\Theta'}, \cP_{d'}^{\mathbf c\Lambda}) \text{ and  } \e_{d, d'}\circ\Phi^{\b_1}: (Y, \d Y) \to (\cP_{d'}^{\mathbf c\Theta'}, \cP_{d'}^{\mathbf c\Lambda}).$$ 
 
 Thus we get a well-defined map 
\begin{eqnarray}\label{Y_dY}
 \Phi_{d, d'}(Y, \d Y; \mathbf{c}\Theta, \mathbf c\Lambda; \mathbf{c}\Theta'):\; \mathcal{QT}^{\mathsf{i/e}}_{d, d'}(Y, \d Y; \mathbf{c}\Theta, \mathbf c\Lambda; \mathbf{c}\Theta') \to \nonumber \\ 
  \to [[(Y, \d Y),\; \e_{d, d'}: (\cP_d^{\mathbf c\Theta}, \cP_d^{\mathbf c\Lambda}) \to (\cP_{d'}^{\mathbf c\Theta'}, \cP_{d'}^{\mathbf c\Lambda})]],  
 \end{eqnarray}
 from the (oriented) quasitopies of immersions/embeddings to the homotopy classes of triples of spaces, as introduced in Definition \ref{triples}.\smallskip
\end{itemize}
\end{proposition}

\begin{proof} The two claims follow from Theorem \ref{th.LIFT},  being applied first to proper immersions $\b_0: M_0 \to \R \times Y$ and $\b_1:\b_1: M_1 \to \R \times Y$, and then to the proper immersion  $B: (N, \delta N) \to (\R \times Y \times [0, 1],\, \R \times \d Y \times [0, 1])$  of a relative cobordism $N$ between $M_0$ and $M_1 \cup \delta N$. There is only one technical wrinkle in this application: $N$ is a compact manifold \emph{with corners}  $\d M_0 \coprod \d M_1$; so Theorem \ref{th.LIFT} must be adjusted to such a cituation. Note that $B$ is transversal to both hypersurfaces, $\R \times Y \times \d[0, 1]$ and to $\R \times \d Y \times [0, 1]$. This fact allows to add a collar to $Y \times [0,1]$ and to extend the arguments of Theorem \ref{th.LIFT} to a new enlarged $(n+1)$-manifold $\hat Z$ whose interior contains $Y \times [0,1]$. 

Thus the map $\Phi^B: (Y \times [0,1],\, \d Y \times [0,1]) \to (\cP_{d'}^{\mathbf c\Theta'}, \cP_{d'}^{\mathbf c\Lambda})$ delivers a homotopy between the maps $\e_{d, d'}\circ\Phi^{\b_0}$ and $\e_{d, d'}\circ\Phi^{\b_1}$. Revisiting Definition \ref{triples}, this implies that (\ref{Y_dY}) is a well-defined map. In fact, similar map may be defined from the oriented quasitopies.

It will follow from formula (\ref{surjective_A}) below that the map in (\ref{Y_dY}) is \emph{surjective}. 
\end{proof}

\begin{corollary}\label{cor_cL_empty}
Under the hypotheses of Proposition \ref{th.IMMERSIONS} and assuming the $\Lambda$-condition (\ref{special}), 
we get a well-defined map 
 \begin{eqnarray}\label{Y_dY_A} 
 \Phi_{d, d'}(Y; \mathbf{c}\Theta; \mathbf{c}\Theta'):\; \mathsf{QT}^{\mathsf{i/e}}_{d, d'}(Y; \mathbf{c}\Theta; \mathbf{c}\Theta') \to \nonumber \\ 
  \to [[(Y, \d Y),\; \e_{d, d'}: (\cP_d^{\mathbf c\Theta}, pt) \to (\cP_{d'}^{\mathbf c\Theta'}, pt')]],  
 \end{eqnarray}
 where the base point $pt \in \cP_d^{\mathbf c\Lambda}$ and $pt' = \e_{d, d'}(pt)$. Similarly, a map from $\mathsf{OQT}^{\mathsf{i/e}}_{d, d'}(Y; \mathbf{c}\Theta; \mathbf{c}\Theta')$ to the same target is available.
 \end{corollary}
 
 \begin{proof} If $\mathbf c\Lambda = (\emptyset)$ or $\mathbf c\Lambda = (1)$, then the spaces $\cP_d^{\mathbf c\Lambda}$ and $\cP_{d'}^{\mathbf c\Lambda}$ are contractible. So the corollary follows from Proposition \ref{th.IMMERSIONS}.
 \end{proof}
 
Let $Y$ be a connected $n$-dimensional manifold  with boundary. Let us pick a preferred component $\d Y_i$ of $\d Y$. Then the group  $\pi_n(\cP_d^{\mathbf c\Theta}, pt)\big/ \ker((\e_{d, d'})_\ast)$ acts naturally on the set $[[(Y, \d Y),\; \e_{d, d'}]]$ by forming the connected sum $g \#_{\d Y_i} f$ of maps $f: (D^n, \d D^n) \to (\cP_d^{\mathbf c\Theta}, pt)$ and $g: (Y, \d Y) \to (\cP_d^{\mathbf c\Theta}, pt)$.  As a result, we get an action of $\big(\pi_n(\cP_d^{\mathbf c\Theta}, pt)\big/ \ker((\e_{d, d'})_\ast)\big)^r$ on $[[(Y, \d Y),\; \e_{d, d'}]]$, where $r = \#(\pi_0(\d Y))$. Let us denote this action by $h \#_{\d_1} f_1 \#_{\d_2} \ldots \#_{\d_r} f_r$, where $h \in [[(Y, \d Y),\; \e_{d, d'}]]$, $f_i \in \pi_n(\cP_d^{\mathbf c\Theta}, pt)$, and $\#_{\d_i} f_i$ stands for forming a boundary connected sum map $h \#_\d f_i$ with $h$ by using the $i^{th}$ component of $\d Y$. 
\smallskip

{$\bullet$ \sf Assuming the $\Lambda$-condition (\ref{special}), in what follows, we use the abbreviation}  
$$\Phi_{d, d'}(n; \mathbf{c}\Theta; \mathbf{c}\Theta')=_{\mathsf{def}} \Phi_{d, d'}(D^n, \d D^n; \mathbf{c}\Theta, \mathbf c\Lambda; \mathbf{c}\Theta').$$
Combining Proposition \ref{group} with Proposition \ref{th.IMMERSIONS}, leads to the following result.

\begin{proposition}\label{prop_(D, dD)} Under the $\Lambda$-condition (\ref{special}),
\begin{itemize}
\item the map in (\ref{Y_dY}) generates a group homomorphism
\begin{eqnarray}\label{D_dD}
\Phi_{d, d'}(n; \mathbf{c}\Theta 
; \mathbf{c}\Theta'):\;\; \mathsf G^{\mathsf{i/e}}_{d, d'}(n; \mathbf{c}\Theta
;  \mathbf{c}\Theta') \to \nonumber \\ 
\to \pi_n(\cP_d^{\mathbf c\Theta}, pt)\big/ \ker\big\{(\e_{d, d'})_\ast: \pi_n(\cP_d^{\mathbf c\Theta}, pt) \to \pi_n(\cP_{d'}^{\mathbf c\Theta'}, pt')\big\},
 \end{eqnarray}
where the group operation in $\mathsf G^{\mathsf{i/e}}_{d, d'}(n; \mathbf{c}\Theta
;  \mathbf{c}\Theta')$ is ``\,$\uplus$".\footnote{By Corollary \ref{D^n_A},  $\Phi_{d, d'}^{\mathsf{emb}}(n; \mathbf{c}\Theta
; \mathbf{c}\Theta')$ is an isomorphism, and $\Phi_{d, d'}^{\mathsf{imm}}(n; \mathbf{c}\Theta 
; \mathbf{c}\Theta')$ is an epimorphism.} 
\smallskip

For $d \equiv 0 \mod 2$, the base points $pt \in \cP_d^{\mathbf c\Theta}$ resides in the chamber $\mathsf R_d^{(\emptyset)}$ of positive monic polynomials, and for $d \equiv 1 \mod 2 $, in the chamber $\mathsf R_d^{(1)}$ of monic polynomials with a single simple real root. Similar choices of the base point are made for $pt' \in \cP_{d'}^{\mathbf c\Theta'}$.
\smallskip

\item Let $r = \#(\pi_0(\d Y))$. The $r^{th}$ power of the homomorphism $\Phi_{d, d'}(n; \mathbf{c}\Theta; \mathbf{c}\Theta')$ from (\ref{D_dD}) generates a representation of the  group $\big(\mathsf G^{\mathsf{i/e}}_{d, d'}(n;\mathbf{c}\Theta;  \mathbf{c}\Theta')\big)^r$  in the group $$\Big(\pi_n\big(\cP_d^{\mathbf c\Theta}, pt\big)\big/ \ker\big\{(\e_{d, d'})_\ast: \pi_n(\cP_d^{\mathbf c\Theta}, pt) \to \pi_n(\cP_{d'}^{\mathbf c\Theta'}, pt')\big\}\Big)^r.$$  
With the help of this representation, the map in (\ref{Y_dY}) is \emph{equivariant} with respect to the  $\big(\mathsf G^{\mathsf{i/e}}_{d, d'}(n; \mathbf{c}\Theta;  \mathbf{c}\Theta')\big)^r$-action on the set $\mathsf{QT}^{\mathsf{i/e}}_{d, d'}(Y; \mathbf{c}\Theta; \mathbf{c}\Theta')$ and with respect to the 
$\big(\pi_n(\cP_d^{\mathbf c\Theta}, pt)\big/ \ker\big\{(\e_{d, d'})_\ast: \pi_n(\cP_d^{\mathbf c\Theta}, pt) \to \pi_n(\cP_{d'}^{\mathbf c\Theta'}, pt')\big\}\big)^r \text {- action }$
 on the set $[[(Y, \d Y),\; \e_{d, d'}: (\cP_d^{\mathbf c\Theta}, \cP_d^{\mathbf c\Lambda}) \to (\cP_{d'}^{\mathbf c\Theta'}, \cP_{d'}^{\mathbf c\Lambda})]].$ 
 \smallskip
 
\item Similar claims hold for the oriented quasitopies $\mathsf{OQT}^{\mathsf{i/e}}_{d, d'}(Y; \mathbf{c}\Theta; \mathbf{c}\Theta')$ and the \hfill\break$\big(\mathsf{OG}^{\mathsf{i/e}}_{d, d'}(n; \mathbf{c}\Theta;  \mathbf{c}\Theta')\big)^r$-action on them.
\end{itemize}
\end{proposition}

\begin{proof} By Proposition \ref{th.IMMERSIONS}, we get a group representation  $\Phi_{d, d'}(n; \mathbf{c}\Theta, \mathbf c\Lambda; \mathbf{c}\Theta')$ from (\ref{D_dD}). By Corollary \ref{cor_G^r-action}, the group $\big(\mathsf G^{\mathsf{i/e}}_{d, d'}(n; \mathbf{c}\Theta; \mathbf{c}\Theta')\big)^r$ acts on 
$\mathsf{QT}^{\mathsf{i/e}}_{d, d'}(Y; \mathbf{c}\Theta; \mathbf{c}\Theta')$. Employing Corollary \ref{cor_cL_empty}, we get a well-defined map $\Phi(Y) =_{\mathsf{def}} \Phi_{d, d'}(Y, \d Y;  \mathbf{c}\Theta, \mathbf c\Lambda; \mathbf{c}\Theta')$ from (\ref{Y_dY}). 

If a map $F: (Y, \d Y) \to (\cP_d^{\mathbf c\Theta}, pt)$, being composed with $\e_{d, d'}$, becomes null-homotopic in $(\cP_d^{\mathbf c\Theta'}, pt')$, then for any map $G: (D^n, \d D^n) \to (\cP_d^{\mathbf c\Theta}, pt)$ such that $\e_{d, d'} \circ G$ is null-homotopic in $(\cP_d^{\mathbf c\Theta'}, pt')$, the map $\e_{d, d'} \circ (G \#_\d F)$ is null-homotopic in $(\cP_d^{\mathbf c\Theta'}, pt')$ as well.  Thus the action of the group $\pi_n((\cP_d^{\mathbf c\Theta}, pt))/ \ker \{\e_{d, d'}\}$ on the set $[[(Y, \d Y),\; \e_{d, d'}: (\cP_d^{\mathbf c\Theta}, \cP_d^{\mathbf c\Lambda}) \to (\cP_{d'}^{\mathbf c\Theta'}, \cP_{d'}^{\mathbf c\Lambda})]]$ is well-defined.
Now, thanks to Theorem \ref{th.LIFT}, it is on the level of definitions to verify the equivariant property of $\Phi(Y)$, namely, that
 $$\Phi(Y)(\b \uplus_1 \g_1 \uplus_2 \g_2 \ldots   \uplus_r \g_r) = \Phi(Y)(\b)\, \#_{\d_1} \Phi(D^n)(\g_1) \ldots  \#_{\d_r}\, \Phi(D^n)(\g_r),$$
 where $\b \in \mathsf{QT}^{\mathsf{i/e}}_{d, d'}(Y; \mathbf{c}\Theta; \mathbf{c}\Theta')$ and $\g_i \in \mathsf G^{\mathsf{i/e}}_{d, d'}(n; \mathbf{c}\Theta;  \mathbf{c}\Theta')$, or $\b \in \mathsf{OQT}^{\mathsf{i/e}}_{d, d'}(Y; \mathbf{c}\Theta; \mathbf{c}\Theta')$ and $\g_i \in \mathsf{OG}^{\mathsf{i/e}}_{d, d'}(n; \mathbf{c}\Theta;  \mathbf{c}\Theta')$.
\end{proof}

\begin{definition}\label{reduction} 
Let $\Theta' \subset  \Theta \subset \Lambda \subset \bfOm$ be closed subposets.

Consider the obvious map $\mathcal{QT}^{\mathsf{i/e}}_{d, d'}(Y, \d Y; \mathbf{c}\Lambda, \mathbf{c}\Lambda; \mathbf{c}\Theta') \to \mathcal{QT}^{\mathsf{i/e}}_{d, d'}(Y, \d Y; \mathbf{c}\Theta, \mathbf{c}\Lambda; \mathbf{c}\Theta')$. 

\noindent We denote by $^\dagger\mathcal{QT}^{\mathsf{i/e}}_{d, d'}(Y, \d Y; \mathbf{c}\Theta, \mathbf{c}\Lambda; \mathbf{c}\Theta')$ its image. It is represented by immersions/ embeddings $\b: (M, \d M) \to (\R \times Y, \R \times \d Y)$, whose combinatorial tangency patterns to the foliation $\mathcal L$ belong to the open subposet $\mathbf{c}\Lambda \subset \mathbf{c}\Theta$. We call such elements {\sf combinatorially trivial}. 

 Then we introduce the quotient set by the formula:
\begin{eqnarray}
\qquad \mathbf{QT}^{\mathsf{i/e}}_{d, d'}(Y, \d Y; \mathbf{c}\Theta, \mathbf{c}\Lambda; \mathbf{c}\Theta') =_{\mathsf{def}} \qquad\qquad\qquad\qquad \\
 \mathcal{QT}^{\mathsf{i/e}}_{d, d'}(Y, \d Y; \mathbf{c}\Theta, \mathbf{c}\Lambda; \mathbf{c}\Theta')\big/\; ^\dagger\mathcal{QT}^{\mathsf{i/e}}_{d, d'}(Y, \d Y; \mathbf{c}\Lambda, \mathbf{c}\Lambda; \mathbf{c}\Theta').  \qquad  \qquad \diamondsuit \nonumber 
\end{eqnarray}
\end{definition}

Under the $\Lambda$-condition (\ref{special}), we get
\begin{eqnarray}\label{eq_dagger}
\mathbf{QT}^{\mathsf{i/e}}_{d, d'}(Y, \d Y; \mathbf{c}\Theta, \mathbf{c}\Lambda; \mathbf{c}\Theta') = \\ 
= \mathcal{QT}^{\mathsf{i/e}}_{d, d'}(Y, \d Y; \mathbf{c}\Theta, \mathbf{c}\Lambda; \mathbf{c}\Theta') =_{\mathsf{def}} \mathsf{QT}^{\mathsf{i/e}}_{d, d'}(Y; \mathbf{c}\Theta; \mathbf{c}\Theta'), \nonumber
\end{eqnarray}
since $^\dagger\mathcal{QT}^{\mathsf{i/e}}_{d, d'}(Y, \d Y; \mathbf{c}\Lambda, \mathbf{c}\Lambda; \mathbf{c}\Theta')$ is represented by a single element, an empty embedding or the embedding $\mathsf{id}: Y \to \{0\} \times Y \subset \R \times Y$. In particular, if $\d Y = \emptyset$, then we get the bijection (\ref{eq_dagger}). When $Y$ is \emph{orientable},  similar  definitions/constructions of $\mathbf{OQT}^{\mathsf{i/e}}_{d, d'}(Y, \d Y; \mathbf{c}\Theta, \mathbf{c}\Lambda; \mathbf{c}\Theta')$ are available for the oriented quasitopies.
 
\begin{lemma}\label{lem.E-reg} 
Let $\Theta' \subset  \Theta \subset \Lambda \subset \bfOm$  be closed subposets,  and let  $Y$ be a compact manifold. We pick a pair of natural numbers  $d \leq d'$ of the same parity.
\begin{itemize} 
\item 
\noindent Any smooth $(\d \mathcal E_d)$-regular map $h: (Y, \d Y) \to (\cP_d^{\mathbf c\Theta}, \cP_d^{\mathbf c\Lambda})$  produces a proper regular \emph{embedding} $\b_h: (M, \d M) \hookrightarrow (\R\times Y,\, \R \times \d Y)$ of some compact $n$-manifold $M$. The combinatorial patterns of $\b_h$, relative to the foliation $\mathcal L$, belong to $\mathbf c\Theta$ and of its restriction to $\d M$ belong to $\mathbf c\Lambda$. Moreover, $\Phi^{\b_h} = h$. \smallskip

\item Any pair $h_0, h_1: (Y, \d Y) \to (\cP_d^{\mathbf c\Theta}, \cP_d^{\mathbf c\Lambda})$ of 
$(\d\mathcal E_d)$-regular  maps such that $\e_{d,d'} \circ h_0$ and $\e_{d,d'} \circ h_1$ are linked by a  smooth $(\d\mathcal E_{d'})$-regular homotopy  $$H: (Y \times [0,1], \d Y \times [0,1]) \to (\cP_{d'}^{\mathbf c\Theta'}, \cP_{d'}^{\mathbf c\Lambda})$$ gives rise to a $(d, d'; \mathbf c\Theta, \mathbf c\Lambda; \mathbf c\Theta')$-quasitopy $B_H: N \hookrightarrow \R\times Y\times [0, 1]$ between $\b_{h_0}$ and $\b_{h_1}$. Moreover, $\Phi^{B_H} = H$. 
\item If $Y$ is orientable, then so are $M$ and $N$.
\end{itemize}
\end{lemma}

\begin{proof}  The first claim of the lemma follows from Definition \ref{def.E-regular}, since for a $\d \mathcal E_d$-regular map, the preimage of the hypersurface $\d\mathcal E_d \subset \R \times \mathcal P_d$ under $\mathsf{id}\times h$ is a regularly embedded codimension $1$ submanifold $M$ of $\R \times Y$, and the preimage of $\d\mathcal E_d$ under the map $\mathsf{id} \times h|_{\d Y}$, where $h|_{\d Y}$ is a $\d \mathcal E_d$-regular map, is the boundary $\d M$, regularly embedded in $\R \times \d Y$. Recall that the map $\mathsf{id} \times h$ pulls back the universal function $\wp(u, x): \R \times \R^d \to \R$ to the function $(\mathsf{id} \times h)^\ast(\wp)$ on $\R \times Y$ for which $0$ is a regular value. 
Using this function $(\mathsf{id} \times h)^\ast(\wp)$,  the arguments from Theorem \ref{th.LIFT} imply that $\Phi^{\b_h} = h$.  \smallskip

Let us validate the second claim.  Recall that $\e_{d,d'}(\cP_d^{\mathbf c\Lambda}) \subset \cP_{d'}^{\mathbf c\Lambda}$, since $\e_{d,d'}$ preserves the combinatorial types $\om$ of real divisors of real polynomials.

By definition of $\e_{d, d'}$, we get $\e_{d, d'}^\ast(\wp_{d'}) = (u^2+1)^{(d'-d)/2} \cdot  \wp_{d}$.   
Thus $$\big(\mathsf{id} \times (\e_{d, d'}\circ h)\big)^\ast(\wp_{d'}) = (\mathsf{id} \times h)^\ast(\wp_d) \cdot Q,$$ where $Q: \R \times Y \to \R_+$ is a smooth positive function.  So the zero loci  $$\{\big(\mathsf{id} \times (\e_{d, d'}\circ h)\big)^\ast(\wp_{d'}) = 0\} \text{ and } \{\big(\mathsf{id} \times h\big)^\ast(\wp_{d}) = 0\}$$ coincide. Moreover, using that $Q > 0$,  if the $1$-jet of $\big(\mathsf{id} \times h\big)^\ast(\wp_{d})$ does not vanish along $\{\big(\mathsf{id} \times h\big)^\ast(\wp_{d}) = 0\}$, then the $1$-jet of $(\mathsf{id} \times h)^\ast(\wp_d) \cdot Q$ does not vanish along the shared locus. 
Therefore, if $(\mathsf{id} \times h)^\ast(\wp_d)$ has $0$ is a regular value, then 
$\big(\mathsf{id} \times (\e_{d, d'}\circ h)\big)^\ast(\wp_{d'})$ has $0$ as a regular value as well. 
\smallskip

Now, given two $\d \mathcal E_d$-regular maps $h_0, h_1: (Y, \d Y) \to (\cP_d^{\mathbf c\Theta}, \cP_d^{\mathbf c\Lambda})$ such that $\e_{d,d'} \circ h_0$ and $\e_{d,d'} \circ h_1$ are linked by a  smooth $\d \mathcal E_{d'}$-regular homotopy  $$H: (Y \times [0,1], \d Y \times [0,1]) \to (\cP_{d'}^{\mathbf c\Theta'}, \cP_{d'}^{\mathbf c\Lambda}),$$
we conclude that  the preimage of $\d\mathcal E_{d'}$ under $\mathsf{id}\times H$ is a regularly embedded codimension $1$ submanifold $N \subset \R \times Y \times [0, 1]$ with corners $\d M_0 \coprod \d M_1$. Its boundary $\d N = (\d M_0 \coprod \d M_1) \bigcup \delta N$, where $\delta N = H^{-1}(\d \mathcal E_{d'}) \cap (\R \times \d Y \times [0,1]).$ 
Moreover, for $i = 0, 1$, the transversality of $\mathsf{id}\times H$ to $\d \mathcal E_{d'}$ implies the transversality of $$\b_i: M_i = (\e_{d,d'} \circ h_i)^{-1}(\d \mathcal E_{d'}) \hookrightarrow \R \times Y$$ to $\R \times \d Y$, and the transversality of $B: N = H^{-1}(\d \mathcal E_{d'}) \hookrightarrow \R \times Y \times [0,1]$ to $\R \times \d Y  \times [0,1]$. 
Using the pull-back $(\mathsf{id}\times H)^\ast(\wp_{d'})$, again by the arguments from Theorem \ref{th.LIFT}, we get that $\Phi^{B_H} = H$.

Finally, if $Y$ is oriented, then so are $\R \times Y$ and $\R \times Y \times [0, 1]$. Then the pull-backs under $\mathsf{id}\times h$ or $\mathsf{id}\times H$ of the normal fields  $n(\d\mathcal E_d, \mathcal E_d)$ or $n(\d\mathcal E_{d'}, \mathcal E_{d'})$ help to orient $M$ or $N$.   
\end{proof}

Lemma \ref{lem.Ek-reg} below is a natural generalization of Lemma \ref{lem.E-reg} from the case $k=1$ to the case of a general $k$.  It is based on Lemma \ref{lem_kE-regular}. We skip the proof of Lemma \ref{lem.Ek-reg}, since we will not rely on it. However, its formulation may give a bit wider perspective of our effort.

\begin{lemma}\label{lem.Ek-reg}
Let $\Theta' \subset  \Theta \subset \Lambda \subset \bfOm$  be closed sub-posets, and $Y$ a compact manifold. Let $k \leq d$, $k' \leq d'$, $d \leq d'$, $k \leq k'$, and $d \equiv d' \mod 2$.
\begin{itemize} 

\item Any smooth $(\d \mathcal E_d, k)$-regular map $h: (Y, \d Y) \to (\cP_d^{\mathbf c\Theta}, \cP_d^{\mathbf c\Lambda})$  produces a proper immersion $\b_h: (M, \d M) \hookrightarrow (\R\times Y,\, \R \times \d Y)$ of some compact $n$-manifold $M$. The combinatorial patterns of $\b_h$, relative to the foliation $\mathcal L$, belong to $\mathbf c\Theta$ and of its restriction to $\d M$ to $\mathbf c\Lambda$. Moreover, $\Phi^{\b_h} = h$. \smallskip

\item Any pair $h_0, h_1: (Y, \d Y) \to (\cP_d^{\mathbf c\Theta}, \cP_d^{\mathbf c\Lambda})$ of 
$(\d \mathcal E_d, k)$-regular  maps such that $\e_{d,d'} \circ h_0$ and $\e_{d,d'} \circ h_1$ are linked by a  smooth $(\d \mathcal E_{d'}, k')$-regular homotopy  $$H: (Y \times [0,1],\, \d Y \times [0,1]) \to (\cP_{d'}^{\mathbf c\Theta'},\, \cP_{d'}^{\mathbf c\Lambda}),$$ where $k \leq k' \leq d'$, gives rise to a $(d, d'; \mathbf c\Theta, \mathbf c\Lambda; \mathbf c\Theta')$-quasitopy $B_H: N \to \R\times Y\times [0, 1]$ between the immersions $\b_{h_0}$ and $\b_{h_1}$. Moreover, $\Phi^{B_H} = H$.  \hfill $\diamondsuit$
\end{itemize}
\end{lemma}

The next theorem provides homotopy-theoretical invariants of the quasitopy classes of immersions and embeddings into $\R \times Y$. For embeddings, they compute the quasitopies.

\begin{theorem}\label{th.E-reg}
Let $\Theta' \subset \Theta \subset  \Lambda \subset \mathbf\Om$  be closed subposets, and $Y$ be a smooth compact manifold. Assume that $d' \geq d$ and $d' \equiv d \mod 2$.\smallskip

\noindent $\bullet$ Then there is a canonicalal bijection 
\begin{eqnarray}\label{bold_bijection}
\mathbf{\Phi}^{\mathsf{emb}}: \mathbf{QT}^{\mathsf{emb}}_{d, d'}(Y, \d Y; \mathbf{c}\Theta, \mathbf{c}\Lambda; \mathbf{c}\Theta') \stackrel{\approx}{\longrightarrow} \nonumber \\
\stackrel{\approx}{\longrightarrow} \big[\big[(Y, \d Y),\; \e_{d, d'}: (\cP_d^{\mathbf c\Theta}, \cP_d^{\mathbf c\Lambda}) \to (\cP_{d'}^{\mathbf c\Theta'}, \cP_{d'}^{\mathbf c\Lambda})\big]\big].
\end{eqnarray}
$\bullet$ Assuming the $\Lambda$-condition (\ref{special}), we get a bijection
\begin{eqnarray}\label{emb_homotopy_A}
\qquad\qquad \Phi^{\mathsf{emb}}: \mathsf{QT}^{\mathsf{emb}}_{d, d'}(Y; \mathbf{c}\Theta; \mathbf{c}\Theta') \stackrel{\approx}{\longrightarrow} 
 \big[\big[(Y, \d Y),\; \e_{d, d'}: (\cP_d^{\mathbf c\Theta}, \cP_d^{\mathbf c\Lambda}) \to (\cP_{d'}^{\mathbf c\Theta'}, \cP_{d'}^{\mathbf c\Lambda})\big]\big]
\end{eqnarray}
and a surjection 
\begin{eqnarray}\label{surjective_A}
\qquad\qquad \Phi^{\mathsf{imm}}: \mathsf{QT}^{\mathsf{imm}}_{d, d'}(Y; \mathbf{c}\Theta; \mathbf{c}\Theta') \longrightarrow 
 \big[\big[(Y, \d Y),\; \e_{d, d'}: (\cP_d^{\mathbf c\Theta}, \cP_d^{\mathbf c\Lambda}) \to (\cP_{d'}^{\mathbf c\Theta'}, \cP_{d'}^{\mathbf c\Lambda})\big]\big].
\end{eqnarray}
By Proposition \ref{prop.right_inverse} below, $\Phi^{\mathsf{imm}}$ admits a right inverse.
\smallskip

\noindent $\bullet$ If $Y$ is orientable, the similar claims are valid for oriented quasitopies.
\end{theorem}

\begin{proof}  By Proposition \ref{th.IMMERSIONS}, the map $\Phi^{\mathsf{emb}}$ from (\ref{emb_homotopy_A}) is well-defined.
\smallskip

First we show that the map $\Phi^{\mathsf{emb}}$  is surjective.
 Recall that $\cP_d^{\mathbf c\Lambda}\subset \cP_d^{\mathbf c\Theta}$ both are open subspaces of $\cP_d \approx \R^d$, and $\cP_{d'}^{\mathbf c\Lambda}\subset \cP_{d'}^{\mathbf c\Theta}$ are open subspaces of $\cP_{d'} \approx \R^{d'}$. Thus any continuous map $h: (Y, \d Y) \to (\cP_d^{\mathbf c\Theta}, \cP_d^{\mathbf c\Lambda})$ may be approximated by a smooth map $\tilde h$  in the relative homotopy class of $h$. Similarly, any continuous map $H: (Y \times [0, 1], \d Y\times [0, 1]) \to (\cP_{d'}^{\mathbf c\Theta'}, \cP_{d'}^{\mathbf c\Lambda})$ may be approximated by a smooth map $\tilde H$ which shares with $H$ its relative homotopy class. 
 
By Corollary \ref{cor.E-reg}, $\tilde h$ may be approximated further by a $(\d\mathcal E_d)$-regular map $h^\dagger$ which is in the relative homotopy class of $h$. By Lemma \ref{lem.E-reg}, $h^\dagger$ is realized by a proper imbedding of some compact $n$-manifold $\b: (M, \d M) \subset (\R \times Y,\, \R \times \d Y)$ whose tangency to $\mathcal L$ patterns  belong to $\mathbf c\Theta$, while the tangency patterns of $\b(\d M)$ belong to $\mathbf c\Lambda$. Since the target set of the map 
 $\Phi^{\mathsf{emb}}$ is a quotient of the set $[(Y, \d Y), (\cP_d^{\mathbf c\Theta}, \cP_d^{\mathbf c\Lambda})]$ by an equivalence relation, we conclude that $\Phi^{\mathsf{emb}}$ is surjective.\smallskip
 
Now we show that the map $\mathbf\Phi^{\mathsf{emb}}$ from (\ref{bold_bijection}) is well-defined and injective.  Consider some proper $(\d\mathcal E_d)$-regular embedding $\b_0: (M, \d M) \subset  (\R \times Y, \R \times \d Y)$ such that the map $\psi_0 =_{\mathsf{def}} \e_{d, d'} \circ \Phi^{\mathsf{emb}}(\b_0)$ is homotopic to a map $\psi_1: (Y, \d Y) \to (\cP_{d'}^{\mathbf c\Theta'}, \cP_{d'}^{\mathbf c\Lambda})$ whose image is contained in $\cP_{d'}^{\mathbf c\Lambda}$. Again, by Corollary \ref{cor.E-reg}, this homotopy $H: (Y, \d Y) \times [0, 1] \to (\cP_{d'}^{\mathbf c\Theta'}, \cP_{d'}^{\mathbf c\Lambda})$ can be approximated by a $(\d\mathcal E_{d'})$-regular homotopy $H^\dagger$ which coincides with $\psi_0$ on $Y \times \{0\}$. Let $N = (\mathsf{id}\times H^\dagger)^{-1}(\d \mathcal E_{d'})$. By Lemma \ref{lem.E-reg},  $H^\dagger$ is produced by a cobordism $$B: (N, \delta N) \subset (\R \times Y \times [0, 1], \R \times \d Y \times [0, 1])$$ whose other end $\b_1: (M, \d M) \subset  (\R \times Y \times \{1\}, \R \times \d Y \times \{1\})$ has tangency patters (to the fibration $\mathcal L^\bullet$) that belong $\mathbf c\Lambda$. In other words, $\b_0 \in \, ^\dagger\mathcal{QT}^{\mathsf{emb}}_{d, d'}(Y, \d Y; \mathbf{c}\Lambda, \mathbf{c}\Lambda; \mathbf{c}\Theta')$, which represents the trivial element of the quotient $\mathbf{QT}^{\mathsf{emb}}_{d, d'}(Y, \d Y; \mathbf{c}\Theta, \mathbf{c}\Lambda; \mathbf{c}\Theta')$. Thus we get the bijection (\ref{bold_bijection}). 

When $\mathbf c\Lambda= (\emptyset) \text{ or } (1)$, the set $^\dagger\mathcal{QT}^{\mathsf{emb}}_{d, d'}(Y, \d Y; \mathbf{c}\Lambda, \mathbf{c}\Lambda; \mathbf{c}\Theta')$ consists of a single element. Hence we get the bijection in (\ref{emb_homotopy_A}). 

Finally, by Proposition \ref{th.IMMERSIONS}, the map $\Phi^{\mathsf{imm}}$ in (\ref{surjective_A}) is surjective, since the map $\Phi^{\mathsf{emb}}$ factors through it.

Assuming that $Y$ is oriented, the preimages  of the loci $\d\mathcal E_{d}$ and $\d\mathcal E_{d'}$ under regular maps are oriented manifolds. This remark validates the last claim.
 \end{proof}

\begin{corollary}\label{cor.homotopy_type} Under the hypotheses of Theorem \ref{th.E-reg}, including the  $\Lambda$-condition (\ref{special}), 
the quasitopy set $\mathsf{QT}^{\mathsf{emb}}_{d, d'}(Y, \d Y; \mathbf{c}\Theta,  \mathbf{c}\Theta')$ depends only on the \emph{homotopy type} of the pair $(Y, \d Y)$. 
 \smallskip

If $Y$ is orientable, then the orientation forgetting map $$\mathsf{OQT}^{\mathsf{emb}}_{d, d'}(Y, \d Y; \mathbf{c}\Theta,  \mathbf{c}\Theta') \approx \mathsf{QT}^{\mathsf{emb}}_{d, d'}(Y, \d Y; \mathbf{c}\Theta,  \mathbf{c}\Theta')$$ is a bijection. 
In particular, 
$\mathsf{OG}^{\mathsf{emb}}_{d, d'}(n; \mathbf{c}\Theta;  \mathbf{c}\Theta') \approx \mathsf G^{\mathsf{emb}}_{d, d'}(n; \mathbf{c}\Theta;  \mathbf{c}\Theta')$ is a group isomorphism. 
\end{corollary}

\begin{proof} The first claim of the corollary follows instantly from formula (\ref{emb_homotopy_A}) and its analogue for the oriented quasitopies.  \smallskip

The rest of the claims can be derived from the following already familiar observation. Since any element $\b \in  \mathsf{QT}^{\mathsf{emb}}_{d, d'}(Y; \mathbf{c}\Theta,  \mathbf{c}\Theta')$ is represented by a map $\mathsf{id}\times \Phi^\b : \R \times (Y, \d Y) \to \R \times (\cP_d^{\mathbf c\Theta}, pt)$ which is transversal to $\d \mathcal E_d$, the manifold $M := (\mathsf{id}\times \Phi^\b)^{-1}(\d \mathcal E_d)$ has a trivial normal bundle in $\R \times Y$. Thus $M$ is orientable, when $Y$ is. The same argument applies to the cobordisms that connect quasitopic $\b$'s. Therefore, when $Y$ is orientable, the orientation forgetful map from the second claim is onto. For a similar reason, applied to cobordisms, the orientation forgetful map is injective. 
\end{proof}

\begin{remark}
\emph{
The first claim of Corollary \ref{cor.homotopy_type} is a bit surprising, since not any homotopy equivalence $h: (Y, \d Y) \to (Y', \d Y')$ of smooth $n$-dimensional pairs may be homotoped to a diffeomorphism. Thus it is not clear how to transfer directly a regular embedding $\b: (M, \d M) \to (\R \times Y, \R \times \d Y)$ to a regular  embedding $\b': (M', \d M') \to (\R \times Y', \R \times \d Y')$ without using the bijection $\Phi^{\mathsf{emb}}$. Similarly, it is unclear whether $\mathsf{QT}^{\mathsf{imm}}_{d, d'}(Y; \mathbf{c}\Theta,  \mathbf{c}\Theta')$ depends only on the homotopy type of the pair $(Y, \d Y)$. Conjecturally, it does.\hfill $\diamondsuit$ 
}
\end{remark}

\begin{corollary}\label{D^n_A}
Assume that $d' \geq d$,  $d' \equiv d \mod 2$, and the $\Lambda$-condition (\ref{special})  is in place.

Then the group homomorphism in (\ref{D_dD}) is a \emph{group isomorphism} for embeddings and a \emph{split epimorphism} for the immersions. \smallskip

 With the help of this group isomorphism (and by Theorem \ref{prop_(D, dD)}), the map in (\ref{emb_homotopy_A}) is equivariant, and so is the map in (\ref{surjective_A}) with the help of the epimorhism of the corresponding groups.
\end{corollary}

\begin{proof} By Theorem \ref{th.E-reg}, the map in (\ref{emb_homotopy_A}) is bijective and the map in (\ref{surjective_A}) is split surjective for any $Y$; in particular, this is the case for $Y = D^n$. By Theorem \ref{prop_(D, dD)}, all these maps are group homomorphisms. By    Proposition \ref{prop.right_inverse} below, the map in (\ref{surjective_A}) is a split epimorphism.
\end{proof}

Letting $d' = d$ and $\Theta' = \Theta$ in Theorem \ref{th.E-reg}, we obtain the following two corollaries.

\begin{corollary}\label{cor.E-regA}
Let $\Theta \subset  \Lambda \subset \mathbf\Om$  be a pair of closed subposets, and $Y$ be a compact manifold. 

\noindent $\bullet$ Then there is a canonical bijection
\begin{eqnarray}
\mathbf{\Phi}^{\mathsf{emb}}: \mathbf{QT}^{\mathsf{emb}}_{d, d}(Y, \d Y; \mathbf{c}\Theta, \mathbf{c}\Lambda; \mathbf{c}\Theta) \stackrel{\approx}{\longrightarrow} [(Y, \d Y), (\cP_d^{\mathbf c\Theta}, \cP_d^{\mathbf c\Lambda})].
\end{eqnarray}
$\bullet$ Assuming the $\Lambda$-condition (\ref{special}),  
we get a bijection
\begin{eqnarray}\label{emb_homotopy}
\Phi^{\mathsf{emb}}: \mathsf{QT}^{\mathsf{emb}}_{d,d}(Y; \mathbf{c}\Theta; \mathbf{c}\Theta) \stackrel{\approx}{\longrightarrow} [(Y, \d Y), (\cP_d^{\mathbf c\Theta}, pt)] 
\end{eqnarray}
and a surjection  
\begin{eqnarray}\label{surjective}
\Phi^{\mathsf{imm}}: \mathsf{QT}^{\mathsf{imm}}_{d, d}(Y; \mathbf{c}\Theta; \mathbf{c}\Theta) \longrightarrow [(Y, \d Y), (\cP_d^{\mathbf c\Theta}, pt)]. \quad \diamondsuit
\end{eqnarray}
\end{corollary}

\begin{corollary}\label{D^n_A}
Let $\Theta \subset  \Lambda \subset \mathbf\Om_{\langle d]}$  be a pair of closed sub-posets. 
There exists a  group isomorphism (abelian for $n>1$)
\begin{eqnarray}\label{cor.group_iso}
\mathbf{\Phi}^{\mathsf{emb}}: \mathbf{QT}^{\mathsf{emb}}_{d, d}(D^n, \d D^n; \mathbf{c}\Theta, \mathbf{c}\Lambda; \mathbf{c}\Theta) \stackrel{\approx}{\longrightarrow} \pi_n(\cP_d^{\mathbf c\Theta}, \cP_d^{\mathbf c\Lambda}).
 \end{eqnarray}

Assuming the $\Lambda$-condition (\ref{special}), we get a group isomorphism 
\begin{eqnarray}
\Phi^{\mathsf{emb}}: \mathsf{OG}^{\mathsf{emb}}_{d, d}(n; \mathbf{c}\Theta; \mathbf{c}\Theta) \approx \mathsf{G}^{\mathsf{emb}}_{d, d}(n; \mathbf{c}\Theta; \mathbf{c}\Theta) \stackrel{\approx}{\longrightarrow} \pi_n(\cP_d^{\mathbf c\Theta}, pt).
\end{eqnarray} 
\end{corollary}

\begin{proof} The claim follows instantly from Theorem \ref{th.E-reg} and Proposition \ref{group}. 
\end{proof}

The next proposition is an attempt ``to reverse the direction" of the maps in (\ref{Y_dY}), (\ref{D_dD}). 

\begin{proposition}\label{prop.right_inverse} 
Let $\Theta' \subset \Theta \subset  \Lambda \subset \mathbf\Om$  be closed subposets, and $Y$ be a smooth compact manifold. Assume that $d' \geq d$ and $d' \equiv d \mod 2$.

Then the obvious  map 
\begin{eqnarray}\label{emb_im}
\mathcal A:\; \mathcal{QT}^{\mathsf{emb}}_{d, d'}(Y, \d Y; \mathbf{c}\Theta, \mathbf{c}\Lambda; \mathbf{c}\Theta') \to \mathcal{QT}^{\mathsf{imm}}_{d, d'}(Y, \d Y; \mathbf{c}\Theta,\mathbf{c}\Lambda; \mathbf{c}\Theta') 
\end{eqnarray}
is injective. That is, if two embeddings $\b_1: (M_1, \d M_1) \hookrightarrow (\R \times Y, \R \times \d Y)$ and $\b_2: (M_2, \d M_2) \hookrightarrow (\R \times Y,\, \R \times \d Y)$ are quaitopic via immersions, then they are quaitopic via embeddings.

Moreover, there exists a surjective ``resolution map"
\begin{eqnarray}\label{emb_imAA}
\mathcal{R}:\; \mathcal{QT}^{\mathsf{imm}}_{d, d'}(Y, \d Y; \mathbf{c}\Theta, \mathbf{c}\Lambda; \mathbf{c}\Theta') \to \mathcal{QT}^{\mathsf{emb}}_{d, d'}(Y, \d Y; \mathbf{c}\Theta, \mathbf{c}\Lambda; \mathbf{c}\Theta') 
\end{eqnarray}
that serves as the {\sf right inverse} of $\mathcal A$. In other words, there exists a {\sf canonical resolution} of any such immersion into an embedding, well-defined within the relevant quasitopy classes.\smallskip

If $Y$ is orientable, similar claims hold for the oriented quasitopies and the analogues of  maps $\mathcal A$, $\mathcal R$ that are available for them.
\end{proposition}

\begin{proof} In the proof, to simplify the notations further, put \hfill\break $\mathcal{QT}^{\mathsf{imm}}(Y) =_{\mathsf{def}}\; \mathcal{QT}^{\mathsf{imm}}_{d, d'}(Y, \d Y; \mathbf{c}\Theta,  \mathbf{c}\Lambda; \mathbf{c}\Theta')$, \; $\mathcal{QT}^{\mathsf{emb}}(Y) =_{\mathsf{def}}\; \mathcal{QT}^{\mathsf{emb}}_{d, d'}(Y, \d Y; \mathbf{c}\Theta, \mathbf{c}\Lambda; \mathbf{c}\Theta')$.

Using the map (\ref{Y_dY}), we see that the map $\Phi^{\mathsf{emb}}(Y)$  is a composition of  the map $\mathcal A$ and the map $$\Phi^{\mathsf{imm}}(Y):\; \mathcal{QT}^{\mathsf{imm}}(Y) \to \big[\big[(Y, \d Y),\; \e_{d, d'}: (\cP_d^{\mathbf c\Theta}, \cP_d^{\mathbf c\Lambda}) \to (\cP_{d'}^{\mathbf c\Theta'}, \cP_{d'}^{\mathbf c\Lambda})\big]\big].$$ 
However, by formula (\ref{emb_homotopy_A}) from Theorem \ref{th.E-reg}, we have a bijective map $$\Phi^{\mathsf{emb}}(Y): \mathcal{QT}^{\mathsf{emb}}(Y) \approx \big[\big[(Y, \d Y),\; \e_{d, d'}: (\cP_d^{\mathbf c\Theta}, \cP_d^{\mathbf c\Lambda}) \to (\cP_{d'}^{\mathbf c\Theta'}, \cP_{d'}^{\mathbf c\Lambda})\big]\big].$$

Hence we define the desired surjective map $\mathcal R$ by the formula $\mathcal R = \Phi^{\mathsf{emb}}(Y)^{-1}\circ\,\Phi^{\mathsf{imm}}(Y)$. It is on the level of definitions to check that $\mathcal R \circ \mathcal A = \mathsf{id}$. 

If $Y$ is orientable, the same arguments hold for the oriented quasitopies since $\mathcal{OQT}^{\mathsf{emb}}(Y) \hfill\break \approx \mathcal{QT}^{\mathsf{emb}}(Y)$ by Corollary \ref{cor.homotopy_type}.
\end{proof}

\begin{remark}\emph{ Of course, applying the resolution $\mathcal R$ to an immersion $\b: M \to \R\times Y$ produces an embedding $\b': M' \to \R\times Y$, where the topology of $M'$ is quite different from the topology of $M$. However, the \emph{sets} $\b'(M')$ and $\b(M)$ can be assumed to be arbitrary close in the Hausdorff distance in $\R\times Y$. } \hfill $\diamondsuit$ 
\end{remark}

\begin{corollary}\label{cor.fibers_of_R}
Each fiber of the resolution map $\mathcal{R}$ consists of the elements that share the same $\Phi^{\mathsf{imm}}$-value in $[[(Y, \d Y),\; \e_{d, d'}: (\cP_d^{\mathbf c\Theta}, \cP_d^{\mathbf c\Lambda}) \to (\cP_{d'}^{\mathbf c\Theta'}, \cP_{d'}^{\mathbf c\Lambda})]].$
\smallskip   

In particular, for $Y = D^n$, we get a splittable groups' extension
\begin{eqnarray}\label{group_extension}
1 \to \mathsf K \to \mathsf{QT}^{\mathsf{imm}}_{d, d'}(n; \mathbf{c}\Theta; \mathbf{c}\Theta') \stackrel{\mathcal R}{\rightarrow} \mathsf{QT}^{\mathsf{emb}}_{d, d'}(n; \mathbf{c}\Theta; \mathbf{c}\Theta') \to 1,
\end{eqnarray}
whose kernel $\mathsf K \subset \mathsf{QT}^{\mathsf{imm}}_{d, d'}(n; \mathbf{c}\Theta; \mathbf{c}\Theta')$ consists of quasitopies of immersions that belong to the kernel of the map in (\ref{Y_dY_A}) with $Y = D^n$. Via the operation $\uplus$, the subgroup $\mathsf  K$ acts on the fibers of the map $\mathcal R$ from (\ref{emb_imAA}).\smallskip
\end{corollary}

\begin{proof} We use the notations from the previous proof. By Proposition \ref{prop.right_inverse}, each fiber of $\mathcal{R}$ consists of the elements that share the same 
$\Phi^{\mathsf{imm}}$-value in $[[(Y, \d Y),\; \e_{d, d'}: (\cP_d^{\mathbf c\Theta}, \cP_d^{\mathbf c\Lambda}) \to (\cP_{d'}^{\mathbf c\Theta'}, \cP_{d'}^{\mathbf c\Lambda})]].$    
For $Y = D^n$, this implies that the group $\mathcal{QT}^{\mathsf{imm}}(D^n)$ is a semidirect product $\mathsf K \bowtie \mathcal{QT}^{\mathsf{emb}}(D^n)$ of the groups $\mathsf K$ and  $\mathcal{QT}^{\mathsf{emb}}(D^n)$. Again,  $\mathsf K$ is characterized by the property of having trivial $\Phi^{\mathsf{imm}}$-values. 

Since, for $\b_1 \in \mathcal{QT}^{\mathsf{imm}}(Y)$ and $\b_2 \in \mathcal{QT}^{\mathsf{imm}}(D^n)$, we have $\Phi^{\mathsf{imm}}(\b_1 \uplus \b_2) = \Phi^{\mathsf{imm}}(\b_1) \uplus \Phi^{\mathsf{imm}}(\b_2)$, we conclude that $\Phi^{\mathsf{imm}}(\b_1 \uplus \mathsf K) = \Phi^{\mathsf{imm}}(\b_1) \uplus 0 = \Phi^{\mathsf{imm}}(\b_1)$. Therefore $\mathsf K$ acts on fibers of the map $\mathcal R$.
\end{proof}

Since $\b \in \mathcal{QT}^{\mathsf{imm}}_{d, d'}(Y, \d Y; \mathbf{c}\Theta, \mathbf{c}\Lambda; \mathbf{c}\Theta')$ and $\mathcal R(\b) \in \mathcal{QT}^{\mathsf{emb}}_{d, d'}(Y, \d Y; \mathbf{c}\Theta, \mathbf{c}\Lambda; \mathbf{c}\Theta')$ share the same invariants  $\Phi^{\mathsf{imm}}(\b) = \Phi^{\mathsf{emb}}(\mathcal R(\b)) \in \big[\big[(Y, \d Y),\; \e_{d, d'}: (\cP_d^{\mathbf c\Theta}, \cP_d^{\mathbf c\Lambda}) \to (\cP_{d'}^{\mathbf c\Theta'}, \cP_{d'}^{\mathbf c\Lambda})\big]\big]$, we should be looking for new invariants that distinguish between $\b$ and $\mathcal R(\b)$. \smallskip
 
Relying on Proposition \ref{prop.k-normal_imm_to_bord}, 
we may use the bordism classes $[\Sigma_k^\b] \in \mathbf B_{n-k+1}(Y, \d Y)$ (or $[\Sigma_k^\b] \in \mathbf{OB}_{n-k+1}(Y, \d Y)$ when $Y$ is oriented) of the the $k$-self-intersection manifolds $\{\Sigma_k^\b\}_{k\geq 2}$ of  $\b$ to produce the desired distinguishing invariants.  
In a sense, these quasitopy invariants ``ignore" the foliations $\mathcal L, \mathcal L^\bullet$. Evidently, they vanish  on $\mathcal{QT}^{\mathsf{emb}}_{d, d'}(Y, \d Y; \mathbf{c}\Theta, \mathbf{c}\Lambda; \mathbf{c}\Theta')$. 

Therefore, for $\dim(Y)= n$, assuming that $\min\{n+1,\, \hat m(\mathbf c\Theta_{\langle d]})\} \geq 2$ (see Definition \ref{def.m_hat}), the correspondence $\b \leadsto \{\pi \circ \b \circ p_1: (\Sigma_k^\b, \d \Sigma_k^\b) \to (Y, \d Y)\}_k$ produces a map 
\begin{eqnarray}\label{eq.Sigma-bordisms}
\mathcal{BB}:\, \mathcal{QT}^{\mathsf{imm}}_{d, d'}(Y, \d Y; \mathbf{c}\Theta, \mathbf{c}\Lambda; \mathbf{c}\Theta')\big / \mathcal{QT}^{\mathsf{emb}}_{d, d'}(Y, \d Y; \mathbf{c}\Theta, \mathbf{c}\Lambda; \mathbf{c}\Theta') \to \\ \to \bigoplus_{k \in [2,\, \min\{n+1,\, \hat m(\mathbf c\Theta_{\langle d]})\}]} \mathbf B_{n-k+1}(Y, \d Y), \nonumber
\end{eqnarray}
which has a potential to discriminate between  $\mathcal{QT}^{\mathsf{imm}}_{d, d'}(Y, \d Y;  \mathbf{c}\Theta, \mathbf{c}\Lambda; \mathbf{c}\Theta')$ and $ \mathcal{QT}^{\mathsf{emb}}_{d, d'}(Y, \d Y; \hfill\break\mathbf{c}\Theta, \mathbf{c}\Lambda; \mathbf{c}\Theta')$. Lemma \ref{odd_intersections} provides the simplest example where it does.
\smallskip

Next, we ``mix" the tangency patterns of immersions $\b$ to $\mathcal L$ with the self-intersections of $\b$ (which also influence the $\mathcal L$-tangency patterns). For simplicity, we assume the $\Lambda$-condition and that $Y$ is oriented. 

Let $\b: (M, \d M) \to (\R \times Y, \R \times \d Y)$ be an immersion whose quasitopy class belongs to $\mathcal{QT}^{\mathsf{imm}}_{d, d'}(Y, \d Y;  \mathbf{c}\Theta, \mathbf{c}\Theta')$. Let $\mathsf A$ be an abelian group. We pick a cohomology class $\theta \in H^{n-k+1}(\cP_d^{\mathbf c \Theta}, pt ; \mathsf A)$ and evaluate its pull-back $(\Phi^\b \circ \pi \circ \b \circ p_1)^\ast(\theta)$ on the relative fundamental class $[\Sigma_k^\b, \d \Sigma_k^\b]$. This construction leads to the following proposition. 

\begin{proposition}\label{prop.evaluation} Let $\Theta' \subset \Theta \subset \mathbf\Om$ be closed posets. We assume the $\Lambda$-condition. Let $Y$ be an oriented smooth compact manifold. Pick a positive integer $k \leq \min\{\hat m(\mathbf c\Theta_{\langle d]}, n+1\}$. 

For any cohomology class $\theta \in H^{n-k+1}(\cP_d^{\mathbf c \Theta}, pt ; \mathsf A)$, the evaluation
$$\chi_k(\theta, \b) := \big\langle (\Phi^\b \circ \pi \circ \b \circ p_1)^\ast(\theta),\;  [\Sigma_k^\b, \d \Sigma_k^\b] \big\rangle \in \mathsf A$$
is an invariant of the quasitopy class $[\b] \in \mathcal{QT}^{\mathsf{imm}}_{d, d'}(Y, \d Y;  \mathbf{c}\Theta, \mathbf{c}\Theta')$. 

If $\b$ is an embedding, then $\chi_k(\theta, \b) = 0$ for $k \geq 2$. 
\end{proposition}

\begin{proof}  By Corollary \ref{cor.Sigma_k_is_bordism_invariant}, the bordism class of $\pi \circ \b \circ p_1: (\Sigma_k^\b, \d \Sigma_k^\b) \to (Y, \d Y)$ is invariant under a change of $\b$ within its quasitopy class.  By Theorem \ref{th.LIFT}, the cocycle $(\Phi^\b)^\ast(\theta)$ is an invariant of the quasitopy class of $\b$.  By the topological Stokes' Theorem, the cocycle $(\Phi^\b)^\ast(\theta)$, being evaluated on the cycle $(\pi \circ \b \circ p_1)(\Sigma_k^\b, \d \Sigma_k^\b)$ is an invariant of $[\b] \in \mathcal{QT}^{\mathsf{imm}}_{d, d'}(Y, \d Y;  \mathbf{c}\Theta, \mathbf{c}\Theta')$. 
\end{proof}

Here is the simplest example of an invariant $\rho_\b := \frac{1}{n+1}\big[(\pi \circ \b \circ p_1): \Sigma_{n+1}^\b \to Y\big ]$, delivered by $\mathcal{BB}$ from (\ref{eq.Sigma-bordisms}), which does discriminate between the quasitopies of embeddings and immersions. 
 
\begin{lemma}\label{odd_intersections}
Let $\dim M = n$ and let $\rho_\b$ be the number of points in $Y$ where exactly $n+1$ branches of $\b(M)$ meet transversally in $\R \times Y$. If $\rho_\b \equiv 1 \mod 2$, then $\b \in \mathcal{QT}^{\mathsf{imm}}_{d, d'}(Y, \d Y; \mathbf{c}\Theta, \mathbf{c}\Lambda; \mathbf{c}\Theta')$ is nontrivial, provided $\mathbf{c}\Lambda = (\emptyset), (1)$. 
\end{lemma}

 \begin{proof}
 Let $F_\b$ denote the finite set of points, where exactly $n+1$ branches of $\b(M)$ meet transversally in $\R \times Y$.  Assume that $\b$ be quasitopic to trivial immersion with the help of a cobordism $B: N \to \R \times Y \times [0, 1]$. Then, in general, $B(N)$ may have points where exactly $n+2$ branches of $B(N)$ meet transversally. Let $E_B$ denote their set and let $\eta_B := \#(E_B)$. The points, where exactly $n+1$ branches of $B(N)$ meet transversally, form a graph $\Gamma_B \subset  \R \times Y \times [0, 1]$. If $\mathbf{c}\Lambda = (\emptyset), (1)$, then $\Gamma_B$ has $\rho_\b$ univalent vertices and $\eta_B$ vertices of valency $n+2$. Every edge of $\Gamma_B$ that does not terminate at a valency one vertex from $F_\b$ is attached to $E_B$. Thus counting the edges of $\Gamma_B$ we get $2(n+2)\eta_B - \rho_\b \equiv 0 \mod 2$. Therefore, when $\rho_\b \equiv 1 \mod 2$,  we get a contradiction with the assumption that $\b$ is quasitopic to trivial immersion. As a result, any immersion $\b$ with an odd number $\rho_\b$ is nontrivial in  $\mathcal{QT}^{\mathsf{imm}}_{d, d'}(Y, \d Y; \mathbf{c}\Theta, \mathbf{c}\Lambda; \mathbf{c}\Theta')$, provided $\mathbf{c}\Lambda = (\emptyset), (1)$.
\end{proof}

\begin{figure}[ht]
\centerline{\includegraphics[height=2.5in,width=4.7in]{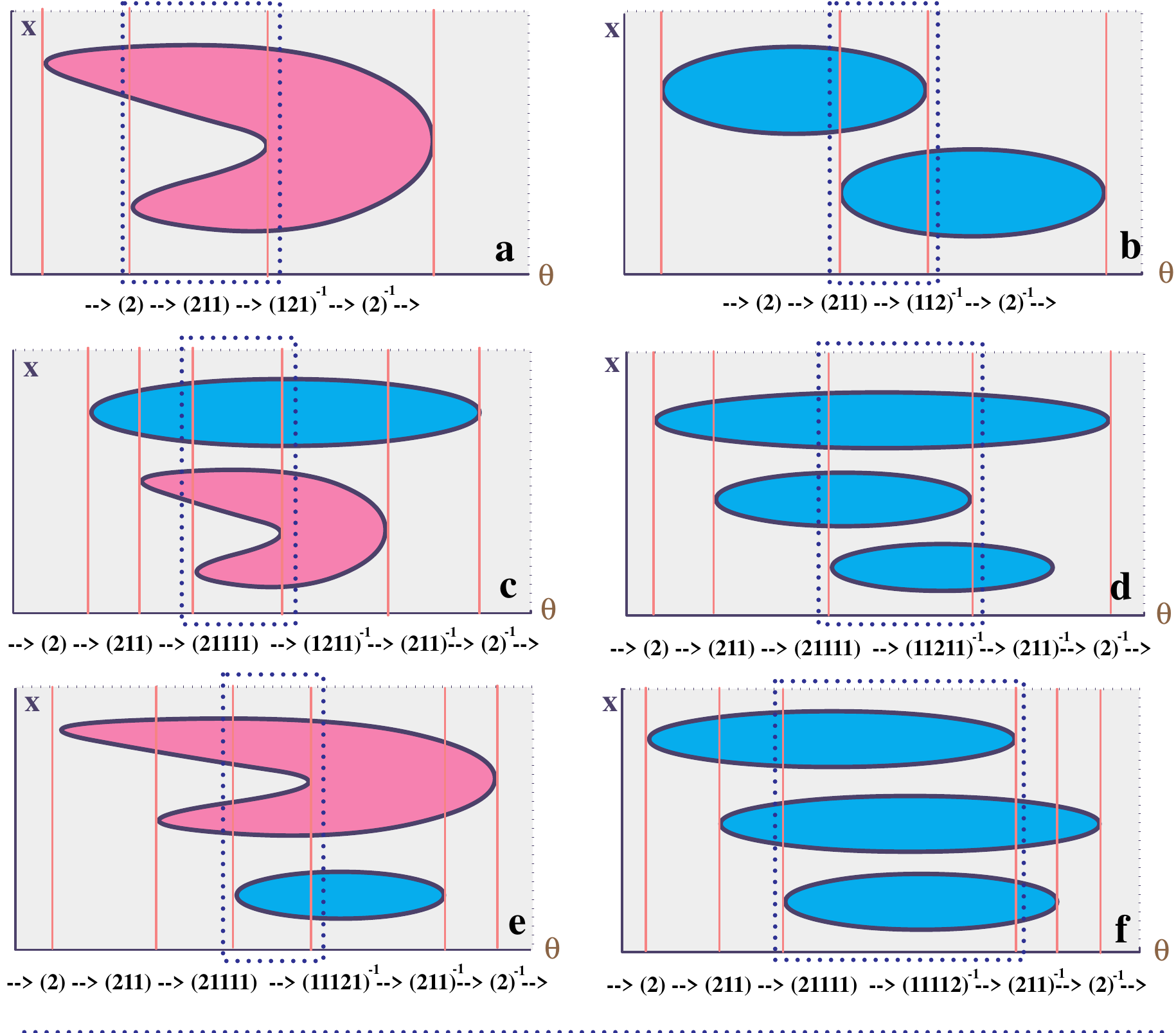}}
\bigskip
\caption{\small{Six generators $a, b, c, d, e, f$ of the free group $\mathsf{QT}^{\mathsf{emb}}_{6, 6}(1; \mathbf{c}\Theta; \mathbf{c}\Theta)$, where $\Theta \subset \mathbf\Om_{\langle 6]}$ is the closed poset generated by the elements $\om$ with two $2$'s or one $3$. The vertical lines mark crossing the walls of the $6$-dimensional chambers in $\cP_6^{\mathbf c\Theta}$.}} 
\label{fig.BLA}
\end{figure}

\begin{example}\label{ex.3.1} 
\emph{
Pick $n=1$, $d' = d= 6$. Let $\mathbf{c}\Theta' = \mathbf{c}\Theta$ consist of $\om \in \mathbf\Om_{\langle 6]}$  such that all their entries are $1$'s and $2$'s and a single $2$ is present at most. Put $Y = D^1$.  Then $\mathsf{QT}^{\mathsf{emb}}_{6, 6}(1; \mathbf{c}\Theta; \mathbf{c}\Theta) \approx \pi_1(\cP_6^{\mathbf c\Theta}, pt)$.  By \cite{KSW1}, Theorem 2.4, the  latter group is a free group $\mathsf F_6$ in $6$ generators (see Fig. 3). The group extension in (\ref{group_extension}) reduces to $$1 \to \mathsf K \to \mathsf{QT}^{\mathsf{imm}}_{6, 6}(1; \mathbf{c}\Theta; \mathbf{c}\Theta) \stackrel{\mathcal R}{\rightarrow} \mathsf F_6 \to 1.$$
The figure $\infty$, placed ``horizontally" in $\R \times D^1$ ($D^1$ being the horizontal direction) represents an immersion $\b: S^1 \to \R \times D^1$. We claim that $\b$ belongs to the kernel $\mathsf K$ from (\ref{group_extension}). Indeed, consider a flip $\tau: \R \times D^1 \to \R \times D^1$ whose fixed point set is the vertical line through the singularity of figure $\infty$. Due to the symmetry of $\infty$ with respect to $\tau$,  the path $\Phi^{\mathsf{imm}}(D^1, \d D^1) \subset (\cP_6^{\mathbf c\Theta}, pt)$ is relatively contractible. Thus $\infty$ is indeed in the kernel $\mathsf K$. By Lemma \ref{odd_intersections}, $\b$ is nontrivial in $\mathsf{QT}^{\mathsf{imm}}_{6, 6}(1; \mathbf{c}\Theta; \mathbf{c}\Theta)$. By the same lemma, the quasitopy class of any collection of smooth curves with an \emph{odd} number of crossings is nontrivial in $\mathsf{QT}^{\mathsf{imm}}_{6, 6}(1; \mathbf{c}\Theta; \mathbf{c}\Theta)$. At the same time, the immersion $\infty \uplus \infty \subset \R \times D^1$ is trivial in $\mathsf{QT}^{\mathsf{imm}}_{6, 6}(1; \mathbf{c}\Theta; \mathbf{c}\Theta)$ since it is the boundary of a horseshoe surface $N \approx \infty \times I \subset \R \times D^1 \times [0, 1]$ with the $\mathcal L^\bullet$-tangency patterns that belong to $\mathbf{c}\Theta$.  So the immersion $\b$, such that $\b(S^1) = \infty$, is an element of order two in the kernel $\mathsf K$. However, if we orient $S^1$, then $\infty$ becomes an element of infinite order in $\mathsf{ker}(\mathsf{OQT}^{\mathsf{imm}}_{6, 6}(1; \mathbf{c}\Theta; \mathbf{c}\Theta) \to \mathsf{OQT}^{\mathsf{emb}}_{6, 6}(1; \mathbf{c}\Theta; \mathbf{c}\Theta))$.} 

\emph{In fact, any  collection of curves $\b(M)$ in $\R \times D^1$ with the tangency patterns in $\mathbf{c}\Theta$, with an odd number of transversal crossings, and which is symmetric under an involution $\tau: \R \times D^1 \to \R \times D^1$ which preserves the oriented foliation $\mathcal L$,  belongs to $\mathsf K$ and has order $2$ there. Based on a sparse evidence, we conjecture that  $\mathsf K \approx \Z_2$. If this is true, then  $\mathsf{QT}^{\mathsf{imm}}_{6, 6}(1; \mathbf{c}\Theta; \mathbf{c}\Theta)$ is a semi-direct product $\Z_2 \bowtie \mathsf F_6$. 
} 
\hfill $\diamondsuit$
\end{example}

The next theorem describes one very general mechanism for generating {\sf characteristic cohomology classes} for immersions/embeddings $\b: M \to \R \times Y$ with $\Theta$-restricted tangency patterns to $\mathcal L$.
\begin{theorem}\label{th.characteristic classes} For a compact manifold $Y$, $d' =d$, and $\Theta' = \Theta$, assuming the $\Lambda$-condition, any immersion/embedding $\b: (M, \d M) \to (\R \times Y, \R \times \d Y)$ as in Proposition \ref{th.IMMERSIONS} induces a characteristic homomorphism $(\Phi^\b)^\ast$ from the (co)homology of the differential complex  $\{\d^\#: \Z[\Theta^\#_{\langle d]}] \to \Z[\Theta^\#_{\langle d]}]\}$ in (\ref{eq.quotient_complex}) to the cohomology $H^\ast(Y, \d Y; \Z)$. 

Quasitopic immersions/embeddings induce the same characteristic homomorphisms.
\end{theorem}
\begin{proof} For any closed poset $\Theta \subset \mathbf\Om_{\langle d]}$, by Corollary 2.6 and Theorem 2.2 (\cite{KSW2}), the homology of the differential complex in (\ref{eq.quotient_complex})  
is isomorphic to $H^\ast(\mathcal P_d^{\mathbf c\Theta}, pt; \Z)$. By Proposition \ref{th.IMMERSIONS}, any immersion $\b: (M, \d M) \to (\R \times Y,\, \R \times \d Y)$ whose tangency patterns belong to $\mathbf c\Theta$ and the ones of $\b|_{\d M}$ to $\mathbf c\Lambda = (\emptyset) \text{ or } (1)$, produces a map $\Phi^\b: (Y, \d Y) \to (\mathcal P_d^{\mathbf c\Theta}, pt)$ whose homotopy class is determined by $\b$. This $\b$ induces a natural map $\b^\ast: H^\ast(\mathcal P_d^{\mathbf c\Theta}, pt; \Z) \to H^\ast(Y, \d Y; \Z)$ in cohomology, and therefore, a characteristic homomorphism $(\tilde\Phi^\b)^\ast: H^\ast(\Z[\Theta^\#], \d^\#) \to H^\ast(Y, \d Y; \Z)$. 

By Proposition \ref{th.IMMERSIONS}, $(d, d; \mathbf c\Theta, (\emptyset); \mathbf c\Theta)$-quasitopic immersions/embeddings $\b$ induce homotopic maps $\tilde\Phi^\b$ and hence  the same characteristic homomorphisms $(\tilde\Phi^\b)^\ast$.
\end{proof}

Thus various homomorphisms $(\tilde\Phi^\b)^\ast$ may distinguish between different quasitopy classes of immersions $\b$. We will soon exhibit multiple examples, where they do.


\subsection{Quasitopies of immersions/embeddings with special forbidden combinatorics $\Theta$ and the stabilization by $d$} 
In this subsection, we take advantage of a few results from \cite{V} and \cite{KSW2}, as formulated in Section 2, to advance computations of quasitopies of embeddings with special combinatorial patterns $\mathbf{c}\Theta$ and $\mathbf{c}\Theta'$ of tangency to the foliations $\mathcal L$ and $\mathcal L^\bullet$, respectively. \smallskip

We start with the Arnold-Vassiliev case of real polynomials with, so called, {\sf moderate singularities} \cite{Ar}, \cite{V}. Let $\Theta_{\max \geq k} \subset \mathbf\Om_{\langle d]}$ be the closed poset consisting of $\om$'s with the maximal entry $\geq k$. For $k \geq 3$, the cohomology $H^j(\cP^{\mathbf c\Theta_{\max \geq k}}, pt; \Z)$ is isomorphic to $\Z$  in each dimension $j$ of the form $(k- 2)m$, where the integer $m \leq d/k$, and vanishes otherwise \cite{Ar}.

Based on Vassiliev's computation of the cohomology ring $H^\ast(\cP^{\mathbf c\Theta_{\max \geq k}}, pt; \Z)$ (see \cite{V}, Theorem 1 on page 87), consider the graded ring $\mathcal Vass_{d, k}$, multiplicatively generated over $\Z$ by the elements $\{e_m\}_{m \leq d/k}$ of the degrees $\deg(e_m) = m(k-2)$, subject to the relations
\begin{eqnarray} 
e_l \cdot e_m = \frac{(l+m)!}{l! \cdot m!}\; e_{l+m} \text{\; for } k\equiv 0 \mod 2, \text{ and the relations } 
\end{eqnarray}
\begin{eqnarray}
e_1\cdot e_1 = 0, \quad e_1\cdot e_{2m} = e_{2m+1}, \nonumber \
\end{eqnarray}
\begin{eqnarray}
e_{2l} \cdot e_{2m} = \frac{(l+m)!}{l! \cdot m!}\; e_{2l+2m} \text{\; for } k\equiv 1 \mod 2.
\end{eqnarray}

Guided by Theorem \ref{th.characteristic classes} and employing Theorem \ref{th.LIFT}, we get the following claim. 

\begin{proposition}\label{prop.Vass} Let $k \geq 3$. Consider an immersion $\b: (M, \d M) \to (\R \times Y, \R \times \d Y)$ whose tangency patterns to the foliation $\mathcal L$ belong to $\mathbf c\Theta_{\max \geq k} \subset \mathbf\Om_{\langle d]}$ and the ones of $\b|_{\d M}$ either form an empty set $(\emptyset)$ when $d \equiv 0 \mod 2$, or the set $(1)$ when $d \equiv 1 \mod 2$. \smallskip

Then $\b$  generates a characteristic ring homomorphism $(\Phi^\b)^\ast: \mathcal Vass_{d, k} \to H^\ast(Y, \d Y; \Z)$. 
The homomorphism $(\Phi^\b)^\ast$ is an invariant of the quasitopy class of $\b$. 
In other words, we get a map 
$\Phi^\ast_{d, k}: \mathsf{QT}^{\mathsf{imm}}_{d, d}(Y; \mathbf{c}\Theta_{\max \geq k};  \mathbf{c}\Theta_{\max \geq k}) \to \mathsf{Hom_{ring}}\big(\mathcal Vass_{d, k},\; H^\ast(Y, \d Y; \Z)\big). \hfill \diamondsuit$
\end{proposition}

\begin{remark} \emph{Note that if some generator $e_\ell \in \mathcal Vass_{d, k}$ is mapped by $(\Phi^\b)^\ast$ to zero in $H^{(k-2)\ell}(Y, \d Y; \Z)$, then the images of all $e_q$ with $q > \ell$ must be torsion elements. Thus some cohomology rings $H^\ast(Y, \d Y; \Z)$ may not be able to accommodate nontrivial images of the Vassiliev ring $\mathcal Vass_{d, k}$ in positive degrees.}  

\emph{In contrast, here is a case when the accommodation seems possible: take $Y = \C\P^3$, $k =4$, and $d = 12$.  Consider $9$-dimensional variety $\cP_{12}^{\Theta_{\max \geq 4}}$, its compliment  $\cP_{12}^{\mathbf c\Theta_{\max \geq 4}}$, and  the cohomology ring $\mathcal Vass_{12, 4} \approx H^\ast(\cP_{12}^{\mathbf c\Theta_{\max \geq 4}}; \Z)$. Then the ring homomorphism $\Phi^\ast: \mathcal Vass_{12,4} \to H^\ast(\C\P^3; \Z) \approx \Z[a]/\{a^4 = 0\}$ that sends  $\Phi^\ast(e_1) := 6a$, $\Phi^\ast(e_2) := 18a^2$, and $\Phi^\ast(e_3):= 36a^3$ embeds $\mathcal Vass_{12,4}$ as a subring of $\Z[a]/\{a^4 = 0\}$. 
We speculate that there is an embedding $\Phi: \C\P^3 \to \cP_{12}^{\mathbf{c}\Theta_{\max \geq 4}}$ that induces that $\Phi^\ast$. 
} 
\hfill $\diamondsuit$
\end{remark}

Propositions \ref{prop.evaluation}, \ref{prop.Vass} lead to the following corollary.
\begin{corollary} Let $M, Y$ be compact oriented smooth $n$-manifolds. For any generator $e_\ell \in \mathcal Vass_{d, k}$ of degree $\ell(k-2)$ and an immersion $\b: (M, \d M) \to (\R \times Y, \R \times \d Y)$ with $k$-moderate tangency patterns to $\mathcal L$, the integer 
$$\chi_k(e_\ell, \b) = \big\langle (\Phi^\b \circ \pi \circ \b  \circ p_1)^\ast(e_\ell),\,  [\Sigma_{\ell(k-2)}^\b, \d \Sigma_{\ell(k-2)}^\b] \big\rangle$$
is an invariant of the quasitopy class $[\b] \in \mathcal{QT}^{\mathsf{imm}}_{d, d'}(Y, \d Y;  \mathbf{c}\Theta_{\max \geq k}, \mathbf{c}\Theta_{\max \geq k})$. The invariant $\chi_k(e_\ell, \b) = 0$ when $\b$ is an embedding and $k > 1$, or when $H^{\ell(k-2)}(Y, \d  Y; \Z)= 0$.  

\hfill$\diamondsuit$
\end{corollary}

The next proposition is a stabilization result by the increasing $d' \geq d$ for the embeddings with {\sf moderate tangencies} to the foliation $\mathcal L$ on $\R \times Y$.

\begin{proposition}\label{prop.Vass_case} Let $k \geq 4$. If $\dim Y \leq (k-2)(\lceil d/k \rceil +1)- 2$,
then the classifying map 
\begin{eqnarray}\label{emb_homotopy_AA}
\qquad\qquad \Phi^{\mathsf{emb}}: \mathsf{QT}^{\mathsf{emb}}_{d, d'}(Y; \mathbf{c}\Theta_{\max \geq k}; \mathbf{c}\Theta_{\max \geq k}) \stackrel{\approx}{\longrightarrow} 
 \big[(Y, \d Y), (\cP_d^{\mathbf c\Theta_{\max \geq k}}, pt) \big]
\end{eqnarray}
is a bijection, and the classifying map
\begin{eqnarray}\label{emb_homotopy_AAA}
\qquad\qquad \Phi^{\mathsf{imm}}: \mathsf{QT}^{\mathsf{imm}}_{d, d'}(Y; \mathbf{c}\Theta_{\max \geq k}; \mathbf{c}\Theta_{\max \geq k}) \stackrel{\approx}{\longrightarrow} 
 \big[(Y, \d Y), (\cP_d^{\mathbf c\Theta_{\max \geq k}}, pt) \big]
\end{eqnarray}
is a surjection for any $d' \geq d$, $d' \equiv d \mod 2$. 

In particular, for a given $Y$, $\mathsf{QT}^{\mathsf{emb}}_{d, d'}(Y; \mathbf{c}\Theta_{\max \geq k}; \mathbf{c}\Theta_{\max \geq k})$ stabilizes for all  $d' \geq d \geq \frac{k}{k-2}(\dim Y +4 -k)$, a linear function in $\dim Y$.
\end{proposition}

\begin{proof} Let $q := (k-2)(\lceil d/k \rceil +1)- 2$. By \cite{V}, Theorem 3 on page 88, the $\e_{d, d'}$-induced homomorhpism $(\e_{d, d'})_\ast: \pi_i(\cP_d^{\mathbf c\Theta_{\max \geq k}}) \to \pi_i(\cP_{d'}^{\mathbf c\Theta_{\max \geq k}})$ of homotopy groups is an isomorphism for all $i  \leq q$. Thus the the two spaces $q$-connected. Therefore, if $\dim Y \leq q$, then no element of  $\big[(Y, \d Y), (\cP_d^{\mathbf c\Theta_{\max \geq k}}, pt) \big]$ becomes, under the composition with $\e_{d, d'}$, nill-homotopic. Hence, $$\big[\big[(Y, \d Y),\; \e_{d, d'}: \cP_d^{\mathbf c\Theta_{\max \geq k}} \to \cP_{d'}^{\mathbf c\Theta_{\max \geq k}}\big]\big] = \big[(Y, \d Y), (\cP_d^{\mathbf c\Theta_{\max \geq k}}, pt)\big].$$ In particular, if $\dim Y \leq q$, then $\big[(Y, \d Y), (\cP_d^{\mathbf c\Theta_{\max \geq k}}, pt)\big]$ is stable in $d$, provided $d \geq \frac{k}{k-2}(\dim Y +4 -k)$. 
Therefore, by Theorem \ref{th.E-reg}, all the claims of the proposition are validated.
\end{proof}

\begin{example} Let $k=4$. By Proposition \ref{prop.Vass_case} and formula (\ref{emb_homotopy_AA}), for any compact connected surface $Y$, we get bijections:
\begin{eqnarray}
\mathsf{QT}^{\mathsf{emb}}_{4, 4}(Y; \mathbf{c}\Theta_{\max \geq 4}; \mathbf{c}\Theta_{\max \geq 4}) \stackrel{\approx}{\longrightarrow} \pi^2(Y/\d Y),
\nonumber \\
\mathsf{QT}^{\mathsf{emb}}_{4, 6}(Y; \mathbf{c}\Theta_{\max \geq 4}; \mathbf{c}\Theta_{\max \geq 4}) \stackrel{\approx}{\longrightarrow} \pi^2(Y/\d Y),\nonumber \\
\mathsf{QT}^{\mathsf{emb}}_{6, 6}(Y; \mathbf{c}\Theta_{\max \geq 4}; \mathbf{c}\Theta_{\max \geq 4}) \stackrel{\approx}{\longrightarrow} \pi^2(Y/\d Y), \nonumber
\end{eqnarray}
whose target is the second cohomotopy group $\pi^2(Y/\d Y) \approx \Z$, the latter isomorphism being generated by the degree of maps to $S^2$.\smallskip

Let $Z$ be a simply-connected $CW$-complex whose cohomology ring, truncated in dimensions $\geq 5$,  is isomorphic to  
$\Z[e_1, e_2]/\{e_1^2 - 2e_2\}$, where $\deg e_1 =2, \deg e_2 = 4$.  
Then by Proposition \ref{prop.Vass_case}, for any compact smooth manifold $Y$ of dimension $\leq 4$, we get bijections:
\begin{eqnarray}
\mathsf{QT}^{\mathsf{emb}}_{8, 8}(Y; \mathbf{c}\Theta_{\max \geq 4}; \mathbf{c}\Theta_{\max \geq 4}) & \stackrel{\approx}{\longrightarrow} & [Y/\d Y, Z],\nonumber \\
\mathsf{QT}^{\mathsf{emb}}_{8, 10}(Y; \mathbf{c}\Theta_{\max \geq 4}; \mathbf{c}\Theta_{\max \geq 4}) & \stackrel{\approx}{\longrightarrow} & [Y/\d Y, Z],\nonumber \\
\mathsf{QT}^{\mathsf{emb}}_{10, 10}(Y; \mathbf{c}\Theta_{\max \geq 4}; \mathbf{c}\Theta_{\max \geq 4}) & \stackrel{\approx}{\longrightarrow} & [Y/\d Y, Z].
\nonumber
\; \diamondsuit
\end{eqnarray}
\end{example}
\smallskip

\begin{corollary}\label{cor.ArV} 
Let $\Theta_{\max \geq k} \subset \mathbf\Om_{\langle d]}$ be the closed poset of $\om$'s with at least one entry $\geq  k$. Assume that $k \leq d \leq 2k-1$. 

Then there is a group isomorphism $$\mathsf{OG}^{\mathsf{emb}}_{d,d}(n; \mathbf{c}\Theta_{\max \geq k}; \mathbf{c}\Theta_{\max  \geq k}) \approx \mathsf{G}^{\mathsf{emb}}_{d,d}(n; \mathbf{c}\Theta_{\max  \geq k}; \mathbf{c}\Theta_{\max  \geq k}) \approx \pi_n(S^{k-2}).$$ 

In other words, the quasitopy group of embeddings $\b: M \hookrightarrow \R \times D^n$ into the cylinder $\R \times D^n$ with no vertical tangencies of $\b(M)$ to $\mathcal L$ of the order $ \geq  k$ and with the total multiplicities $\{\mu_\b(y) \leq d\}_{y \in D^n}$ is isomorphic to the group $\pi_n(S^{k-2})$, provided that $d \in [k, 2k-1]$. 

As a result, in this range of $d$'s and for $n > 1$, all the groups $\mathsf{G}^{\mathsf{emb}}_{d, d}(n; \mathbf{c}\Theta_{\max  \geq k}; \mathbf{c}\Theta_{\max  \geq k})$ are \emph{finite} abelian, except for $n = k-2$ and $n= 2k-5$, where they are infinite of rank one.
\end{corollary}

\begin{proof} By Arnold's theorem \cite{Ar}, for $d \in [k , 2k+1]$, $\pi_n(\cP_d^{\mathbf c\Theta_{\max \geq k}}) \approx \pi_n(S^{k-2})$. Applying Corollary \ref{D^n_A}, the claim follows by examining the tables of homotopy groups of spheres.
\end{proof}

\begin{example} \emph{For $k=4$,  $d \in [4, 9]$, and all $n$, we get a group isomorphism $$\mathsf{G}^{\mathsf{emb}}_{d, d}(n; \mathbf{c}\Theta_{\max \geq 4}; \mathbf{c}\Theta_{\max \geq 4}) \approx \pi_n(S^2).$$  In particular, for $d \in [4, 9]$, visiting the table of homotopy groups of spheres, we get:
\begin{eqnarray}
&\mathsf{G}^{\mathsf{emb}}_{d, d}(3; \mathbf{c}\Theta_{\max \geq 4}; \mathbf{c}\Theta_{\max \geq 4}) \approx \Z,  \quad  
&\mathsf{G}^{\mathsf{emb}}_{d,d}(4; \mathbf{c}\Theta_{\max \geq 4}; \mathbf{c}\Theta_{\max \geq 4}) \approx \Z_2, \nonumber \\
&\mathsf{G}^{\mathsf{emb}}_{d,d}(14; \mathbf{c}\Theta_{\max \geq 4}; \mathbf{c}\Theta_{\max \geq 4}) \approx \Z_{84} \times \Z_2^2. \quad \quad      \nonumber
\end{eqnarray} 
It would be very interesting to understand and describe, in the sprit of the Arnold's pink ``kidneys" in Fig. 3  (see \cite{Ar}), \emph{the shapes} of embeddings $\b: M \hookrightarrow \R \times D^n$ that produce such mysterious periods... In principle, Theorem \ref{th.E-reg} contains the instructions for such attempts.  Perhaps, at least for $n =1, 2$, one has a fighting chance... \hfill $\diamondsuit$  } 
\end{example}

Let $H, G$ be two groups, and $\mathsf{Hom}(H, G)$ their group of homomorphisms. Then $G$ acts on $\mathsf{Hom}(H, G)$ by the conjugation: for any $\phi: H \to G$,  $h \in H$, and $g \in G$, we define $(Ad_g \phi)(h)$ by the formula $g^{-1}\phi(h)g$. We denote by $\mathsf{Hom}^\bullet(H, G)$ the quotient $\mathsf{Hom}(H, G)/ Ad_G$.
\smallskip

The next corollary deals with special $\Theta$'s for which $\cP_d^{\mathbf c\Theta}$ is a $K(\pi, 1)$-space \cite{KSW1}. 

\begin{corollary}\label{cor.free_group} Let $\mathbf c\Theta$ consist of all $\om$'s with entries $1$ and $2$ only and no more than a single entry $2$. Put $\kappa(d) := \frac{d(d-2)}{4}$ for $d \equiv 0 \mod 2$, and $\kappa(d) := \frac{(d -1)^2}{4}$ for $d \equiv 1 \mod 2$. Let $\mathsf F_{\kappa(d)}$ be the \emph{free} group in $\kappa(d)$ generators. \smallskip

Assume that either $Y$ is a closed manifold, or $\d Y \neq \emptyset$ and the $\Lambda$-condition is in place.
\smallskip

If $\d Y \neq \emptyset$, then there is a bijection 
$$\Phi^{\mathsf{emb}}: \mathsf{QT}^{\mathsf{emb}}_{d,d}(Y; \mathbf{c}\Theta; \mathbf{c}\Theta) \stackrel{\approx}{\longrightarrow} \mathsf{Hom}(\pi_1(Y), \mathsf F_{\kappa(d)})
$$
and a surjection 
$$\Phi^{\mathsf{imm}}: \mathsf{QT}^{\mathsf{imm}}_{d,d}(Y; \mathbf{c}\Theta; \mathbf{c}\Theta) \longrightarrow \mathsf{Hom}(\pi_1(Y), \mathsf F_{\kappa(d)}).$$

When $Y$ is closed, then similar claims hold with the targets of $\Phi^{\mathsf{emb}}$ and $\Phi^{\mathsf{imm}}$ being replaced by the set $\mathsf{Hom}^\bullet(\pi_1(Y), \mathsf F_{\kappa(d)})$.

In particular, $\Phi^{\mathsf{emb}}: \mathsf{QT}^{\mathsf{emb}}_{d,d}(S^1; \mathbf{c}\Theta; \mathbf{c}\Theta) \stackrel{\approx}{\longrightarrow} \mathsf F_{\kappa(d)}/ Ad_{\,\mathsf F_{\kappa(d)}}$, the free group of cyclic words in $\kappa(d)$ letters
(see Fig. 3).
\smallskip

Thus, if $\pi_1(Y)$ has no nontrivial free images, then the group $\mathsf{QT}^{\mathsf{emb}}_{d,d}(Y; \mathbf{c}\Theta; \mathbf{c}\Theta)$ is trivial. 
\end{corollary}

\begin{proof} 
By Theorem 2.4, \cite{KSW1}, for such a $\mathbf c\Theta$, the space $\mathcal P_d^{\mathbf c\Theta}$ is of the homotopy type of $K(\mathsf F_{\kappa(d)}, 1) = \bigvee_{\s =1}^{\kappa(d)} S^1_\s$. Thus, when $Y$ is closed, by the obstruction theory, $[Y, \bigvee_{\s =1}^{\kappa(d)} S^1_\s] \approx \mathsf{Hom}^\bullet(\pi_1(Y), \mathsf F_{\kappa(d)})$ (see \cite{Hu}, Section VI, F). When $\d Y \neq \emptyset$, the boundary of $Y$ is mapped by $\Phi^\b$ to the preferred point $pt \in \mathcal P_d^{\mathbf c\Theta}$. This makes it possible to consider the based loops in $Y$, the base point $b$ being chosen in  $\d Y$, and the based loops in $\bigvee_{\s =1}^{\kappa(d)} S^1_\s$, the base point $pt$ being the apex $\star$ of the bouquet.   Again, by the standard obstruction theory, we get $[(Y, \d Y),\, (\bigvee_{\s =1}^{\kappa(d)} S^1_\s, \star)] \approx \mathsf{Hom}(\pi_1(Y, b),\, \mathsf F_{\kappa(d)})$. 

Now, by Corollary \ref{D^n_A}, the claim follows. In particular, if $\pi_1(Y)$ has no nontrivial free images (say, $\pi_1(Y)$ is a finite group), the group $\mathsf{QT}^{\mathsf{emb}}_{d,d}(Y; \mathbf{c}\Theta; \mathbf{c}\Theta)$ is trivial. 

See Fig. 3 and \cite{K4},  \cite{KSW1} for more illustrations of  the cases when $\mathcal P_d^{\mathbf c\Theta}$ is a $K(\pi, 1)$-space. 
\end{proof}

\begin{proposition}\label{prop.skeleta_and_groups} Let $\Theta' = \Theta =_{\mathsf{def}} \{\om :\; |\om|' \geq k\}$, and $d' \geq d > k > 2$, $d' \equiv d \mod 2$. 

Then, assuming the $\Lambda$-condition (\ref{special}), we get a group homomorphism 
\begin{eqnarray}\label{eq.D_dD_k}
\Phi_{d, d'}^{\mathsf{i/e}}(n; \mathbf{c}\Theta; \mathbf{c}\Theta):\; \mathsf G^{\mathsf{i/e}}_{d, d'}(n; \mathbf{c}\Theta;  \mathbf{c}\Theta)   
\longrightarrow \pi_n(\cP_d^{\mathbf c\Theta})\big / \ker((\e_{d, d'})_\ast) \approx  \nonumber \\
\approx \; \pi_n\big(\bigvee_{i=1}^{A(d,\, k)} S^{k-1}_i\big) \big / \ker\Big\{\pi_n\big(\bigvee_{i=1}^{A(d,\, k)} S^{k-1}_i\big) \stackrel{(\tilde\e_{d, d'})_\ast}{\longrightarrow} \pi_n\big(\bigvee_{j=1}^{A(d',\, k)} S^{k-1}_j \big)\Big\},
\end{eqnarray}
which is an isomorphism for the embeddings and an epimorphism for immersions. The homomorphism $(\tilde\e_{d, d'})_\ast$ in (\ref{eq.D_dD_k}) is induced by the embedding $\e_{d, d'}: \cP_d^{\mathbf{c}\Theta} \to \cP_{d'}^{\mathbf{c}\Theta}$. The number $A(d, k) := A(d, k, 0)$ has been  introduced in (\ref{eq.A-bouquet}). \smallskip


For $n < k-1$, the group $\mathsf G^{\mathsf{emb}}_{d, d'}(n; \mathbf{c}\Theta, \mathbf c\Lambda;  \mathbf{c}\Theta)$ is trivial. 
\end{proposition}

\begin{proof} By formula (\ref{D_dD}) from Theorem \ref{prop_(D, dD)}, we produce the map $\Phi_{d, d'}^{\mathsf{i/e}}(n; \mathbf{c}\Theta; \mathbf{c}\Theta)$. By Corollary \ref{D^n_A}, it is a group epimomorphism for immersions and an isomorphism for the embeddings.

By Proposition \ref{prop.skeleton} and the Alexander duality, the spaces $\cP_d^{\mathbf c\Theta}$ and $\cP_{d'}^{\mathbf c\Theta}$  have the homology types of bouquets of $(k-1)$-spheres. Recall that a simply-connected $CW$-complex whose torsion-free $\Z$-homology is concentrated in a single dimension $q$ is homotopy equivalent to a bouquet of  $q$-spheres (see \cite{Ha}, Theorem 4C.1). 
Therefore the spaces $\cP_d^{\mathbf c\Theta}$ and $\cP_{d'}^{\mathbf c\Theta}$  have the homotopy types of bouquets of $(k-1)$-spheres, provided $k > 2$.  
With the help of these homotopy equivalences, the homomorphism $(\tilde\e_{d, d'})_\ast: \pi_n\big(\bigvee_{i=1}^{A(d,\, k)} S^{k-1}_i\big) \to \pi_n\big(\bigvee_{j=1}^{A(d',\, k)} S^{k-1}_j \big)$  is induced by the embedding 
$\e_{d, d'}: \cP_d^{\mathbf{c}\Theta} \to \cP_{d'}^{\mathbf{c}\Theta}$. This validates the first claim.

The last claim follows from the cellular approximation of a continuous map in its homotopy  class. 
\end{proof}

For $d' =d$, we are able to compute the quasitopy classes of $k$-{\sf flat} (see Definition \ref{def.k-flat}) embeddings into $\R \times Y$. They correspond to the case $q =0$ in the following proposition.  

\begin{proposition}\label{prop.kq-skeleton} For $k \in [3, d-1]$ 
and $q \in [0, d]$, consider the closed subposet $\Theta_{|\sim |' \geq k}^{(q)} := \big\{\om:\; |\om|' \geq k,\,  |\om| \in [q, d]\big\}$. Assume the $\Lambda$-condition (\ref{special}).

Then, for a compact smooth $n$-manifold $Y$ with boundary, the quasitopy set of embeddings is described by the formula 
$$\mathsf{QT}^{\mathsf{emb}}_{d, d}\big(Y; \mathbf{c}\Theta_{|\sim |' \geq k}^{(q)}; \mathbf{c}\Theta_{|\sim |' \geq k}^{(q)}\big) \; \approx \; \big[(Y, \d Y),\; (\bigvee_{i=1}^{A(d,\, k,\, q)} S^{k-1}_i,\; \star)\big],$$ 
where $\star$ is the apex of the bouquet, and the integer $A(d,\, k,\, q)$ is defined by (\ref{eq.A-bouquet}).
\smallskip

In particular, the quasitopic classes of embeddings in the cylinder $\R \times D^n$ with the combinatorial patterns $\om$ such that $|\om|' < k$ and $|\om| < q$ are described by the group isomorphism $$\mathsf{QT}^{\mathsf{emb}}_{d,d}\big(D^n; \mathbf{c}\Theta_{|\sim |' \geq k}^{(q)};  \mathbf{c}\Theta_{|\sim |' \geq k}^{(q)}\big)\; \approx \; \pi_n\big(\bigvee_{i=1}^{A(d,\, k,\, q)} S^{k-1}_i,\, \star \big).$$
\end{proposition}

\begin{proof} The arguments are similar to the ones used in the proof of Proposition \ref{prop.skeleta_and_groups}. To simplify notations, temporarily, put $\Theta = \Theta_{|\sim |' \geq k}^{(q)}$. 

By Theorem \ref{prop.skeleton}, $\bar \cP^{\Theta}_d$ has the integral homology type of a bouquet of $A(d, k, q)$ many  $(d-k)$-dimensional spheres (\cite{KSW2}). Since the space $\bar\cP^{\Theta}_d$ is $(d-k)$-dimensional  $CW$-complex \cite{KSW1},  for $k > 2$, the space $\cP^{\mathbf c \Theta}_d$ is simply-connected.  
By the Alexander duality, its integral homology is concentrated in the dimension $k-1$ and is isomorphic to $\Z^{A(d, k, q)}$. By Proposition 4C.1 from \cite{Ha}, $\cP^{\mathbf c \Theta}_d$ has the homotopy type of the bouquet $ \bigvee_{i =1}^{A(d, k, q)} S^{k-1}_i$.  
Now, by Theorem \ref{th.E-reg}, the claim follows. 
\end{proof}

\begin{theorem}\label{isotopies_from_table} Let $Y$ be either a closed smooth $n$-manifold or a manifold with boundary, in which case, we assume the $\Lambda$-condition. We denote by $\langle \om \rangle$ the closed subposet of $\mathbf\Om_{\langle d]}$, consisting of elements $\om' \preceq \om$. 

Then for all $d \leq 13$ and all $\om$ such that $|\om|' > 2$, the quasitopy $\mathsf{QT}^{\mathsf{emb}}_{d, d}(Y; \mathbf{c}\langle \om \rangle; \mathbf{c}\langle \om \rangle)$ either is trivial (consists of single element), or is isomorphic to the cohomotopy set $\pi^k(Y)$, where $k = k(\om) \in [|\om|' -1,\, d-1]$. 

For $d \leq 13$, the table (\ref{big_list}) in Appendix 
lists all $\om$ and the corresponding $k = k(\om)$ for which $\mathcal P_d^{\mathbf c \langle \om \rangle}$ is non-contractible (in fact, a homology $k$-sphere). 
\end{theorem}

\begin{proof}The claim follows by combining Theorem \ref{th.E-reg} with the table (\ref{big_list}) in Appendix.
\end{proof}

\begin{example}\emph{Let $d =8$ and $\om = (4)$. Then, for any closed $Y$, using the list (\ref{big_list}) in Appendix, we get a bijection $\mathsf{QT}^{\mathsf{emb}}_{8,8}(Y; \mathbf{c}\langle(4) \rangle; \mathbf{c}\langle(4) \rangle) \approx \pi^4(Y)$, the $4$-cohomotopy set of $Y$. In particular, the we get the following group isomorphism: $$\mathsf{QT}^{\mathsf{emb}}_{8,8}(S^7; \mathbf{c}\langle(4) \rangle; \mathbf{c}\langle(4) \rangle) \approx \pi_7(S^4) \approx \Z \times \Z_{12}.$$} 
\emph{Let $d =12$ and $\om = (11213)$. Then, for any closed $Y$, by (\ref{big_list}), we get a bijection $$\mathsf{QT}^{\mathsf{emb}}_{12,12}(Y; \mathbf{c}\langle(11213) \rangle; \mathbf{c}\langle(11213) \rangle) \approx \pi^6(Y).$$ In particular, we get an isomorphism  $\mathsf{QT}^{\mathsf{emb}}_{12,12}(S^9; \mathbf{c}\langle(11213) \rangle; \mathbf{c}\langle(11213) \rangle) \approx \pi_9(S^6) \approx \Z_{24}$.} 

\hfill $\diamondsuit$
\end{example}

\subsection{Relation between quasitopies and inner framed cobordisms}

For a compact smooth $Y$ of dimension $n \geq k-1$, Proposition \ref{prop.kq-skeleton} below points to a peculiar  connection between the quasitopies of embeddings $\mathsf{QT}^{\mathsf{emb}}_{d, d}\big(Y; \mathbf{c}\Theta_{|\sim |' \geq k}^{(q)}; \mathbf{c}\Theta_{|\sim |' \geq k}^{(q)}\big)$ and the {\sf inner framed cobordisms} $\mathcal{FB}^{k-1}_\Xi(Y)$ of $Y$ (see the definition below). 

\begin{definition}\label{def.framed_cobordisms}
The cobordisms $\mathcal{FB}^{k-1}_\Xi(Y)$ are based on codimension $k-1$ smooth  
closed submanifolds  $Z$ of  $Y$ of the form $Z = \coprod_{\s=1}^{A(d, k, q)} Z_\s$. 
The normal bundle $\nu(Z, Y)$ is framed. The disjoint ``components" $\{Z_\s\}_\s$ of $Z$ are marked with different colors $\s$ from a {\sf pallet} $\Xi$ of cardinality $A(k, d, q)$. 

Two such submanifolds $Z_0, Z_1 \subset Y$ are said to be {\sf inner cobordant} if there exist codimension $k-1$ 
submanifold  $W = \coprod_{\s=1}^{A(d, k, q)} W_\s$ of $Y \times [0,1]$, such that \hfill \break $\d W =  (Z_0\times \{0\}) \cup (Z_1\times \{1\})$;  moreover, $\d W_\s = \big((Z_0)_\s \times \{0\}\big) \coprod \big((Z_1)_\s  \times \{1\}\big)$ for each color $\s \in \Xi$. 
The normal bundle $\nu := \nu(W, Y \times [0,1])$ is framed, the restrictions $\nu|_{Y\times \{0\}} \approx \nu(Z_0, Y)$, $\nu|_{Y\times \{1\}} \approx \nu(Z_1, Y)$, and the framing of $\nu$ extends the framings of $\nu(Z_0, Y)$ and $\nu(Z_1, Y)$. 
\hfill $\diamondsuit$
\end{definition}

As with quasitopies of immersions, for any pair $Y_1$, $Y_2$ of compact smooth manifolds with boundary and any choice of $\kappa_1 \in \pi_0(\d Y_1)$, $\kappa_2 \in \pi_0(\d Y_2)$, an operation 
\begin{eqnarray}\label{eq.uplus_for_framed}
\uplus_{\mathcal F}: \mathcal{FB}^{k-1}_\Xi(Y_1) \times \mathcal{FB}^{k-1}_\Xi(Y_2) \to \mathcal{FB}^{k-1}_\Xi(Y_1 \#_\d Y_2)
\end{eqnarray}
may be defined in the obvious way as $Z_1 \coprod Z_2 \subset Y_1 \#_\d Y_2$. This operation $\uplus_{\mathcal F}$ converts the inner framed cobordisms $\mathcal{FB}^{k-1}_\Xi(D^{n})$ into a 
  group. Choosing $\kappa \in \pi_0(\d Y)$, we get a group action 
\begin{eqnarray}\label{eq.uplus_action_framed}
 \uplus_{\mathcal F}: \mathcal{FB}^{k-1}_\Xi(D^n) \times \mathcal{FB}^{k-1}_\Xi(Y) \to \mathcal{FB}^{k-1}_\Xi(Y).
\end{eqnarray}

\begin{proposition}\label{prop.framed} Let a pallet $\Xi$ be of the cardinality $A(d, k, q)$. For a compact smooth manifold $Y$ of dimension $n \geq k-1$, where $k \in [3, d-1]$, there is a bijection 
$$Th: \mathcal{FB}^{k-1}_\Xi(Y) \approx \mathsf{QT}^{\mathsf{emb}}_{d, d}\big(Y; \mathbf{c}\Theta_{|\sim |' \geq k}^{(q)}; \mathbf{c}\Theta_{|\sim |' \geq k}^{(q)}\big),$$
delivered by the Thom construction on the trivialized normal bundle $\nu(Z, Y)$, where \hfill\break $[Z \hookrightarrow Y] \in \mathcal{FB}^{k-1}_\Xi(Y)$. \smallskip
The bijection $Th$ respects the operations $\uplus_{\mathcal F}$ and $\uplus$ in framed inner bordisms and in the quasitopies: 
\begin{eqnarray}\label{eq.equivariant_framed}
Th(Z_1 \uplus_{\mathcal F} Z_2) = Th(Z_1) \uplus Th(Z_2).
\end{eqnarray}
\end{proposition}

\begin{proof} Given a framed embedding $Z := \coprod_{\s=1}^{A(d, k, q)} Z_\s \hookrightarrow Y$ of a closed smooth $Z$ of codimension $k-1$, the Thom construction on $\nu(Z, Y)$ produces a continuous map $T_Z: Y \to \bigvee_{\s=1}^{A(d, k, q)} S^{k-1}_\s$ to the bouquet. Since $Z$ is closed,  $\d Y$ is mapped to the base point $\star$ of the bouquet. 
The homotopy class of $T_Z$ depends only on the inner cobordism class of $Z \subset Y$ in $\mathcal{FB}^{k-1}_\Xi(Y)$.  

Let us fix a homotopy equivalence $\tau: (\cP^{\mathbf{c}\Theta_{|\sim |' \geq k}^{(q)}}, pt) \sim (\bigvee_{\s=1}^{A(d, k, q)} S^{k-1}_\s, \star).$
By Corollary \ref{cor.E-regA}, the homotopy class of the map $\tilde T_Z := \tau^{-1} \circ T_Z$ determines the quasitopy class of some embedding $\b_{\tilde T_Z}: M \subset \R \times Y$, whose combinatorial tangency patterns to $\mathcal L$ belong to the poset $\mathbf{c}\Theta_{|\sim |' \geq k}^{(q)}$ and the tangency patterns of $\b_{\tilde T_Z}: \d M \to \R \times \d Y$ to the poset $\mathbf c\Lambda = (\emptyset), (1)$. Thus the map $Th : Z \leadsto \b_{\tilde T_Z}$ is well-defined. 
\smallskip

To show that $Th$ is onto, we take any embedding $\b: M \subset \R \times Y$ of a $n$-manifold $M$, whose tangency patterns reside in $\mathbf c\Theta_{|\sim |' \geq k}^{(q)}$ and such that the tangency patterns of $\b: \d M \subset \R \times \d Y$ reside in $\mathbf c\Lambda$ as above. We use $\b$ to produce a map $\Phi^\b: (Y, \d Y) \to (\cP^{\mathbf{c}\Theta_{|\sim |' \geq k}^{(q)}}, pt)$. Composing $\Phi^\b$ with the homotopy equivalence 
$\tau$, we get a map $\tilde\Phi^\b := \tau \circ \Phi^\b$ from $Y$ to the bouquet such that $\d Y$ is mapped to the base point $\star$ of the bouquet. For each sphere $S^{k-1}_\s$ from the bouquet, we fix a point $b_\s \in S^{k-1}_\s$, different from $\star$, and a $(k-1)$-frame at $b_\s$. Perturbing $\tilde\Phi^\b$ within its homotopy class, we may assume that $\tilde\Phi^\b$ is transversal to each $b_\s$. Then the closed framed submanifold $Z_\b := (\tilde\Phi^\b)^{-1}(\coprod_{\s \in \Xi} b_\s) \subset Y$ is an element of  $\mathcal{FB}^{k-1}_\Xi(Y)$. Moreover, for a given $\b$, the map $\tau^{-1} \circ Th(Z_\b)$ is in the homotopy class of $\Phi^\b$. 

The injectivity of $Th$ has a similar validation, using the $(n+1)$-dimensional cobordism $B: N \to \R \times Y \times [0,1]$ that bounds an embedding $\b: M \to \R \times Y$ which  is quasitopic to the trivial embedding. Again, we perturb $\tau \circ \Phi^B: Y \times [0,1] \to \bigvee_{\s=1}^{A(d, k, q)} S^{k-1}_\s$ to make it transversal to $\coprod_{\s \in \Xi} b_\s$, thus producing a framed cobordism $W$ in $Y \times [0,1]$ which bounds  $Z_\b$.
\smallskip

Finally, the validation of the property $Th(Z_1 \uplus_{\mathcal F} Z_2) = Th(Z_1) \uplus Th(Z_2)$ is on the level of definitions. 
\end{proof}

For $k > 2$ and an oriented $n$-dimensional $Y$, where $\d Y \neq \emptyset$, Proposition \ref{prop.framed} gives a limited insight into the actions of the (abelian when $n \geq 2$) group  $\mathsf{G}^{\mathsf{emb}}_{d,d}(n; \mathbf{c}\Theta_{|\sim |' \geq k}^{(q)}; \mathbf{c}\Theta_{|\sim |' \geq k}^{(q)})$ on the set $\mathsf{QT}^{\mathsf{emb}}_{d, d}\big(Y; \mathbf{c}\Theta_{|\sim |' \geq k}^{(q)}; \mathbf{c}\Theta_{|\sim |' \geq k}^{(q)}\big)$.\smallskip

Here is one useful observation: if $\b: M \hookrightarrow \R\times Y$ is represented, via $Th^{-1}$,  by a closed normally framed submanifold $Z_\b = \coprod_{\s \in \Xi} Z_{\b, \s} \subset Y$, then the fundamental homology classes $\{[Z_{\b,\s}] \in H_{n-k+1}(Y; \Z)\}_{\s \in \Xi}$ are \emph{constant} along the $\mathsf{G}^{\mathsf{emb}}_{d,d}(n; \mathbf{c}\Theta_{|\sim |' \geq k}^{(q)}; \mathbf{c}\Theta_{|\sim |' \geq k}^{(q)})$-orbit through $Th(Z_\b)$. Indeed, using the property $Th(Z_1 \uplus_{\mathcal F} Z_2) = Th(Z_1) \uplus Th(Z_2)$, all the changes of $Z_\b$, produced by the $\uplus_{\mathcal F}$-action, are incapsulated in a ball $B^n \subset Y$, and thus their fundamental classes are null-homologous in $Y$. 

For a pallet $\Xi$ of cardinality $A(d, k, q)$, consider the subset $H_{n-k+1}^{[\Xi \sim fr]}(Y; \Z)$ of the group $\big(H_{n-k+1}(Y; \Z)\big)^{A(d, k, q)}$ that is realized by $A(d, k, q)$ \emph{disjoint} distinctly colored and normally framed closed submanifolds $\{Z_\s\}_{\s \in \Xi}$ of $Y$,  $\dim Z_\s = n-k+1$. For $2k - 2 > n$, by a general position argument, $H_{n-k+1}^{[\Xi \sim fr]}(Y; \Z)$ is a \emph{subgroup} of $\big(H_{n-k+1}(Y; \Z)\big)^{A(d, k, q)}$.\smallskip

Using Proposition \ref{prop.framed} and tracing the correspondence $[\b] \leadsto Th^{-1}([\b]) \leadsto \{[Z_{\b,\s}]\}_\s$, validates instantly to the following claim.

\begin{proposition}\label{prop.framed_bordisms} Let $Y$ be a compact smooth oriented $n$-manifold $Y$, where $n \geq k-1$ and $k \in [3, d-1]$. Assume that $\d Y \neq \emptyset$. 
Then the orbit-space of the $\mathsf{G}^{\mathsf{emb}}_{d,d}(n; \mathbf{c}\Theta_{|\sim |' \geq k}^{(q)}; \mathbf{c}\Theta_{|\sim |' \geq k}^{(q)})$-action on the set $\mathsf{QT}^{\mathsf{emb}}_{d, d}\big(Y; \mathbf{c}\Theta_{|\sim |' \geq k}^{(q)}; \mathbf{c}\Theta_{|\sim |' \geq k}^{(q)}\big)$ admits a surjection on the set $H_{n-k+1}^{[\Xi \sim fr]}(Y; \Z)$. 

\hfill $\diamondsuit$
\end{proposition}

\begin{example}\label{ex.Sigma10} \emph{Take $k=3$,  $d = 6$, $q = 0$, and $n = 3$. Then $A(6, 3, 0) = 10$ and $\#\Xi= 10$. Thus we get a homotopy equivalence $\tau: \cP_6^{\mathbf{c}\Theta_{|\sim |' \geq 3}} \sim \bigvee_{\s=1}^{10} S^2_\s$. For a compact oriented $3$-manifold $Y$, any normally framed submanifold $Z$ acquires an orientation. Using the isomorphism 
$$Th: \mathcal{FB}^2_\Xi(Y) \approx \mathsf{QT}^{\mathsf{emb}}_{6, 6}\big(Y; \mathbf{c}\Theta_{|\sim |' \geq 3}; \mathbf{c}\Theta_{|\sim |' \geq 3}\big)$$ from Proposition \ref{prop.framed}, any collection of at most ten disjoint framed closed $1$-dimensional submanifolds $Z_1, \ldots Z_{10} \subset Y$ (that is, any \emph{framed link} $Z \subset Y$, colored with $10$ distinct colors at most) produces a quasitopy class of an embedding $\b: M^3 \to \R \times Y$, where $M$ is closed. Its tangency patterns $\om$, with the exception of $(\emptyset)$, have only entries from the list $\{1, 2, 3\}$ so that no more than one $3$ and no more than two $2$'s are present  in $\om$, while $|\om| \leq 6$.} 
\smallskip

\emph{The orbit-space of the $\mathsf{G}^{\mathsf{emb}}_{6,6}(3; \mathbf{c}\Theta_{|\sim |' \geq 3}; \mathbf{c}\Theta_{|\sim |' \geq 3})$-action on $\mathsf{QT}^{\mathsf{emb}}_{6, 6}\big(Y; \mathbf{c}\Theta_{|\sim |' \geq 3}; \mathbf{c}\Theta_{|\sim |' \geq 3}\big)$ admits a surjection onto the group $(H_1(Y; \Z))^{10}$. For example, for $Y = T^3 \setminus \textup{int}(D^3)$, where $T^3$ is the $3$-torus, the orbit-space is mapped onto the lattice $\Z^{30}$.
\hfill $\diamondsuit$}
\end{example}

\begin{corollary} We adopt the hypotheses of Proposition \ref{prop.kq-skeleton}. 
\begin{itemize}
\item For $\dim Y = k- 1$ and any choice of $\kappa \in \pi_0(\d Y)$,  the group $\mathsf{G}^{\mathsf{emb}}_{d,d}(k-1; \mathbf{c}\Theta_{|\sim |' \geq k}^{(q)};\hfill\break \mathbf{c}\Theta_{|\sim |' \geq k}^{(q)})$ acts freely and transitively on the set  $\mathsf{QT}^{\mathsf{emb}}_{d, d}\big(Y; \mathbf{c}\Theta_{|\sim |' \geq k}^{(q)}; \mathbf{c}\Theta_{|\sim |' \geq k}^{(q)}\big)$. Thus both sets are in a $1$-to-$1$ correspondence.  \smallskip
 
\item For a simply-connected $Y$, $\dim Y = k > 4$, the group $\mathsf{G}^{\mathsf{emb}}_{d,d}\big(k; \mathbf{c}\Theta_{|\sim |' \geq k}^{(q)}; \mathbf{c}\Theta_{|\sim |' \geq k}^{(q)}\big)$ acts freely and transitively on the set  $\mathsf{QT}^{\mathsf{emb}}_{d, d}\big(Y; \mathbf{c}\Theta_{|\sim |' \geq k}^{(q)}; \mathbf{c}\Theta_{|\sim |' \geq k}^{(q)}\big)$. Again, both sets are in a $1$-to-$1$ correspondence.  
\end{itemize}
\end{corollary}

\begin{proof} By Proposition \ref{prop.framed_bordisms}, there exists a bijection $$Th: \mathcal{FB}^{k-1}_\Xi(Y) \to \mathsf{QT}^{\mathsf{emb}}_{d, d}\big(Y; \mathbf{c}\Theta_{|\sim |' \geq k}^{(q)};\hfill\break \mathbf{c}\Theta_{|\sim |' \geq k}^{(q)}\big).$$ When $\dim Y = k-1$, the elements of $\mathcal{FB}^{k-1}_\Xi(Y)$ are represented by  finite sets $Z$ of framed oriented singletons, marked with $\#\Xi$ colors at most. (We may assume that no two singletons of the same color have opposite orientations.) Evidently, any such finite set $Z$ can be incapsulated into a standard smooth ball $D^{k-1} \subset Y$ such that $\d D^{k-1} \cap \d Y = D^{k-2}$, a standard $(k-2)$-ball. Therefore $Z$ belongs to  the $\mathcal{FB}^{k-1}_\Xi(D^{k-1})$-orbit of the empty collection $Z_\emptyset$. By (\ref{eq.equivariant_framed}), $Th(Z)$ belongs to the $\mathsf{G}^{\mathsf{emb}}_{d,d}(k-1; \mathbf{c}\Theta_{|\sim |' \geq k}^{(q)};\mathbf{c}\Theta_{|\sim |' \geq k}^{(q)})$-orbit of  the trivial embedding (the empty one for $d\equiv 0 \mod 2$ and the embedding $\b_\star: Y \to \{0\} \times Y \subset \R \times Y$ for $d \equiv 1 \mod 2$). This validates the first claim.\smallskip

For a simply-connected $Y$, $\dim Y = k > 4$, again by Proposition \ref{prop.framed_bordisms}, any element of $\mathsf{QT}^{\mathsf{emb}}_{d, d}\big(Y; \mathbf{c}\Theta_{|\sim |' \geq k}^{(q)}; \mathbf{c}\Theta_{|\sim |' \geq k}^{(q)}\big)$ is represented by a framed collection $Z$ of disjoint loops in $Y$, colored with $\#\Xi$ colors at most.  Using that $\pi_1(Y) =1$ and $\dim Y = k > 4$, we can bound each loop $Z_\s$ by a regularly embedded smooth $2$-disk $W_\s \subset \textup{int}(Y)$ so that all $W_\s$'s are disjoint. Therefore each loop $Z_\s$ is contained in a standard ball $D^k_\s$, the regular neighborhood of $W_\s$. By performing $1$-surgeries on $\coprod_\s D^k_\s$, we construct a $k$-ball $\tilde B^k$ that contains all $Z_\s$'s.  Adding a $1$-handle $H \subset Y \setminus \tilde B^k$ that connects $\d \tilde B^k$ with $\d Y$, we exhibit a ball $B^k := \tilde B^k \cup H$ which contains all the loops and such that $\d B^k \cap \d Y$ is a $(k-1)$-dimensional ball. Therefore $Z$ is, once more, in the orbit of the empty collection $Z_\emptyset$ of loops. By (\ref{eq.equivariant_framed}), $Th(Z)$ belongs to the $\mathsf{G}^{\mathsf{emb}}_{d,d}(k; \mathbf{c}\Theta_{|\sim |' \geq k}^{(q)};\mathbf{c}\Theta_{|\sim |' \geq k}^{(q)})$-orbit of the trivial embedding.

In fact, this conclusion is also valid for $3$-manifolds $Y$ with a \emph{spherical boundary}, but its validation is not as direct as for the case $k > 4$.  If $\pi_1(Y) =1$, 
 by the Perelman's solution of the Poincar\'{e} Conjecture \cite{P1}-\cite{P3}, we conclude that $Y$ is the  standard ball $D^3$. Thus any $\Xi$-colored and framed link $Z \subset Y$ is obviously contained in $D^3$. 
 
 As usual, the case $k = 4$ seems to be challenging. 
\end{proof}

\subsection{More about the stabilization of $\mathcal{QT}^{\mathsf{emb}}_{d, d'}(Y, \d Y; \mathbf{c}\Theta, \mathbf{c}\Lambda; \mathbf{c}\Theta')$ by the degrees $d, d'$}

We start with the embeddings of $1$-manifolds in the product $\R \times D^1$ and the forbidden poset $\Theta$ that consists of $\om$'s with the reduced norms $\geq 2$.

\begin{proposition} Let $d' \geq d+2$, $d' \equiv d \mod 2$. For a closed subposet 
$\Theta \subset \mathbf\Om_{\langle d],\, |\sim|' \geq 2}$, let $\hat\Theta_{\{d'\}}$ be the smallest subposet in $\mathbf\Om_{\langle d']}$ that contains $\Theta$. 
Then we get a stabilization by $d'$:
$$\mathsf{QT}^{\mathsf{emb}}_{d+2,\,d'}(D^1; \mathbf{c}\hat\Theta_{\{d+2\}}; \mathbf{c}\hat\Theta_{\{d'\}}) \approx \pi_1(\mathcal P_{d+2}^{\mathbf c \hat \Theta_{\{d+2\}}}, pt).$$
\end{proposition}

\begin{proof} The claim follows from Theorem \ref{th.E-reg}  by combining it with Theorem \ref{prop.fundamental_groups} (see also Theorem 2.15 from \cite{KSW1}), which claims that $(\e_{d+2, d'})_\ast: \pi_1(\cP_{d+2}^{\mathbf c\hat\Theta_{\{d+2\}}}) \approx \pi_1(\cP_{d'}^{\mathbf c\Theta_{\{d'\}}})$ is an isomorphism, induced by the embedding $\e_{d+2, d'}$ of the two spaces. See \cite{KSW1} for an explicit description of presentations of such fundamental groups $\pi_1(\mathcal P_{d+2}^{\mathbf c \hat\Theta_{\{d+2\}}})$. 
\end{proof}


\begin{theorem}\label{th.stable_embeddings} Let $\Theta$  be a closed poset such that $|\om|' \geq 3$ for any $\om \in \Theta$. Let $d' \geq d$ and $d' \equiv d \mod 2$.
Assume the $\Lambda$-condition (\ref{special}). 

For any compact smooth $Y$ such that $\dim Y  \leq d+2 - \psi_\Theta(d+2)$
 (see(\ref{eq.eta(Theta)})) we get  the classifying map 
 \begin{eqnarray}\label{stable, no 1, 2} 
 \Phi_{d, d+2}^{\mathsf{i/e}}(Y; \mathbf{c}\Theta; \mathbf{c}\Theta):\; \mathsf{QT}^{\mathsf{i/e}}_{d, d+2}(Y; \mathbf{c}\Theta; \mathbf{c}\Theta) \to [(Y, \d Y), (\cP_{d}^{\mathbf c\Theta}, pt)], 
\end{eqnarray}
which is bijective for embeddings and surjective for immersions. \smallskip 

For any profinite (in the sense of Definition \ref{def.profinite}) 
poset $\Theta$ and any  $Y$, the set $\mathsf{QT}^{\mathsf{emb}}_{d, d'}(Y; \mathbf{c}\Theta; \hfill\break \mathbf{c}\Theta)$  stabilizes  for all sufficiently big $d$.
\end{theorem}

\begin{proof}The argument is a bit similar to the one we used in the proof of Proposition \ref{prop.skeleta_and_groups}.

Under the hypotheses, by Theorem \ref{th.main_stab}, the homomorphism $(\e_{d, d'})_\ast: H_j(\mathcal P_d^{\mathbf c\Theta}; \Z) \to H_j(\mathcal P_{d+2}^{\mathbf c\Theta}; \Z)$ is an isomorphism for all $j  < d+2 - \psi_\Theta(d+2)$. Since $|\om|' \geq 3$ for any $\om \in \Theta$, all the spaces $\{\cP_{d'}^{\mathbf c\Theta}\}_{d' \geq d}$ are simply connected. Hence,  $\e_{d, d+2}$ induces an isomorphism in the integral $j$-homology  
of $\mathcal P_{d}^{\mathbf c\Theta}$ and of $\mathcal P_{d+2}^{\mathbf c\Theta}$ for all $j  < d+2 - \psi_\Theta(d+2)$. By the Whitehead Theorem (the inverse Hurewicz Theorem) (see Theorem (7.13) in \cite{Wh}), $\e_{d, d+2}$ induces isomorphisms of the homotopy groups of $\mathcal P_{d}^{\mathbf c\Theta}$ and of $\mathcal P_{d+2}^{\mathbf c\Theta}$ in dimensions $\leq d +2 -\psi_\Theta(d+2)$. 
By standard application of the obstruction theory, if $\dim Y \leq d+2 - \psi_\Theta(d+2)$, no nontrivial element of $[(Y, \d Y), (\cP_{d}^{\mathbf c\Theta}, pt)]$ becomes, via $\e_{d, d+2}$, null-homotopic in $[(Y, \d Y), (\cP_{d+2}^{\mathbf c\Theta}, pt')]$. Thus, for such $Y$, we get $$[[(Y, \d Y),\; \e_{d, d+2}: (\cP_d^{\mathbf c\Theta}, pt) \to (\cP_{d+2}^{\mathbf c\Theta'}, pt')]] = [(Y, \d Y), (\cP_{d}^{\mathbf c\Theta}, pt)].$$ 

By Lemma 4.3 from \cite{KSW2}, for any closed profinite $\Theta$, $\lim_{d' \to \infty} d' - \psi_\Theta(d') = +\infty$. Hence, repeating the previous arguments, we get a similar conclusion for all $d' \geq d \gg \dim Y$. 

Therefore, by Theorem \ref{th.E-reg}, the all the claims of Theorem \ref{th.stable_embeddings}  follow. 
\end{proof}

\section{Appendix: Topology of $\mathcal P_d^{\mathbf c \langle \om \rangle}$ for small $d$}

In conclusion, let us state somewhat surprising results of one computer-assisted computation from \cite{KSW2}, Figures 6-8. The tables below are produced from these three figures with the help of the Alexander duality. 
For a given $\om \in \mathbf\Om$, we denote by $\langle \om \rangle$ the minimal closed poset that contains $\om$.
In (\ref{big_list}), for $d \le 13$, we list all compositions $\om$ for which the space $\cP_d^{\mathbf c \langle \om \rangle}$ is homologically nontrivial. In fact, for $d \le 13$, every homologically nontrivial $\cP_d^{\mathbf c \langle \om \rangle}$ is a \emph{homology} sphere! Moreover, at least for $\om$ with $|\om|' > 2$, all such spaces $\cP_d^{\mathbf c \langle \om \rangle}$ are \emph{homotopy spheres}. Unfortunately, the reason for such phenomena is mysterious to us... \smallskip

In the list  below, we denote by $``\mathbf{S^k}"$ the homology $k$-sphers.  We use the notation $``\mathbf{S^k_\star}"$ for the homotopy $k$-spheres. 
To make the writing of $\om$'s more compact, we adopt the following notation: for example, we abbreviate  $\om = (1132555)$ as $(1^2 3 2 5^3)$. 

\begin{eqnarray}\label{big_list}
\text{\bf{List of nontrivial $\Z$-homology types of }} \cP_d^{\mathbf c\langle \om \rangle} \text{\bf{ for }} 4 \leq d \leq 13: 
\end{eqnarray}

\small{ 
\noindent $\underline{\mathsf {d=4}}$. 

\noindent  $|\om|'=0$:\;  $\om=(1^2) \;\Rightarrow  \mathbf{S^0}$

\noindent $|\om|'=1$:\;   $\om=(2) \;\Rightarrow \mathbf{S^0}$ 

\noindent  $|\om|'=3$:\;   $\om=(4) \;\Rightarrow \mathbf{S^2_\star}$

\medskip

\noindent $\underline{\mathsf {d=5}}$.   

\noindent $|\om|'=0$:\;  $\om=(1^3)\;\Rightarrow \mathbf{S^0}$

\noindent $|\om|'=4$:\;  $\om=(5)\;\Rightarrow  \mathbf{S^3_\star}$

\medskip 

\noindent $\underline{\mathsf {d=6}}$. 

\noindent $|\om|'=0$:\; $\om=(1^4), (1^2) \;\Rightarrow \mathbf{S^0}$

\noindent $|\om|'=1$:\;  $\om=(2)\;\Rightarrow  \mathbf{S^0}$

\noindent $|\om|'=5$:\;  $\om=(6)\;\Rightarrow  \mathbf{S^4_\star}$

\medskip 

\noindent $\underline{\mathsf {d=7}}$.   

\noindent $|\om|'=0$:\;  $\om=(1^5), (1^3) \;\Rightarrow  \mathbf{S^0}$

\noindent  $|\om|'=6$:\;  $\om=(7)\;\Rightarrow \mathbf{S^5_\star}$

\medskip 

\noindent $\underline{\mathsf {d=8}}$.   

\noindent $|\om|'=0$:\;  $\om=(1^6), (1^4), (1^2)\;\Rightarrow \mathbf{S^0}$
\quad 

\noindent $|\om|'=1$:\;   $\om=(1,2,1)\;\Rightarrow \mathbf{S^4}$ 
\quad  $\om=(2)\;\Rightarrow \mathbf {S^0}$
\quad 

\noindent $|\om|'=2$:\;   $\om=(1,3), (3,1)\;\Rightarrow  \mathbf{S^4}$

\noindent $|\om|'=3$:\;   $\om=(4)\;\Rightarrow  \mathbf{S^4_\star}$

\noindent $|\om|'=7$:\;  $\om=(8)\;\Rightarrow \mathbf{S^6_\star}$. 

\medskip 
\noindent $\underline{\mathsf {d=9}}$.  

\noindent $|\om|'=0$:\;  $\om=(1^7), (1^5), (1^3)\;\Rightarrow \mathbf{S^0}$; 

\noindent $|\om|'=1$:\;  $\om=(1^2,2,1), (1,2,1^2), \;\Rightarrow  \mathbf{S^4}$;

\noindent  $|\om|'=2$:\;  $\om=(1^2,3), (1,3,1), (3,1^2) \;\Rightarrow \mathbf{S^4}$;

\noindent $|\om|'=8$:\;  $\om=(9)\;\Rightarrow \mathbf{S^7_\star}$.

\medskip
\noindent $\underline{\mathsf {d=10}}$.  

\noindent $|\om|'=0$:\;   $\om=(1^8), (1^6), (1^4), (1^2) \;\Rightarrow  \mathbf{S^0}$;

\noindent $|\om|'=1$:\;  $\om=(1^3,2,1), (1^2,2,1^2), (1,2,1^3), \;\Rightarrow \mathbf{S^4}$;
\quad  $\om=(2)\;\Rightarrow \mathbf{S^0}$;

\noindent $|\om|'=2$:\;  $\om=(1^3,3), (1^2,3,1), (1,3,1^2), (3,1^3)\;\Rightarrow \mathbf{S^4}$;

\noindent $|\om|'=9$:\; $\om=(10)\;\Rightarrow  \mathbf{S^8_\star}$.

\medskip 
\noindent $\underline{\mathsf {d=11}}$.   

\noindent $|\om|'=0$:\;   $\om=(1^9), (1^7), (1^5), (1^3)\;\Rightarrow  \mathbf{S^0}$;

\noindent $|\om|'=1$:\;  $\om=(1^4,2,1), (1^3,2,1^2), (1^2,2,1^3), (1,2,1^4)\;\Rightarrow  \mathbf{S^4}$;

\noindent $|\om|'=2$:\;   $\om=(1^4,3), (1^3,3,1), (1^2,3,1^2), (1,3,1^3), (3,1^4) \;\Rightarrow \mathbf{S^4}$;
  
  \quad \quad \quad  $\om=(1,2,1,2,1)\;\Rightarrow  \mathbf{S^6}$;
\smallskip

\noindent $|\om|'=3$:\;  $\om=(1,2,1,3), (3,1,2,1)\;\Rightarrow  \mathbf{S^6_\star}$;
\smallskip

\noindent $|\om|'=4$:\;  $\om=(3,1,3)\;\Rightarrow  \mathbf{S^6_\star}$;
\smallskip

\noindent $|\om|'=10$:\;  $\om=(11)\;\Rightarrow  \mathbf{S^9_\star}$.

\medskip 
\noindent $\underline{\mathsf {d=12}}$.   

\noindent $|\om|'=0$:\;  $\om=(1^{10}), (1^8), (1^6), (1^4), (1^2) \;\Rightarrow  \mathbf{S^0}$;

\noindent $|\om|'=1$:\;  $\om=(1^5,2,1), (1^4,2,1^2),(1^3,2,1^3), (1^2,2,1^4), (1,2,1^5) \;\Rightarrow \mathbf{S^4}$;

  \quad \quad \quad   $\om=(1,2,1)\;\Rightarrow  \mathbf{S^6}$;
 \quad 
 
  \quad \quad \quad   $\om=(2)\; \Rightarrow \mathbf{S^0}$;

\noindent $|\om|'=2$:\;  $\om=(1^5,3), (1^4,3,1), (1^3,3,1^2), (1^2,3,1^3), (1,3,1^4),(3,1^5) \;\Rightarrow  \mathbf{S^4}$;
\quad 

  \quad \quad \quad  $\om=(1^2,2,1,2,1), (1,2,1^2,2,1), (1,2,1,2,1^2), (1,3), (3,1)\;\Rightarrow  \mathbf{S^6}$;

  \quad \quad \quad  $\om=(1,2^2,1)\;\Rightarrow \mathbf{S^8}$; 
 \smallskip

\noindent $|\om|'=3$:\;  $\om=(1^2,2,1,3), (1,2,1^2,3), (1,2,1,3,1), (1,3,1,2,1), (3,1^2,2,1),$   

\quad \quad \quad   $(3,1,2,1^2), (4)\;\Rightarrow  \mathbf{S^6_\star}$;

 \quad \quad \quad  $\om=(1,2,3), (3,2,1), (1,4,1) \;\Rightarrow \mathbf{S^8_\star}$;
\smallskip 

\noindent $|\om|'=4$:\;  $\om=(1,3,1,3), (3,1,3,1)\;\Rightarrow \mathbf{S^6_\star}$;
\quad 

  \quad \quad \quad  $\om=(1,5), (5,1),(3^2)\;\Rightarrow  \mathbf{S^8_\star}$;
\smallskip

\noindent  $|\om|'=5$:\;  $\om=(6)\;\Rightarrow \mathbf{S^8_\star}$;
\smallskip
 
\noindent   $|\om|'=11$:\;  $\om=(12)\;\Rightarrow \mathbf{S^{10}_\star}$.

\medskip
 
\noindent $\underline{\mathsf {d=13}}$.  

\noindent $|\om|'=0$:\; $\om=(1^{11}), (1^9), (1^7), (1^5), (1^3)\;\Rightarrow \mathbf{S^0}$;
\smallskip

\noindent $|\om|'=1$:\;  $\om=(1^6,2,1), (1^5,2,1^2), (1^4,2,1^3), (1^3,2,1^4), (1^2,2,1^5), (1,2,1^6)\;\Rightarrow  \mathbf{S^4}$; 

  \quad \quad \quad  $\om=(1^2,2,1), (1,2,1^2) \;\Rightarrow  \mathbf{S^6}$;
\smallskip   

\noindent $|\om|'=2$:\;  $\om=(1^6,3), (1^5,3,1), (1^4,3,1^2), (1^2,1,3,1^3), (1^2,3,1^4), (1,3,1^5), (3,1^6) \;\Rightarrow  \mathbf{S^4}$; 

  \quad \quad \quad  $\om=(1^2,3), (1,3,1),(3,1^2), (1^3,2,1,2,1), (1^2,2,1^2,2,1), (1^2,2,1,2,1^2), (1,2,1^3,2,1), \hfill\break (1,2,1^2,2,1^2), (1,2,1,2,1^3) \;\Rightarrow \mathbf{S^6}$;

  \quad \quad \quad   $\om=(1^2,2^2,1), (1,2^2,1^2) \;\Rightarrow \mathbf{S^8}$; 
 \smallskip

\noindent $|\om|'=3$:\;  $\om=(1^3,2,1,3), (1^2,2,1^2,3), (1^2,2,1,3,1), (1^2,3,1,2,1), (1,2,1^2,3,1), \hfill\break (1,2,1,3,1^2),  (1,3,1^2,2,1), (1,3,1,2,1^2), (3,1^2,2,1^2), (3,1,2,1^3)  \;\Rightarrow \mathbf{S^6_\star}$;

  \quad \quad \quad  $\om=(1^2,2,3), (1,2,3,1), (1,3,2,1), (3,2,1^2), (1^2,4,1), (1, 4,1^2) \;\Rightarrow \mathbf{S^8_\star}$;

\noindent $|\om|'=4$:\;  $\om=(1^2,3,1,3), (1,3,1,3,1), (3,1,3,1^2), \;\Rightarrow \mathbf{S^6_\star}$;  

  \quad \quad \quad  $\om=(1,3^2), (3^2,1), (1^2,5), (1,5,1), (5,1^2)\;\Rightarrow \mathbf{S^8_\star}$;
\smallskip
 
\noindent $|\om|'=12$:\;  $\om=(13)\;\Rightarrow  \mathbf{S^{11}_\star}$.
\hfill $\diamondsuit$
}

\smallskip


{\it Acknowledgment:} The author is grateful to the Department of Mathematics of  Massachusetts  Institute of Technology for many years of hospitality.


\end{document}